\documentclass[11pt]{article}
\usepackage{amsmath}
\usepackage{amssymb}
\usepackage{ulem}
\usepackage{amscd}
\usepackage[matrix,arrow]{xy}
\usepackage{amsmath,amsthm,amsfonts}
\usepackage[usenames]{xcolor}
\usepackage{hyperref}
\usepackage{textcomp}

\hypersetup{
  pdftitle    = {Categorical measures for finite group actions},
  pdfauthor   = {D.\,Bergh, S.\,Gorchinskiy, M.\,Larsen, V.\,Lunts},
  colorlinks=false}

\theoremstyle{plain}
\newtheorem{theo}{Theorem}[section]
\newtheorem{lemma}[theo]{Lemma}
\newtheorem{prop}[theo]{Proposition}
\newtheorem{conj}[theo]{Conjecture}

\numberwithin{equation}{section}

\theoremstyle{remark}
\newtheorem{example}[theo]{Example}
\newtheorem{remark}[theo]{Remark}

\theoremstyle{definition}
\newtheorem{defi}[theo]{Definition}

\makeatletter
\newtheoremstyle{maintheoremstyle} % name
    {\topsep}                    % Space above
    {\topsep}                    % Space below
    {\itshape}                   % Body font
    {}                              % Indent amount
    {\scshape}                  % Theorem head font
    {.}                             % Punctuation after theorem head
    {.5em}                       % Space after theorem head
    {}  % Theorem head spec (can be left empty, meaning ‘normal’)
\makeatother

\theoremstyle{maintheoremstyle}
\newtheorem{maintheo}{Theorem}

\def\R{\mathbb{R}}
\def\C{\mathbb{C}}
\def\Z{\mathbb{Z}}

\def\coh{\operatorname{coh}}

\def\PP{{\mathbb P}}

\def\pre-tr{\operatorname{pre-tr}}

\def\Hom{\operatorname{Hom}}
\def\End{\operatorname{End}}

\def\Eq{\operatorname{Eq}}

\newcommand{\bbA}{{\mathbb A}}
\newcommand{\bbC}{{\mathbb C}}

\newcommand{\bbR}{{\mathbb R}}

\newcommand{\bbZ}{{\mathbb Z}}

\newcommand{\bbP}{{\mathbb P}}

\newcommand{\bbQ}{{\mathbb Q}}

\newcommand{\Gm}{{\mathbb G}_m}

\newcommand{\KM}{{\mathcal K \mathcal M}}
\newcommand{\cY}{{\mathcal Y}}
\newcommand{\cX}{{\mathcal X}}

\newcommand{\cQ}{{\mathcal Q}}
\newcommand{\cF}{{\mathcal F}}
\newcommand{\cG}{{\mathcal G}}
\newcommand{\cO}{{\mathcal O}}
\newcommand{\cP}{{\mathcal P}}

\newcommand{\cM}{{\mathcal M}}
\newcommand{\cN}{{\mathcal N}}
\newcommand{\cD}{{\mathcal D}}
\newcommand{\cV}{{\mathcal V}}
\newcommand{\cA}{{\mathcal A}}
\newcommand{\cB}{{\mathcal B}}

\newcommand{\cE}{{\mathcal E}}
\newcommand{\cW}{{\mathcal W}}

\newcommand{\cS}{{\mathcal S}}
\newcommand{\cT}{{\mathcal T}}

\newcommand{\cH}{{\mathcal H}}

\newcommand{\cZ}{{\mathcal Z}}

\newcommand{\iner}{\operatorname{iner}}
\newcommand{\ind}{\operatorname{ind}}
\newcommand{\mods}{\operatorname{mod}}
\newcommand{\rep}{\operatorname{rep}}

\newcommand{\DG}{\operatorname{DG}}

\newcommand{\Ker}{\operatorname{Ker}}

\newcommand{\Spec}{\operatorname{Spec}}

\newcommand{\red}{\operatorname{red}}

\newcommand{\cs}{\operatorname{cs}}
\newcommand{\Br}{\operatorname{Br}}
\newcommand{\Mat}{\operatorname{Mat}}
\newcommand{\Mods}{\operatorname{Mod}}

\newcommand{\coarse}{\ensuremath{\mathrm{cs}}}
\newcommand{\K}{\ensuremath{\mathrm{K}}}
\newcommand{\Var}{\ensuremath{\mathrm{Var}}}

\newcommand{\DM}{\ensuremath{\mathrm{DM}}}
\newcommand{\smp}{\mathrm{sp}}

\newcommand{\gm}{\mathrm{gm}}
\newcommand{\sat}{\mathrm{sat}}
\newcommand{\tp}{\mathrm{top}}

\newcommand{\eq}{\mathrm{eq}}
\newcommand{\ex}{\mathrm{ex}}

\newcommand{\GL}{\mathrm{GL}}
\newcommand{\PGL}{\mathrm{PGL}}
\newcommand{\PSL}{\mathrm{PSL}}
\newcommand{\SL}{\mathrm{SL}}
\newcommand{\Irr}{\mathrm{Irr}}
\newcommand{\uIrr}{\underline{\mathrm{Irr}}}

\title{Categorical measures for finite group actions}

\author{D.\,Bergh\footnote{University of Copenhagen, dbergh@math.ku.dk}, S.\,Gorchinskiy\footnote{Steklov Mathematical Institute of Russian Academy of Sciences, National Research University Higher School of Economics, Russian Federation, gorchins@mi.ras.ru}, M.\,Larsen\footnote{Indiana University, Bloomington, IN 47405, USA, mjlarsen@indiana.edu}, V.\,Lunts\footnote{Indiana University, Bloomington, IN 47405, USA, vlunts@indiana.edu}}

\begin{document}

\maketitle

\abstract{Given a variety with a finite group action, we compare its equivariant categorical measure, that is, the categorical measure of the corresponding quotient stack, and the categorical measure of the extended quotient. Using weak factorization for orbifolds, we show that for a wide range of cases, these two measures coincide. This implies, in particular, a conjecture of Galkin and Shinder on categorical and motivic zeta-functions of varieties. We provide examples showing that, in general, these two measures are not equal. We also give an example related to a conjecture of Polishchuk and Van den Bergh, showing that a certain condition in this conjecture is indeed necessary.}

\tableofcontents

\section{Introduction}

Given an algebraic variety with an action of a finite group, we compare its equivariant categorical measure and the categorical measure of the extended quotient. Both measures take values in the Grothendieck ring of saturated dg-categories and are closely related to semiorthogonal decompositions of bounded derived categories of (equivariant) coherent sheaves on smooth projective varieties.

\bigskip

Let us motivate the questions that we address in the paper. Given a smooth compact manifold $M$, one has the Chern character isomorphism between its rationalized topological \mbox{$K$-groups}~$K_{\bullet}^{\tp}(M)_{\bbQ}:=K^{\tp}_{\bullet}(M)\otimes_{\bbZ}\bbQ$ and (direct sums of) singular cohomology~$H^{\bullet}(M,\bbQ)$ with rational coefficients by Atiyah and Hirzebruch~\cite{AH0}. When trying to generalize this to an equivariant setting, that is, in the presence of an action of a finite group $G$ on $M$, one finds that equivariant topological \mbox{$K$-groups} $K^{G,\tp}_{\bullet}(M)_{\bbQ}$ are not isomorphic
any more to (direct sums of) equivariant cohomology $H^{\bullet}_G(M,\bbQ)$. Indeed, this can be seen already in the case when~$M$ is a point.
Note that since we are working with rational coefficients, we have canonical isomorphisms $H^{\bullet}_G(M,\bbQ) \simeq H^{\bullet}(M,\bbQ)^G\simeq H^\bullet(M/G, \bbQ)$.

Atiyah and Segal~\cite{AS} (see also a paper by Hirzebruch and H\"ofer~\cite{HH}) discovered that one should replace equivariant cohomology by cohomology of the {\it extended quotient}
$$
M/^{\ex}G:=\widetilde{M}/G\,,
$$
where~$\widetilde{M}$ consists of pairs $(g,m)\in G\times M$ such that $g(m)=m$, and $G$ acts on the first factor by conjugation. Explicitly, we have an equality
\begin{equation}\label{eq:decompexintro}
{M/^{\ex}G=\coprod\limits_{[g]\in C(G)}M^g/Z(g)}\,,
\end{equation}
where~$g$ runs over a system of representatives of the conjugacy classes~$C(G)$ in~$G$ and $Z(g)\subset G$ is the centralizer of~$g$. Note that, in general, the extended quotient $M/^{\ex}G$ can be singular.

It turns out that there is a canonical isomorphism between~$K_{\bullet}^{G,\tp}(M)_{\bbC}$ and (direct sums of) singular cohomology $H^{\bullet}(M/^{\ex}G,\bbC)\simeq H^{\bullet}(\widetilde{M},\bbC)^G$ with complex coefficients. The Chern character isomorphism for $\widetilde{M}$ implies that this assertion is equivalent to existence of canonical isomorphisms $K_i^{G,\tp}(M)_{\bbC}\simeq K^{\tp}_i(M/^{\ex}G)_{\bbC}$ for $i=0,1$, see~\cite[Theor.\,2]{AS}. Note that it is important here to work with complex coefficients, not just with rational ones.

\medskip

An algebraic version of these isomorphisms was investigated further by Vistoli~\cite{Vis0}. Namely, let $X$ be a smooth quasi-projective algebraic variety over a field $k$ and let a finite group $G$ act on~$X$. Assume that the order $n:=|G|$ of $G$ is prime to the characteristic of~$k$ and that~$k$ contains a primitive $n$-th root of unity $\zeta_n$. Then there are canonical isomorphisms
\begin{equation}\label{eq:mainintro}
K_i^{G}(X)_R\simeq K_i(X/^{\ex}G)_R\,,\qquad i\geqslant 0\,,
\end{equation}
between algebraic $K$-groups, where $R:=\Z[1/n,\zeta_n]$ and the extended quotient $X/^{\ex}G$ for the algebraic variety $X$ is defined similarly as for manifolds (actually, {\it op.~cit.} contains a stronger result, but we provide the above formulation to ease exposition). Note that $K_i^G(X)$ is canonically isomorphic to the algebraic $K$-group $K_i\big([X/G]\big)$ of the quotient stack $[X/G]$.

The isomorphisms~\eqref{eq:mainintro} (and their stronger versions) were generalized independently by Vezzosi, Vistoli~\cite{VV} and To\"en~\cite{To0},~\cite{To1} to the case when $G$ is a linear algebraic group and~$G$ acts on~$X$ as above with finite geometric stabilizers.

An analog of isomorphisms~\eqref{eq:mainintro} when $G$ is finite with algebraic $K$-groups being replaced by (generalized) cyclic homology was obtained by Baranovsky~\cite{Bar}. Furthermore, recently, Tabuada and Van den Bergh~\cite{TVdB} proved an analog of~\eqref{eq:mainintro} with the algebraic $K$-groups being replaced by an arbitrary additive invariant of dg-categories.
%(when $X$ is smooth projective, by a theorem of Tabuada~\cite[Th\'eor.\,6.1]{Tab}, this amounts to an isomorphism between non-commutative motives of $[X/G]$ and~$X/^{\ex}G$ %with coefficients in $R$).

\medskip

In what follows we assume that $k$ is of characteristic zero. It is natural to try to lift isomorphisms between various invariants of $[X/G]$ and $X/^{\ex}G$ to a more geometrical context. Clearly, the stack $[X/G]$ itself is not isomorphic to the variety~$X/^{\ex}G$ in general. Recall that one has derived categories of coherent sheaves living in between the geometrical side and the side of additive invariants. However, the bounded derived category~$D^G(X)$ of \mbox{$G$-equivariant} coherent sheaves on $X$ is not necessarily equivalent to the bounded derived category $D^b(X/^{\ex}G)$ of coherent sheaves on~$X/^{\ex}G$. For instance, this can be seen by considering an effective action of $G=\Z/n\Z$ on $X=\bbP^1$.

Nevertheless, one has the following conjecture by Polishchuk and Van den Bergh~\cite{PV}: there exists a semiorthogonal decomposition of $D^G(X)$ whose pieces form the collection of categories~$D^b\big(X^g/Z(g)\big)$, where~$g$ runs over a system of representatives of the classes in~$C(G)$ (cf. formula~\eqref{eq:decompexintro}) provided that all quotients $X^g/Z(g)$ are smooth and the action of $G$ on $X$ is effective (see Conjecture~\ref{conj:pvdb} for a precise formulation). The conjecture was proved for a number of cases in {\it op.~cit.} The conjecture would imply that under its conditions there exists an isomorphism as in formula~\eqref{eq:mainintro} without tensoring by~$R$. We show that the conjecture would be false if the assumption that the action be effective were removed
(see Theorem~\ref{thm:PolvdB} and Theorem~\ref{thmD} below).

\medskip

It is natural to test the conjecture of Polishchuk and Van den Bergh on the level of the Grothendieck ring of dg-categories. To explain this, first recall that the Grothendieck ring~$K_0(\Var_k)$ of varieties over $k$ is the abelian group generated by isomorphism classes~$\{X\}$ of varieties $X$ subject to the scissors relations ${\{X\} = \{X \smallsetminus Z\} + \{Z\}}$, where $Z$ is a closed subvariety of $X$. By
Bittner's result~\cite{Bit}, $K_0(\Var_k)$ is also isomorphic to the abelian group generated by isomorphism classes $\{X\}$ of smooth projective varieties $X$ subject to the blow-up relations ${\{X\} - \{Z\} = \{\widetilde{X}\} -\{E\}}$,
where $\widetilde{X}\to X$ is the blow-up of $X$ in a smooth center $Z\subset X$ and $E\subset\widetilde{X}$ is the exceptional divisor, and the relation $\{\varnothing\}=0$.

Bondal, Larsen, and Lunts~\cite{BLL} (see also papers by Tabuada~\cite{Tab} and To\"en~\cite{ToDG}) defined the Grothendieck ring $K_0(\DG_k^{\sat})$ of saturated dg-categories over $k$, which is generated by quasi-equivalence classes $\{\cM\}$ of saturated \mbox{dg-categories} subject to the relations ${\{\cM\}=\{\cM'\}+\{\cM''\}}$ coming from semiorthogonal decompositions ${H^0(\cM)=\langle \,H^0(\cM'),H^0(\cM'')\,\rangle}$ (see, e.g., Subsection~\ref{subsec:Grothdg}). In~\cite{BLL}, a ring homomorphism
$$
\mu\;:\;K_0(\Var_k)\longrightarrow K_0(\DG_k^{\sat})\,,\qquad \{X\}\longmapsto \mu(X)\,,
$$
called the {\it categorical measure}, is constructed. It sends the class $\{X\}$ of a smooth projective variety to the class~$\{\cD(X)\}$ of a dg-enhancement of $D^b(X)$. Note that such \mbox{dg-enhancement} is unique by a theorem of Lunts and Orlov~\cite{LO}. The proof of the fact that $\mu$ is well defined relies on Bittner's result and on a result of Orlov~\cite{O0} about derived categories of blow-ups.

Similarly, one defines the Grothendieck group $K_0(\Var^G_k)$ generated by classes $\{X\}^G$ of varieties $X$ with an action of a finite group $G$ (see, e.g., Subsection~\ref{subsect:GrothGvar}) and Bittner's theorem holds true in the equivariant setting as well. One constructs a homomorphism of groups (see Subsection~\ref{subsec:Grothdg})
$$
\mu^G\;:\;K_0(\Var^G_k)\longrightarrow K_0(\DG_k^{\sat})\,,\qquad \{X\}^G\longmapsto \mu^G(X)\,,
$$
called an {\it equivariant categorical measure}, that sends the class $\{X\}^G$ of a smooth projective variety $X$ with an action of $G$ to the class $\{\cD^G(X)\}$ of a dg-enhancement of $D^G(X)$. Such \mbox{dg-enhancement} is unique by a result of Canonaco and Stellari~\cite{CS}. To see that the equivariant categorical measure~$\mu^G$ is well defined, one combines Orlov's result with results by Elagin~\cite{Ela0},~\cite{Ela} about semiorthogonal decompositions of $D^G(X)$.

We see that the conjecture of Polishchuk and Van den Bergh would imply that under its conditions for $X$ smooth projective, there is an equality
\begin{equation}\label{eq:maineqintro}
\mu^G(X)=\mu(X/^{\ex}G)
\end{equation}
in~$K_0(\DG_k^{\sat})$.

Note that the equality~\eqref{eq:maineqintro} makes sense for an arbitrary variety $X$ with an action of a finite group $G$.
Independently, the conjectural equality~\eqref{eq:maineqintro} was studied by Galkin and Shinder in this setting. Their motivation was
%an another conjecture by
another conjecture due to
themselves.
Recall that Kapranov~\cite{Kap} defined the {\it motivic zeta-function} of an arbitrary variety $V$ as the following power series with coefficients in $K_0(\Var_k)$:
$$
Z_{mot}(V,t):=\sum _{n\geqslant 0}\{S^n(V)\}\,t^n\,,
$$
where $S^n(V)=V^{\times n}/\Sigma_n$ is the \mbox{$n$-th} symmetric power of a variety $V$ and $\Sigma_n$ is the symmetric group that permutes $n$ elements. A conjecture of Galkin and Shinder~\cite{GS} predicts a beautiful identity of power series with coefficients in $K_0(\DG^{\sat}_k)$
\begin{equation}\label{eq:GSintro}
\sum_{n\geqslant 0}\mu^{\Sigma_n}(V^{\times n})\,t^n=\mu\Big(\,\prod\limits_{i\geqslant 1}Z_{mot}(V,t^i)\Big)\,.
\end{equation}
Galkin and Shinder called the left hand side of~\eqref{eq:GSintro} a {\it categorical zeta-function}. It is explained in {\it op.~cit.} how to reduce equality~\eqref{eq:GSintro} to some previously known non-trivial results in the case when $V$ is either a curve or a surface. Also, Galkin and Shinder observed that equalities~\eqref{eq:maineqintro} with $G=\Sigma^n$ and $X=V^{\times n}$ for all $n\geqslant 0$ would imply equality~\eqref{eq:GSintro}. In this paper, we prove equality~\eqref{eq:GSintro} and disprove~\eqref{eq:maineqintro} for an arbitrary variety~$X$.

\medskip

As an example of equality~\eqref{eq:maineqintro}, consider the case when the action of $G$ on $X$ is free. Note that in this case $X/^{\ex}G=X/G$ and there is an equivalence of categories ${D^G(X)\simeq D^b(X/G)}$. If, in addition, $X$ is smooth projective, then the equality is obvious, because in this case~$X/G$ is a smooth projective variety, whence both sides of the equality just equal~$\mu(X/G)$. On the contrast, if~$X$ is not smooth projective, then the equality is highly non-trivial.

Indeed, when $X$ is, say, smooth affine, one computes the equivariant categorical measure~$\mu^G(X)$ by taking a $G$-equivariant smooth projective compactification $\overline{X}$. Then $\mu^G(X)$ equals the difference of~$\{\cD^G(\overline X)\}$ and a term defined by the boundary $\overline{X}\smallsetminus X$ with the action of $G$. On the other hand, to compute the categorical measure ${\mu(X/^{\ex}G)=\mu(X/G)}$, one takes a resolution of singularities $Y\to \overline{X}/G$ such that it is an isomorphism over~$X/G$. Then~$\mu(X/^{\ex}G)$ equals the difference of~$\{\cD(Y)\}$ and a term defined by the boundary~${Y\smallsetminus (X/G)}$. This shows that equality~\eqref{eq:maineqintro} in this case is not obvious at all and is closely related to the categorical McKay correspondence, see the paper by Bridgeland, King, Ried~\cite{BKR}, a survey by Ried~\cite{Ried}, and papers by Kawamata~\cite{Kaw1},~\cite{Kaw2}.

\bigskip

Now let us describe our main results. First, we show that equality~\eqref{eq:maineqintro} holds true for some important special cases (see Theorem~\ref{theorem:main}). This gives an analog of the isomorphisms by Atiyah--Segal, Vistoli, and others (see formula~\eqref{eq:mainintro}), in the context of the Grothendieck ring of dg-categories.

\begin{maintheo}\label{thmA}
Let $X$ be a variety over a field $k$ of characteristic zero and let a finite group~$G$ act on~$X$. Suppose that for any $\bar k$-point $x$ of $X$, where $\bar k$ is the algebraic closure of $k$, the following conditions are satisfied:
\begin{itemize}
\item[(i)]
all irreducible representations over $\bar k$ of the stabilizer $N_x\subset G$ of $x$ are defined over~$k$;
\item[(ii)]
the direct sum of all irreducible representations of $N_x$ over $k$ (taken with multiplicity one) extends to a $G_x$-representation, where $G_x\subset G$ is the normalizer of~$N_x$;
\item[(iii)]
there is an $H_x$-equivariant bijection between the set $C(N_x)$ of conjugacy classes in~$N_x$ and the set of isomorphism classes of irreducible representations of $N_x$, where we put ${H_x:=G_x/N_x}$.
\end{itemize}
Then the equality $\mu^G(X)=\mu(X/^{\ex}G)$ holds in $K_0(\DG_k^{\sat})$.
\end{maintheo}

For example, if the action of $G$ on $X$ is free, then the conditions of Theorem~\ref{thmA} are satisfied. Also, they are satisfied when $G$ is an abelian group such that all characters of $G$ over $\bar k$ are defined over $k$ (see Example~\ref{examp:S0}(ii)).

\medskip

It turns out that the conditions of Theorem~\ref{thmA} are satisfied for the action of symmetric groups on powers of varieties. This allows us to prove the conjecture of Galkin and Shinder (see Theorem~\ref{thm:GS} and Proposition~\ref{prop:GS}).

\begin{maintheo}\label{thmB}
For any variety $V$, there is an equality ${\mu^{\Sigma_n}(V^{\times n})=\mu(V^{\times n}/^{\ex}\Sigma_n)}$ for all~${n\geqslant 0}$ and the equality~\eqref{eq:GSintro} holds true in $K_0(\DG^{\sat}_k)$.
\end{maintheo}

\medskip

Our next result shows that, in general, equality~\eqref{eq:maineqintro} does not hold (see Theorem~\ref{theor:mainexamp}), at least not if we replace the group $K_0(\DG_k^{\sat})$ by the Grothendieck group $K_0(\DG_k^{\gm})$ of {\it geometric} \mbox{dg-categories}.
Geometric dg-categories were introduced by Orlov~\cite{Orl}.
Roughly speaking, a dg-category is geometric if it is a semiorthogonal component of a dg-category of the form
$\mathcal{D}(X)$ where $X$ is a smooth projective variety (see Subsection~\ref{subsec:Grothdg} for details).
The ring $K_0(\DG_k^{\gm})$ is defined similarly as $K_0(\DG_k^{\sat})$, and there is a naturally defined (ring) homomorphism ${K_0(\DG_k^{\gm})\to K_0(\DG_k^{\sat})}$ through
which the equivariant categorical measure factors by a theorem of Bergh, Lunts, and Schn\"urer~\cite{BLS}. In particular, it makes sense to study the relation~\eqref{eq:maineqintro} in $K_0(\DG_k^{\gm})$.

We consider the case $k=\bbC$ and we consider cohomology with respect to the classical complex topology. Also, we use the following homomorphism of groups defined by the order of torsion in the topological $K_1$-groups of complex smooth projective varieties (see Subsection~\ref{subsec:DGKM}):
$$
\epsilon\;:\;K_0(\DG_{\bbC}^{\gm})\longrightarrow\bbR\,,\qquad \{\cD(X)\}\longmapsto \log |K_1^{\tp}(X)_{\rm tors}|\,,
$$
where $X$ is a complex smooth projective variety. We do not know whether $\epsilon$ factors through~$K_0(\DG_k^{\sat})$.

\begin{maintheo}\label{thmC}
Let
$$
1 \longrightarrow \Z/n\Z \longrightarrow G \longrightarrow H \longrightarrow 1
$$
be a central extension of a finite group $H$ by $\Z/n\Z$ and denote its class in~$H^2(H,\Z/n\Z)$ by~$z$.
Furthermore, let $X\to B$ be an unramified Galois $H$-cover of complex smooth projective varieties.
Suppose that the following conditions are satisfied:
\begin{itemize}
\item[(i)]
we have $\beta\big(\lambda(z)\big)\ne 0$ in $H^3(B,\Z)_{\rm tors}$, where ${\lambda\colon H^2(H,\Z/n\Z)\to H^2(B,\Z/n\Z)}$ is the natural map and $\beta$ is the Bockstein homomorphism;
\item[(ii)]
there is an equality ${|K_1^{\tp}(B)_{\rm tors}|=|H^{\rm odd}(B,\Z)_{\rm tors}|}$.
\end{itemize}
Then the inequality ${\epsilon\big(\mu^G(X)\big)<\epsilon\big(\mu(X/^{\ex}G)\big)}$ holds,
where $G$ acts on $X$ through $H$.
In particular, ${\mu^G(X)\ne \mu(X/^{\ex}G)}$ in $K_0(\DG^{\gm}_{\C})$.
\end{maintheo}

Note that in this case, $X/^{\ex}G$ is just the disjoint union of $n$ copies of $B$. Also, it follows from a result of Bernardara~\cite{Bern} that $\mu^G(X)$ can be expressed in terms of~$\mu(E)$, where $E\to B$ is a certain Severi--Brauer scheme (see Proposition~\ref{prop:Bern}). Our proof of Theorem~\ref{thmC} is inspired by a construction of Ekedahl~\cite{Eked}, which deals with torsion in cohomology of Severi--Brauer schemes.

There are many examples of central extensions and covers $X\to B$ satisfying the conditions of Theorem~\ref{thmC} (see Proposition~\ref{prop:examp} and Remark~\ref{rmk:mainexamp}).
For instance, one can take $G$ to be the Heisenberg group over $\Z/n\Z$, that is, the group of upper triangular $(3\times 3)$-matrices over $\Z/n\Z$ with units on the diagonal, and $H$ to be the quotient of $G$ by its center.
With this choice of extension, one can take $B$ to be the product $C\times S$, where $C$ is a complex smooth projective curve of positive genus and $S$ is a complex smooth projective surface with non-zero $n$-torsion in $H^2(S,\Z)$ (say, $S$ can be an Enriques surface with $n=2$ or a suitable surface of general type).

\medskip

It follows from Theorem~\ref{thmC} that the conjecture of Polishchuk and Van den Bergh would be false if the condition that the action of the group be effective were removed. We also provide a simpler example of this (see Theorem~\ref{thm:PolvdB}). Consider the action of $(\Z/n\Z)^{\oplus 2}$ on~$(\Z/n\Z)^{\oplus 3}$ by matrices of type
$$
\left(\begin{array}{ccc}
1 & a & b\\
0 & 1 & 0\\
0 & 0 & 1
\end{array}\right)\,,
$$
where $a,b\in\Z/n\Z$, and let $G$ be the semidirect product ${(\Z/n\Z)^{\oplus 3}\rtimes (\Z/n\Z)^{\oplus 2}}$.

\begin{maintheo}\label{thmD}
Suppose that $k$ is algebraically closed and let $E$ be an elliptic curve over~$k$ without complex multiplication. Let $G$ act on $E$ by translation through the quotient~$(\Z/n\Z)^{\oplus 2}$, which we identify with the $n$-torsion subgroup of $E$. Then there is no semiorthogonal decomposition of $D^G(E)$ whose components (up to order) form the collection of triangulated categories $D^b\big(E^g/Z(g)\big)$, ${[g]\in C(G)}$.
\end{maintheo}

\bigskip

The proof of our main result, that is, of Theorem~\ref{thmA}, is based on two principal ingredients,
each interesting in its own right.

\medskip

To explain the first one, consider the following motivating question: let a finite group~$G$ act effectively on a smooth projective variety $X$ and let $Y$ be a smooth projective variety such that there is an isomorphism of fields $k(X)^G\simeq k(Y)$ over $k$. Then how is $X$ with its $G$-action related to  $Y$ geometrically?
One can think of this question as on a factorization of a birational map between the smooth projective quotient stack $[X/G]$ and the smooth projective variety $Y$, which have isomorphic open subvarieties.

A naive idea is to expect a sequence of $G$-equivariant blow-ups $\widetilde{X}\to X$ such that the quotient~$\widetilde{X}/G$ becomes smooth and then to apply the weak factorization theorem by Abramovich, Karu, Matsuki, W{\l}odarczyk in~\cite{AKMW} and by W{\l}odarczyk in~\cite{Wlo} to the smooth projective varieties~$\widetilde{X}/G$ and~$Y$. But, in general, such a sequence of blow-ups does not exist, which can be seen, for example, when one has (locally) a linear action of $\Z/4\Z$ on $\bbA^2$ with weights $(1,2)$.

To solve this problem, one needs to consider more birational transformations besides the classical blow-ups, namely,
those given by the root construction introduced by Cadman~\cite{Cad} and by Abramovich, Graber, Vistoli~\cite{AGV}. Note that the result of a non-trivial root construction applied to a variety is a genuine Deligne--Mumford stack. A typical example would be the morphism $[X'/(\Z/n\Z)]\to X$, where $X'\to X$ is an $n$-cyclic cover ramified over a smooth divisor $D\subset X$. This morphism is an isomorphism over $X\smallsetminus D$ and, by definition,~$D$ is the center of this root construction.

In general, one has the following destackification theorem of
Bergh~\cite{Ber}: given a smooth orbifold $\cX$, there is a sequence of blow-ups and root constructions with smooth centers $\widetilde{\cX}\to \cX$ such that the non-stacky locus in $\cX$ remains unmodified, the morphism $\widetilde{\cX}\to \widetilde{\cX}_{\cs}$ to the coarse space is a sequence of root constructions with smooth centers, and the coarse space $\widetilde{\cX}_{\cs}$ is a smooth algebraic space. In particular, applying this to $\cX=[X/G]$ as above and combining this with the weak factorization theorem for algebraic spaces, one finds the answer to the motivational question.

\medskip

Thus the destackification theorem together with the weak factorization theorem provide a way to construct weak factorization for orbifolds. A more general weak factorization theorem for Deligne--Mumford stacks has recently been obtained by Harper~\cite{Harper} and independently by Rydh (unpublished), but we will not need the general result in this work. This leads to the definition of the Grothendieck ring of smooth projective Deligne--Mumford stacks $K_0(\DM_k^{\smp})$ similarly as in Bittner's result:~$K_0(\DM_k^{\smp})$ is generated by isomorphism classes $\{\cX\}$ of smooth projective stacks $\cX$ subject to the stacky blow-up relations given by classical blow-ups and by root constructions (see more detail in Subsection~\ref{subsect:GrothDMstacks}). Recall that, according to Kresch~\cite{Kre}, smooth projective Deligne--Mumford stacks are the same as quotient stacks $[P/\GL_n]$, where $\GL_n$ acts properly and linearly on a smooth quasi-projective variety $P$ such that the geometric quotient $P/\GL_n$ (which automatically exists) is projective. Bittner's theorem implies that one has natural homomorphisms of groups
$$
\alpha\;:\;K_0(\Var_k)\longrightarrow K_0(\DM_k^{\smp})\,,\qquad \alpha^G\;:\;K_0(\Var^G_k)\longrightarrow K_0(\DM_k^{\smp})\,,
$$
where the second one sends $\{X\}^G$ to $\{[X/G]\}$ for a smooth projective variety $X$ with an action of a finite group $G$. Note that $\alpha$ is actually a ring homomorphism, whereas $\alpha^G$ does not respect the ring structure whenever the group $G$ is non-trivial. We prove the following result (see Theorem~\ref{theo:comparison}):

\begin{maintheo}\label{thmE}
Let $X$ be an arbitrary variety with a free action of a group $G$. Then the equality
$$
\alpha^G(X)=\alpha(X/G)
$$
holds in $K_0(\DM_k^{\smp})$.
\end{maintheo}

It is important that the categorical measure $\mu^G\colon K_0(\Var_k^G)\to \K_0(\DG_{k}^{\sat})$ factors through the homomorphism~$\alpha^G$ (see Subsection~\ref{subsec:catmeas}). Thus Theorem~\ref{thmE} implies Theorem~\ref{thmA} for the case when the action is free.

Note that it is crucial for us to work with the Grothendieck ring of (smooth projective) Deligne--Mumford stacks and not with the Grothendieck ring $K_0({\rm Stck}_k)$ of all Artin stacks described by Ekedahl~\cite{Eked} (see also references therein). The categorical measure does not factor through $K_0({\rm Stck}_k)$ (see Remark~\ref{rmk:catmeas}(ii)).

\medskip

The second ingredient in the proof of Theorem~\ref{thmA} consists in the description of \mbox{$G$-equivariant} sheaves on a variety $X$ with an action of a finite group $G$ that factors through a quotient $G\to H$ with a non-trivial kernel $N\subset G$. We use a finite affine scheme $\uIrr(N):=\Spec(Z)$, where $Z$ is the center of the group algebra $k[N]$ (see Subsection~\ref{subsection:Brinterp}). Note that $\bar k$-points of~$\uIrr(N)$ are in natural bijection with the set of irreducible representations of $N$ over $\bar k$. The quotient $H$ acts on $Z$ by conjugation, which defines an action of $H$ on $\uIrr(N)$. The exact sequence of groups
$$
1\longrightarrow N\longrightarrow G\longrightarrow H\longrightarrow 1
$$
defines in a certain canonical way a class $\theta\in \Br^H(Z)$ in the $H$-equivariant Brauer group. Let $f\colon \uIrr(N)\times X\to \uIrr(N)$ denote the projection. This gives an element ${f^*\theta\in\Br^H\big(\uIrr(N)\times X\big)}$. For the following statement see Theorem~\ref{thm:equivcoh} and Theorem~\ref{cor:trivobstr}.

\begin{maintheo}\label{thmF}
There is an equivalence of categories
$$
\coh^G(X)\simeq \coh^H(f^*\theta)\,,
$$
where $\coh^G(X)$ denotes the category of $G$-equivariant coherent sheaves on $X$ and~$\coh^H(f^*\theta)$ denotes the category of $H$-equivariant coherent modules over an \mbox{$H$-equivariant} Azuamaya algebra on $\uIrr(N)\times X$ which is a representative of $f^*\theta$.
\end{maintheo}

In particular, when $X$ is a point, Theorem~\ref{thmF} gives an equivalence between the category of finite-dimensional $G$-representations over $k$ and the category of $H$-equivariant modules over an $H$-equivariant Azumaya algebra over $Z$ which is a representative of $\theta$.

We give an explicit criterion for vanishing of $\theta$ in terms of representations of $N$ and $G$ (see Proposition~\ref{cor:split}). Together with Theorem~\ref{thmF} this allows us to prove Theorem~\ref{thmA} in full generality.

\medskip

Finally, note that in order to prove Theorem~\ref{thmA}, it is convenient to consider all the Grothendieck groups $K_0(\Var^G_k)$, for varying finite groups $G$, at once. Also, it makes sense to impose the induction relation of type $\{Y\}^H=\{\ind_H^G(Y)\}^G$ thus obtaining a ring $K_0(\Var^{\eq}_k)$ (see Definition~\ref{def:Groeq}), which we call a Grothendieck ring of equivariant varieties (this ring is also defined by Gusein-Zade, Luengo, and Melle-Hern\'andez in~\cite[Def.\,6]{GZ} and related groups are considered by Getzler, Pandharipande in~\cite{GP} and by Bergh in~\cite{Ber0}).

\bigskip

The paper is organized as follows. In Section~\ref{sec:not}, we fix notation and some convention. In particular,
we recall the definition of (quasi-)projective Deligne--Mumford stacks and the root construction.

\medskip

Section~\ref{sec:geom} is of geometrical flavour and is devoted to various Grothendieck rings of geometric objects and their properties. In Subsection~\ref{subsect:GrothGvar}, we recall from Bittner's paper~\cite{Bit} the definition and general properties of the Grothendieck group~$K_0(\Var_k^G)$ of varieties with an action of a fixed finite group~$G$.

We collect these groups all together in Subsection~\ref{subsec:Grotheq} to get the ring~$K_0(\Var^{\eq}_k)$ of varieties with an action of some finite group. We provide a useful collection of additive generators of~$K_0(\Var^{\eq}_k)$ (see Lemma~\ref{lemma:strat}).

In Subsection~\ref{subsec:extquot}, we recall the notion of an extended quotient $X/^{\ex}G$, which is crucial for us. We show that taking extended quotient induces a well-defined ring homomorphism ${K_0(\Var_k^{\eq})\to K_0(\Var_k)}$ (see Lemma~\ref{lem:iner} and the discussion following it). A technical result gives a sufficient condition on a variety $X$ with an action of a group $G$ and on a group homomorphism $\rho\colon K_0(\Var_k^{\eq})\to A$ that guarantees the equality $\rho^G(X)=\rho(X/^{\ex}G)$ (see Proposition~\ref{lemma:techno}).

We introduce the ring $K_0(\DM_k^{\smp})$ of smooth projective Deligne--Mumford stacks in Subsection~\ref{subsect:GrothDMstacks}. We prove Theorem~\ref{thmE} in Subsection~\ref{subsec:proofcomp}, using the destackification theorem of Bergh~\cite{Ber}.

\medskip

In Section~\ref{sec:cat}, which is of dg-categorical flavour, we describe the categorical measure for the Grothendieck ring of Deligne--Mumford stacks. We start by recalling in Subsection~\ref{subsec:Grothdg} generalities on dg-categories and also the definition of the Grothendieck ring of saturated dg-categories and related rings including the Grothendieck ring $K_0(\DG_k^{\gm})$ of geometric \mbox{dg-categories}.

In Subsection~\ref{subsec:catmeas}, using a geometricity theorem of Bergh, Lunts, and Schn\"urer~\cite{BLS}, we construct categorical measures, which are homomorphisms ${K_0(\DM^{\smp}_k)\to K_0(\DG^{\gm}_k)}$ and ${\mu\colon K_0(\Var^{\eq}_k)\to K_0(\DG_k^{\gm})}$.

We discuss in Subsection~\ref{subsec:DGKM} the construction of $K$-motives for geometric dg-categories (see Proposition~\ref{prop:Kgeom}), which was introduced previously by Gorchinskiy and Orlov~\cite{GO}. We apply it to define group homomorphisms $K_0(\DG^{\gm}_{\bbC})\to K_0^{\oplus}(\Z)$ by taking topological $K$-groups, where~$K_0^{\oplus}(\Z)$ is the Grothendieck group of the additive category of finitely generated abelian groups (with split exact sequences).

\medskip

Section~\ref{sec:noneff} is of algebraic flavour and contains a description of equivariant sheaves when a group $G$ acts on a variety $X$ through a quotient $G\to H$ with a non-trivial kernel $N$. In Subsection~\ref{subsec:twoequiv}, we construct an $H$-equivariant $k$-algebra $A$ and prove an equivalence of categories $\coh^G(X)\simeq \coh^H(A\otimes_k\cO_X)$ (see Theorem~\ref{thm:equivcoh}).

We show in Subsection~\ref{subsection:Brinterp} that the center of $A$ is canonically isomorphic to the center~$Z$ of the group algebra~$k[N]$. This leads to an interpretation of the previous result in terms of a class $\theta\in \Br^H(Z)$ in the equivariant Brauer group and allows us to prove Theorem~\ref{thmF}.

Subsection~\ref{subsec:critvan} contains a useful criterion for vanishing of $\theta$ (see Proposition~\ref{cor:split}), which we apply later to prove Theorem~\ref{thmA}. We also consider several examples, including the wreath product case (see Example~\ref{exam:wreath}).

We provide in Subsection~\ref{sec:furthexam} further examples of representative of the class $\theta$. They are not used in the proofs of the main results, but we believe that they are important in their own rights and might be useful for other possible applications. Namely, we construct a representative of~$\theta$ given by a skew group algebra of $G$ with coefficients in the algebra~$\cO(H)$ of functions on $H$ (see Proposition~\ref{lemma:explskew}), we construct an explicit representative of~$\theta$ when~$G$ is the Heisenberg group over $\Z/n\Z$ (see Lemma~\ref{lem:Heis}), and we construct an example when there is a group section $H\to G$ but the class $\theta$ is still non-trivial (see Example~\ref{ex:splitnontriv} and the discussion before it). At the end of Subsection~\ref{sec:furthexam}, we give an invariant description of the category of $H$-equivariant $A$-modules enhanced in \mbox{$H$-equivariant} $Z$-modules in terms of representations of $G$ thus giving an invariant definition of the class $\theta$.

\medskip

Finally, we prove our main results in Section~\ref{sec:mainres}. Theorem~\ref{thmA} is proved in Subsection~\ref{subsec:equalmeas} with the help of Theorem~\ref{thmE} and Theorem~\ref{thmF}. We deduce Theorem~\ref{thmB} from Theorem~\ref{thmA} in Subsection~\ref{subsec:GS}. We prove Theorem~\ref{thmC} in Subsection~\ref{subsec:diffcatmeas} using a result of Bernardara~\cite{Bern} and an analysis of the Atiyah--Hirzubruch and Leray spectral sequences for Severi--Brauer schemes. Theorem~\ref{thmD} is proved in Subsection~\ref{subsec:PVdB} by using Theorem~\ref{thmF} and a result of Orlov~\cite{Orlov1} on equivalences between derived categories.

\bigskip

We are very grateful to Sergey Galkin and Evgeny Shinder for telling us about their conjecture on the categorical and motivic zeta-functions and for sharing the idea to compare the equivariant categorial measure with the measure of the extended quotient, which initiated our research on this topic. We thank Dan Abramovich, Yuval Ginosar, Darrell Haile, Ludmil Katzarkov, Alexander Polishchuk, and Constantin Shramov for useful discussions and suggestions. The first named author was partially supported by the Danish National Research Foundation through the Niels Bohr Professorship of Lars Hesselholt, by the Max Planck Institute for Mathematics in Bonn, and by the DFG through SFB/TR 45. The second named author is partially supported by Laboratory of Mirror Symmetry NRU HSE, RF Government grant, ag. \textnumero~14.641.31.0001. He is also very grateful for excellent working conditions to FIM ETH Z\"urich, where part of this work was done. The third named author was partially supported by a grant from the NSF. The fourth named author was partially supported by the NSA grant H98230-15-1-0255.

%Recently analogous isomorphisms for more general Deligne--Mumford stacks that are quotient stacks were obtained by Krishna and Sreedhar~\cite{KS}.

\section{Terminology and conventions}\label{sec:not}

Throughout the paper, we fix a ground field $k$ of characteristic zero. The condition on the characteristic is needed for two purposes. Firstly, we apply resolution of singularities and weak factorization over $k$. Secondly, given a finite group, we use semisimplicity of its category of (finite-dimensional) representations over $k$. As usual, we denote the algebraic closure of $k$ by~$\bar k$.

\medskip

For short, by a {\it variety} and an {\it algebraic group}, we mean a (not necessarily irreducible) variety and an algebraic group over the ground field $k$. In particular, any finite group naturally defines a constant algebraic group over $k$, which is a finite set of points over~$k$ as a variety. By a {\it scheme}, we mean a scheme of finite type over $k$. By an {\it algebraic space}, we mean an algebraic space, as defined in Knutson's book \cite[Def.~II.1.1]{Knu}, that is of finite type over~$k$.
In particular, any algebraic space we consider is quasi-separated, but not necessarily separated. By an {\it algebra}, we mean an associative unital $k$-algebra. By a {\it dg-category}, we mean a small $k$-linear dg-category, that is, a category enriched in the category of complexes of $k$-vector spaces. Explicitly, morphisms between any two objects in such category form a complex of $k$-vector spaces, which we call a {\it morphism-complex}.

Group actions will always be understood to be left actions. Given a finite group $G$, we call a variety endowed with an algebraic $G$-action a {\it \mbox{$G$-variety}}. If the group in question is clear from context, we sometimes simply write {\it equivariant variety}.
Given a subgroup $H\subset G$ and an $H$-variety $Y$, we denote the {\it induced $G$-variety}~$G\times_H Y$ by $\ind_H^G(Y)$.

\medskip

Given a finite set $S$, we denote the number of elements in $S$ by $|S|$. Given a finite group~$G$, we denote the category of finite-dimensional $k$-linear representations by~$\rep(G)$ and the set of isomorphism classes of irreducible representations in this category by~$\Irr(G)$. We use the notations $\rep_l(G)$ and $\Irr_l(G)$ when we consider representations of $G$ over an arbitrary field~$l$. Denote the group algebra of $G$ with coefficients in $k$ by $k[G]$. Given an element $g\in G$, we denote its centralizer in $G$ by $Z(g)$. Denote the set of conjugacy classes in~$G$ by~$C(G)$. Given a non-negative integer $n$, we denote the symmetric group that permutes~$n$ elements by $\Sigma_n$ (thus both $\Sigma_0$ and $\Sigma_1$ are the trivial group).

Given an algebra $A$, we denote its center by $Z(A)$ and the category of left $A$-modules by~$\Mods(A)$. Given an action of a finite group $G$ on $A$, we let $\Mods^{G}(A)$ denote the category of $G$-equivariant $A$-modules, that is, the category of $A$-modules $M$ with an action of $G$ such that $g(am)=g(a)g(m)$ for all $g\in G$, $a\in A$, and $m\in M$. If $A$ is finite-dimensional over~$k$, we let~$\mods(A)$ denote the category of left $A$-modules that are finite-dimensional over~$k$ and similarly for $\mods^G(A)$.

Given a variety $X$, we denote the category of coherent sheaves on $X$ by $\coh(X)$. Given a coherent sheaf of $\cO_X$-algebras $\cA$ on $X$, we say that a sheaf of $\cA$-modules is {\it coherent} provided that it is coherent when regarded as a sheaf of $\mathcal{O}_{X}$-modules. We denote the category of coherent $\cA$-modules by $\coh(\cA)$. Given an action of a finite group $G$ on $X$, we let~$\coh^G(X)$ denote the category of $G$-equivariant coherent sheaves on $X$. If $\cA$ is a $G$-equivariant coherent sheaf of $\cO_X$-algebras, we let $\coh^G(\cA)$ denote the category of \mbox{$G$-equivariant} coherent~\mbox{$\cA$-modules}.

\medskip

We will always assume that our Deligne--Mumford stacks, or {\it DM-stacks} for short, are of finite presentation over $k$ and that they have finite inertia. The latter condition means that the natural morphism from the inertia stack to a given DM-stack is finite.
Note that this is weaker than requiring the DM-stack to be separated.
By a theorem by Keel--Mori~\cite{KM}, under these hypotheses each \mbox{DM-stack~$\cX$} admits
a {\it coarse space}~$\pi\colon \cX\to \cX_{\,\cs}$, which is initial in the category of morphisms to algebraic spaces.
Furthermore, the morphism $\pi$ is a proper, quasi-finite, universal homeomorphism such that the canonical morphism $\mathcal{O}_{\cX_{\,\cs}}\to\pi_\ast\mathcal{O}_\cX$ is an isomorphism, see also papers by Conrad~\cite{Con} and Rydh~\cite{Ryd}.

Following Kresch~\cite[Def.\,5.5]{Kre}, we say that a DM-stack $\cX$ is {\it (quasi-)projective} if it admits a (locally) closed embedding into a smooth, proper DM-stack with projective coarse space. A DM-stack is quasi-projective if and only if it is isomorphic to the quotient stack~$[P/\GL_n]$, where~$P$ is a quasi-projective variety and the action of $\GL_n$ on $P$ is proper, linear and each point of $P$ is semi-stable with respect to the linearized action.
That is, the morphism $P\times\GL_n\to P\times P$ defined by the projection and the action is proper and
there is a \mbox{$\GL_n$-equivariant} line bundle~$\mathcal{L}$ on~$P$ such that $P$ admits an affine open covering of
subsets of the form $P_s:= \{ s \neq 0\}$, where $s$ is a $\GL_n$-invariant global section of some non-negative power of $\mathcal{L}$.

Indeed, assume that~$\cX$ is a quasi-projective DM-stack.
Then $\cX\simeq [P/\GL_n]$ for some quasi-projective variety $P$ with an $\GL_n$-action and $\cX_\coarse$ is quasi-projective by \cite[Prop.\,5.1, Theor.\,5.3]{Kre}. The action of $\GL_n$ on $P$ is proper, because the morphism $P\times\GL_n\to P\times P$ is the pull-back of the diagonal $\cX\to\cX\times\cX$, which is proper since $\cX$,
being a locally closed substack of separated stack by assumption, is separated.
Moreover, since the morphism $P \to \cX$ is affine, any ample line bundle on $\cX$ pulls back to an equivariant line bundle on $P$ making each point of $P$ semi-stable.
Conversely, assume that $\cX~\simeq [P/\GL_n]$, where $P$ is a quasi-projective variety on which $\GL_n$ acts properly and linearly such that $P$ is everywhere semi-stable with respect to the linearlized action.
Then $P$ is in fact everywhere stable since the action is proper.
Hence there exists a quasi-projective geometric quotient~$P/\GL_n$ by geometric invariant theory.
Since the action is proper, we have $\cX_\coarse \simeq P/\GL_n$ by a result of Koll\'ar~\cite[Cor.\,2.15]{Kol0}.
Hence $\cX$ is quasi-projective by~\cite[Theor.\,5.3]{Kre}.

\medskip

In our paper, we use the root construction, which was introduced by Cadman~\cite{Cad} and by Abramovich, Graber, Vistoli~\cite{AGV}. Given a DM-stack $\cX$, an effective Cartier divisor $\cD\subset \cX$, and a positive integer $r$, we denote the {\it $r$-th root stack} of $\cX$ in $\cD$ by~$\cX_{r^{-1}\cD}$. That is, $\cX_{r^{-1}\cD}$ is the $r$-th root stack associated with the line bundle $\cO_\cX(\cD)$ together with the section that corresponds to the canonical morphism of sheaves $\cO_\cX\to\cO_\cX(\cD)$. Recall that $\cX_{r^{-1}\cD}$ is a DM-stack with a canonical morphism $f\colon \cX_{r^{-1}\cD}\to \cX$, which is an isomorphism over~$\cX\smallsetminus \cD$. Moreover, the morphism $f$ is proper and faithfully flat and the natural morphism $f_{\,\cs}\colon (\cX_{r^{-1}\cD})_{\,\cs}\to \cX_{\,\cs}$ is an isomorphism of algebraic spaces. There is a naturally defined effective Cartier divisor $\cE\subset \cX_{r^{-1}\cD}$ such that $f^*\cD=r\cdot \cE$ and the morphism $\cE\to \cD$ is a $\mu_r$-gerbe. We call $\cE$ the {\it exceptional divisor} and we call $\cD$ the {\it center} of the root construction. The root stack of a smooth stack in a smooth center is smooth.

Let us describe the root construction explicitly in the case $\cX=[P/\GL_n]$ and ${\cD=[Q/\GL_n]}$, where $P$ is a variety with a proper action of $\GL_n$ and $Q$ is a \mbox{$\GL_n$-invariant} effective Cartier divisor on $P$. Denote the total space of the line bundle~$\cO_P(-Q)$ by $L$. The natural morphism of line bundles $\cO_P(-Q)\to\cO_P$ corresponds to a morphism of varieties $L\to P\times\bbA^1$. Its composition with the projection to $\bbA^1$ defines a regular function~$f$ on~$L$, which is linear along the fibers of $L$. The function $f$ is invariant under the natural action of $\GL_n$ on $L$. Let $L^0$ denote the complement to the zero section in $L$. Let $M\subset L^0\times \bbA^1$ be a closed subvariety given by the equation $f-t^r=0$, where $t$ is the coordinate on~$\bbA^1$. Define the action of $\Gm$ on~$L^0\times\bbA^1$ by the formula
$$
\Gm\times (L^0\times\bbA^1)\longrightarrow L^0\times\bbA^1\,,\qquad {\big(\lambda,(s,t)\big)\longmapsto (\lambda^rs,\lambda t)}\,.
$$
Then $M\subset L^0\times \bbA^1$ is invariant under this action of $\Gm$. Also,~$M$ is invariant under the action of $\GL_n$, where $\GL_n$ acts trivially on $\bbA^1$. Let $N\subset M$ be the Cartier divisor defined by the equation $t=0$. Note that $N$ is isomorphic to the complement to the zero section in the total space of the conormal bundle $\cO_Q(-Q)$ to $Q$ in $P$. There are natural isomorphism of stacks
$$
\cX_{r^{-1}\cD}\simeq [M/(\GL_n\times\Gm)]\,,\qquad \cE\simeq [N/(\GL_n\times \Gm)]\,.
$$
Thus $\cX_{r^{-1}\cD}$ is a quotient stack and we see that its coarse space is isomorphic to $\cX_{\,\cs}$. Moreover,~$\cX_{r^{-1}\cD}$ is smooth whenever $\cX$ and $\cD$, or, equivalently, $P$ and $Q$, are smooth.

\medskip

By a {\it stacky blow-up}, we mean either a blow-up of a DM-stack $\cX$ in a closed substack $\cZ\subset \cX$ or a root construction of $\cX$ in an effective Cartier divisor $\cD\subset \cX$. Note that stacky blow-ups in smooth centers preserve the class of smooth projective DM-stacks (for root constructions this follows from the discussion above and for usual blow-ups one should use that they commute with taking quotient stacks).

\medskip

Given a variety $X$, we let $K_0(X)$ denote the Grothendieck group of the exact category of vector bundles on $X$ and we let $K_i(X)$, $i\geqslant 0$, denote the algebraic $K$-groups of this exact category. If $k=\bbC$ is the field of complex numbers, then we let $K_i^{\tp}(X)$, $i\geqslant 0$, denote the topological $K$-groups of the topological space $X(\bbC)$ of complex points of $X$ with the classical topology. Recall that there are canonical isomorphisms $K^{\tp}_{2i}(X)\simeq K_0^{\tp}(X)$ and ${K^{\tp}_{2i+1}(X)\simeq K_1^{\tp}(X)}$ for any $i\geqslant 0$.

\section{Various Grothendieck rings}\label{sec:geom}

\subsection{The Grothendieck group of $G$-varieties}\label{subsect:GrothGvar}

The following notions and results are taken from the paper of Bittner~\cite{Bit} (see also a survey by Looijenga~\cite{Loo}). Fix a finite group $G$.

\begin{defi}\label{defin:Ggroups}
The {\it Grothendieck group $K_0(\Var^G_k)$ of $G$-varieties over $k$} is the abelian group generated by isomorphism classes $\{X\}^G$ of $G$-varieties $X$ over $k$ subject to the {\it scissors relations}
\begin{equation}\label{eq:Gscissor}
\{X\}^G = \{X \smallsetminus Z\}^G + \{Z\}^G\,,
\end{equation}
where $Z$ is a $G$-invariant closed subvariety of a $G$-variety $X$.
\end{defi}

%For simplicity, given a $G$-variety $X$ and an (additive) measure $\phi^G$ on $K_0(\Var^G_k)$, we denote $\phi^G\big(\{X\}^G\big)$ just by~$\phi^G(X)$. The upper index $G$ in the %expression $\phi^G(X)$ is aimed at reminding the reader that we consider $X$ not only as a variety but as a $G$-variety.

When $G$ is the trivial group $\{e\}$, we omit the upper index $G=\{e\}$ in the notation, that is, we use $K_0(\Var_k)$ and $\{X\}$ for $K_0(\Var_k^{\{e\}})$ and $\{X\}^{\{e\}}$, respectively.

\begin{remark}\label{rmk:Bittnerquot}
In \cite[\S\,7]{Bit} the extra relations ${\{\bbP(V)\}^G=\{\bbP^{n-1}\}^G}$, where $V$ is an $n$-dimensional representation of $G$ and $G$ acts trivially on $\bbP^{n-1}$,
are imposed on $K_0(\Var_k^G)$. Although these relations are natural to consider, the results that we prove below hold without them, so we do not include these relations in Definition~\ref{defin:Ggroups}.
\end{remark}

\begin{remark}\label{rmk:multl}
\hspace{0cm}
\begin{itemize}
\item[(i)]
The abelian group~$K_0(\Var_k)$ is a ring with multiplication defined by the natural formula~${\{X\}\cdot \{Y\}:=\{X\times Y\}}$ and with unit being the class $\{*\}$ of the point $*=\Spec(k)$.
\item[(ii)]
One has a ring structure on the abelian group $K_0(\Var^G_k)$ with multiplication defined by the formula $\{X\}^G\cdot \{Y\}^G:=\{X\times Y\}^G$, where $G$ acts diagonally on $X\times Y$. However this is not a convenient ring if one wants to compare it with the Grothendieck rings of DM-stacks (see Definition~\ref{defi:Grothstacks} and Remark~\ref{remark:alphabeta} below).
\end{itemize}
\end{remark}

\medskip

Recall that any $G$-variety $X$ (or more generally, any $G$-algebraic space $X$) admits a canonical quotient morphism
$X \to X/G$ that is initial among $G$-invariant morphisms to algebraic spaces.
Indeed, this is a special case of the theorem by Keel--Mori \cite{KM} mentioned in the previous section.
Note that $X/G$ may be a genuine algebraic space even if $X$ is a variety (see, e.g., Vistoli's explanation in~\cite[\S\,4.4.2]{Vis}). However, if $X$ is a quasi-projective $G$-variety, then every $G$-orbit in $X$ is contained in an affine open subset of $X$. This implies that $X/G$ is a variety and, in addition, a geometric quotient, see~\cite[Exp.\,V, Prop.\,1.8]{SGA}. In fact, in this situation $X/G$ is even quasi-projective.

\begin{remark}\label{rmk:Ggener}
\hspace{0cm}
\begin{itemize}
\item[(i)]
It is also possible to consider smaller classes of $G$-varieties in Definition~\ref{defin:Ggroups}. Namely, one can take isomorphism classes of either quasi-projective or smooth quasi-projective or quasi-affine or smooth quasi-affine $G$-varieties subject to the scissors relations~\eqref{eq:Gscissor} and obtain the same group $K_0(\Var^G_k)$. The reason is that any $G$-variety admits a finite \mbox{$G$-equivariant} stratification with strata being smooth affine \mbox{$G$-varieties}. When~$G$ is trivial, this is obvious. For an arbitrary group $G$, this follows from the fact that any algebraic space admits a finite stratification with strata being affine schemes, see Knutson's book~\cite[Ch.\,2, Prop.\,6.6]{Knu}, which we apply to the algebraic space $X/G$ for a given $G$-variety~$X$.
\item[(ii)]
On the other hand, one can consider in Definition~\ref{defin:Ggroups} a larger class than that of all \mbox{$G$-varieties}. Namely, one can take isomorphism classes of either $G$-schemes or $G$-algebraic spaces subject to the scissors relations~\eqref{eq:Gscissor} and obtain the same group~${K_0(\Var^G_k)}$. Indeed, by definition, the class of a $G$-scheme $X$ in the Grothendieck group is equal to the class of the corresponding closed reduced $G$-subscheme $X_{\red}$, which is a variety. For a $G$-algebraic space $X$, apply the same argument as in~(i) using the fact that the quotient $X/G$ is again an algebraic space.

In particular, for any $G$-algebraic space $X$, one has a well defined class~${\{X\}^G := \{X_0\}^G + \cdots + \{X_n\}^G}$ in~$K_0(\Var^G_k)$, where $X=X_0\cup \ldots\cup X_n$ is a stratification of $X$ into $G$-invariant varieties.
\item[(iii)]
It follows from part~(i) and from the existence of $G$-equivariant smooth compactifications for smooth affine $G$-varieties over $k$ (see the proof of~\cite[Lem.\,7.1]{Bit} for a non-trivial $G$) that the group $K_0(\Var^G_k)$ is generated by isomorphism classes of smooth projective \mbox{$G$-varieties}. However, note that the scissors relations~\eqref{eq:Gscissor} do not make sense for these generators as these relations involve open subvarieties.
\end{itemize}
\end{remark}

\medskip

\begin{defi}\label{defi:GBittner}
Let $K_0(\Var_k^{G,\smp})$ be the abelian group generated by isomorphism classes~$\{X\}^G$ of smooth projective $G$-varieties $X$ over $k$ subject to the {\it blow-up relations}
\begin{equation}\label{eq:GBittner}
\{X\}^G - \{Z\}^G = \{\widetilde{X}\}^G -\{E\}^G\,, \qquad \{\varnothing\}^G = 0\,,
\end{equation}
where $\widetilde{X}\to X$ is the blow-up of a smooth projective $G$-variety $X$ in a $G$-invariant smooth center $Z\subset X$ and $E\subset\widetilde{X}$ is the exceptional divisor. \end{defi}

\begin{remark}\label{remark:disjoint}
The blow-up relations~\eqref{eq:GBittner} imply the equality in $K_0(\Var^{G,\smp}_k)$
$$
\{X\sqcup Y\}^G=\{X\}^G+\{Y\}^G
$$
for all smooth projective $G$-varieties $X$ and $Y$. Indeed, by definition, the blow-up of $X\sqcup Y$ in~$Y$ is just $X$ with an empty exceptional divisor.
\end{remark}

The natural isomorphism of $G$-varieties $X\smallsetminus Z\simeq \widetilde{X}\smallsetminus E$ implies that the blow-up relations~\eqref{eq:GBittner} hold in the group $K_0(\Var^G_k)$. Therefore one has a natural homomorphism of groups
\begin{equation}\label{eq:Bittner}
{K_0(\Var_k^{G,\smp})\longrightarrow K_0(\Var^G_k)}\,.
\end{equation}
It is proved by Bittner in~\cite[Theor.\,3.1, Lem.\,7.1]{Bit} that homomorphism~\eqref{eq:Bittner} is an isomorphism. Essentially, the proof is based on the weak factorization theorem for smooth projective varieties over~$k$ proved by Abramovich, Karu, Matsuki, W{\l}odarczyk in~\cite{AKMW} and by W{\l}odarczyk in~\cite{Wlo} (note that weak factorization holds for smooth projective $G$-varieties as well). We reproduce a part of the argument below in Lemma~\ref{lemma:Bittner} in the context of DM-stacks. The isomorphism~\eqref{eq:Bittner} is usually called the
{\it Bittner presentation} of the group~$K_0(\Var^G_k)$.

%\begin{remark}\label{rmk:Bittner}
%Similarly as in Remark~\ref{rmk:Ggener}(ii), one can take in Definition~\ref{defi:GBittner} isomorphism classes of either smooth proper %\mbox{$G$-varieties} or smooth proper $G$-algebraic spaces subject to the blow-up relations~\eqref{eq:GBittner} and obtain the same group %$K_0(\Var^{G,\smp}_k)\simeq K_0(\Var^G_k)$. The reason is that the weak factorization theorem still holds for smooth proper $G$-varieties and %smooth proper \mbox{$G$-algebraic} spaces, see~\cite[\S\,0.3]{AKMW}.
%\end{remark}

\subsection{The Grothendieck ring of equivariant varieties}\label{subsec:Grotheq}

In the context of comparison with the Grothendieck rings of DM-stacks, it is natural to collect all together the groups $K_0(\Var^G_k)$ as follows.

\begin{defi}\label{def:Groeq}
\hspace{0cm}
\begin{itemize}
\item[(i)]
Define the ring $\bigoplus\limits_{G}K_0(\Var_k^G)$, where the direct sum is taken over the set of isomorphism classes of finite groups and multiplication is given by ${\{X\}^G\cdot \{Y\}^H:=\{X\times Y\}^{G\times H}}$ with the natural action of $G\times H$ on $X\times Y$.
\item[(ii)]
The {\it Grothendieck ring $K_0(\Var_k^{\eq})$ of equivariant varieties over $k$} is the quotient of the ring $\bigoplus\limits_{G}K_0(\Var_k^G)$ by the subgroup generated by elements of type
\begin{equation}\label{eq:eq1}
\{Y\}^H-\{\ind_H^G(Y)\}^G\,,
\end{equation}
where $Y$ is an $H$-variety and we have an embedding of groups $H\subset G$. We call elements in formula~\eqref{eq:eq1} {\it induction relations}.
\end{itemize}
\end{defi}

The ring $K_0(\Var_k^{\eq})$ is also defined by Gusein-Zade, Luengo, Melle-Hern\'andez in~\cite[Def.\,6]{GZ}, where it is denoted by~$K_0^{\rm fGr}(\Var_k)$. Related groups are considered by Getzler and Pandharipande in~\cite{GP} and by Bergh in~\cite{Ber0}.

\begin{remark}\label{rem:EgGr}
\hspace{0cm}
\begin{itemize}
\item[(i)]
The subgroup of $\bigoplus\limits_{G}K_0(\Var_k^G)$ generated by the induction relations~\eqref{eq:eq1} is indeed an ideal, because taking products commutes with induction. Namely, for a finite group $F$ and an \mbox{$F$-variety} $Z$, there is a $(G\times F)$-equivariant isomorphism ${\ind_H^G(Y)\times Z\simeq \ind_{H\times F}^{G\times F}(Y\times Z)}$.
\item[(ii)]
Since the group $\bigoplus\limits_{G}K_0(\Var_k^G)$ is generated by classes of smooth projective equivariant varieties (see Remark~\ref{rmk:Ggener}(iii)) and for an embedding of groups $H\subset G$, the map ${\ind_H^G\colon K_0(\Var^H_k)\to K_0(\Var^G_k)}$ is a group homomorphism, we see that one obtains the same ring $K_0(\Var_k^{\eq})$ if one assumes in Definition~\ref{def:Groeq}(ii) that $Y$ is, additionally, a smooth projective $H$-variety. In other words, $K_0(\Var_k^{\eq})$ is generated by classes of smooth projective equivariant varieties subject to the blow-up relations~\eqref{eq:GBittner} and the induction relations~\eqref{eq:eq1}.
\end{itemize}
\end{remark}

In order to study the ring $K_0(\Var_k^{\eq})$, one often considers homomorphisms
$$
\rho\;:\; K_0(\Var_k^{\eq})\longrightarrow A\,,
$$
where $A$ is an abelian group. Such homomorphisms are usually called {\it motivic measures}. If~$A$ is in addition a ring, then one sometimes requires $\rho$ to be a ring homomorphism. To ease notation, we put
$$
\rho^G(X):=\rho(\{X\}^G)
$$
for a $G$-variety $X$. In particular, if $G$ is trivial, we write~$\rho(X)$ instead of~$\rho(\{X\})$.

Let $X$ be a $G$-variety and let $Z \subset X$ be a $G$-invariant closed subvariety.
Note that since we are working over a field of characteristic zero,
the quotient $Z/G$ is a closed subspace of~$X/G$ (see, e.g.,~\cite[Prop.\,A.3]{Ber} for a proof of this well-known fact).
In particular, it easy to check that there is a well-defined ring homomorphism (see Remark~\ref{rmk:Ggener}(ii) for the definition of the class $\{X/G\}$ when $X/G$ is an algebraic space)
\begin{equation}\label{eq:pi}
\gamma\;:\;K_0(\Var^{\eq}_k)\longrightarrow K_0(\Var_k)\,,\qquad \{X\}^G\longmapsto \{X/G\}\,.
\end{equation}
The composition of the natural homomorphism $i\colon K_0(\Var_k)\to K_0(\Var_k^{\eq})$ with $\gamma$ is the identity, whence $i$ is injective. Thus, for simplicity of notation, we usually omit $i$.

\medskip

We will use also the following system of additive generators of $K_0(\Var^{\eq}_k)$.

\begin{lemma}\label{lemma:strat}
Let $G$ be a finite group and $X$ be a $G$-variety. Let $\Lambda$ denote the set of conjugacy classes of subgroups of $G$,
and choose a system of representatives $\{N_\lambda\}_{\lambda \in \Lambda}$ for~$\Lambda$. Denote the normalizer of $N_{\lambda}$ in $G$ by $G_{\lambda}$ and let~$H_{\lambda}:=G_{\lambda}/N_{\lambda}$.
\begin{itemize}
\item[(i)]
There is a (uniquely defined) locally closed subvariety $Y_\lambda\subset X$ such that $Y_{\lambda}(\bar k)$ consists of all $\bar k$-points of $X$ having $N_\lambda$ as stabilizer.
\item[(ii)]
We have a well-defined action of $G_\lambda$ on $Y_\lambda$ with kernel $N_\lambda$ such
that the induced action by $H_\lambda$ is free.
\item[(iii)]
There is an equality
$$
\{X\}^G=\sum_{\lambda\in\Lambda}\{Y_{\lambda}\}^{G_{\lambda}}
$$
in $K_0(\Var_k^{\eq})$.
\end{itemize}
\end{lemma}
\begin{proof}
$(i)$ We have
$$
Y_{\lambda}=X^{N_{\lambda}}\smallsetminus \bigcup_{N\varsupsetneq N_{\lambda}} X^{N}\,.
$$

$(ii)$ This follows directly from the definition of $Y_{\lambda}$.

$(iii)$ For each $\lambda\in\Lambda$, let $X_{\lambda}\subset X$ be the locally closed subvariety formed by all $\bar k$-points in~$X$ whose stabilizers belong to the conjugacy class~$\lambda$. In particular, $Y_{\lambda}\subset X_{\lambda}$. We have a canonical $G$-equivariant isomorphism
\begin{equation}\label{eq:free}
X_{\lambda}\simeq \ind_{G_{\lambda}}^G(Y_{\lambda})\,.
\end{equation}
On the other hand, we have a finite $G$-equivariant stratification
\begin{equation}\label{eq:strat}
X=\bigcup_{\lambda\in\Lambda}X_{\lambda}\,.
\end{equation}
Therefore the equalities
$$
\{X\}^G=\sum_{\lambda\in\Lambda}\{X_{\lambda}\}^G=\sum_{\lambda\in\Lambda}\{\ind_{G_{\lambda}}^G(Y_\lambda)\}^G=
\sum_{\lambda\in\Lambda}\{Y_\lambda\}^{G_\lambda}
$$
hold in $K_0(\Var_k^{\eq})$.
\end{proof}

Note that the stratification obtained in formula~\eqref{eq:strat} is the coarsest among all stratifications with the property expressed in formula~\eqref{eq:free} such that $G_{\lambda}$ acts through a free action of $H_{\lambda}$ on~$Y_{\lambda}$.

\subsection{Extended quotients}\label{subsec:extquot}

Let $G$ be a finite group and $X$ be a $G$-variety.

\begin{defi}\label{def:extquot}
\hspace{0cm}
\begin{itemize}
\item[(i)]
The {\it inertia} of $X$ is a $G$-variety defined as the equalizer
$\Eq(G\times X\rightrightarrows X)$ of two \mbox{$G$-equivariant} morphisms given by the action of $G$ on $X$ and by the projection. Here, we consider the diagonal action of $G$ on $G\times X$, where $G$ acts on itself by conjugation.
\item[(ii)]
The {\it extended quotient} $X/^{\ex}G$ of $X$ is the algebraic space defined as the quotient of the inertia:
$$
X/^{\ex}G:=\Eq(G\times X\rightrightarrows X)/G\,.
$$
\end{itemize}
\end{defi}

Note that if a \mbox{$G$-variety}~$X$ is quasi-projective, then the extended quotient is a quasi-projective variety as well.

Explicitly, $\bar k$-points of the inertia $\Eq(G\times X\rightrightarrows X)$ are given by pairs $(x,g)$, where $x$ is a~\mbox{$\bar k$-point} of $X$ and $g$ is an element of the stabilizer of $x$ in $G$. Taking the $G$-equivariant projection from the inertia to $G$, we see that the inertia decomposes \mbox{$G$-equivariantly} into a disjoint union as follows:
\begin{equation}\label{eq:inerexpl}
\Eq(G\times X\rightrightarrows X)=\coprod_{[g]\in C(G)}\ind_{Z(g)}^G(X^g)\,,
\end{equation}
where $g\in G$ is a representative of a conjugacy class $[g]\in C(G)$, $X^g\subset X$ is the closed subvariety of $g$-fixed points, and $Z(g)\subset G$ is the centralizer of $g$ in $G$. Hence we have a decomposition of the extended quotient into a disjoint union
\begin{equation}\label{eq:extquotexpl}
X/^{\ex}G=\coprod_{[g]\in C(G)} X^g/Z(g)\,.
\end{equation}

\begin{example}\label{ex:extquottriv}
Given an exact sequence of finite groups
$$
1\longrightarrow N\longrightarrow G\longrightarrow H\longrightarrow 1\,,
$$
let $G$ act on a variety $X$ through a free action of $H$ on $X$. Then the inertia and the extended quotient are as follows:
$$
\Eq(G\times X\rightrightarrows X)=N\times X\,,\qquad X/^{\ex}G=\big(C(N)\times X\big)/H\,,
$$
where $H$ acts diagonally on $C(N)\times X$ with the action on $C(N)$ by conjugation. The collection of varieties $X^g/Z(g)$, $[g]\in C(G)$, from formula~\eqref{eq:extquotexpl} coincides with the collection of varieties~$X/H_c$ parameterized by $H$-orbits $[c]\in C(N)/H$, where $c\in C(N)$ is a representative of an $H$-orbit~$[c]$ and $H_c$ is the stabilizer of $c$ in $H$. In particular, if $N$ is central, then there are equalities~${X/^{\ex}G=N\times (X/H)=\coprod\limits_{i=1}^{|N|} X/H}$.
\end{example}

\medskip

The following fact was also essentially proved in~\cite[Lem.\,1]{GZ} in different terms.

\begin{lemma}\label{lem:iner}
\hspace{0cm}
\begin{itemize}
\item[(i)]
Taking inertia defines a ring endomorphism
$$
{\rm iner}\;:\; K_0(\Var_k^{\eq})\longrightarrow K_0(\Var_k^{\eq})\,,\qquad \{X\}^G\longmapsto \{\Eq(G\times X\rightrightarrows X)\}^G\,.
$$
\item[(ii)]
For any finite group $G$ and a $G$-variety $X$, there is an equality
$$
\iner\,\{X\}^G=\sum_{[g]\in C(G)}\{X^g\}^{Z(g)}\,.
$$
\end{itemize}
\end{lemma}
\begin{proof}
$(i)$ First one checks easily that for each finite group $G$, taking inertia is a well-defined endomorphism of the group $K_0(\Var^G_k)$. This gives a map from~$\bigoplus\limits_{G}K_0(\Var_k^G)$ to itself, which is a ring endomorphism, because taking equalizers commutes with products. Now we need to show that this homomorphism respects the induction relations~\eqref{eq:eq1}. For this it is enough to prove that taking inertia commutes with induction.

In other words, it is enough to prove that for any $H$-variety~$Y$ and an embedding of finite groups $H\subset G$, there is a $G$-equivariant isomorphism
\begin{equation}\label{eq:indeq}
\ind_H^G\big(\Eq(H\times Y\rightrightarrows Y)\big)\simeq\Eq\big(G\times \ind_H^G(Y)\rightrightarrows \ind_H^G(Y)\big)\,.
\end{equation}
The left hand side of~\eqref{eq:indeq} is $G$-equivariantly isomorphic to the quotient by $H$ of the \mbox{$(G\times H)$-invariant} subvariety
$$
S:=\{(g,h,y)\;\mid\;h(y)=y\}\subset G\times H\times Y\,,
$$
where $y$ denotes a $\bar k$-point of $Y$ and the action of $G\times H$ on $G\times H\times Y$ is given by the formula
$$
(g',h')\;:\;(g,h,y)\longmapsto \big(g'g(h')^{-1},h'h(h')^{-1},h'(y)\big)\,.
$$
The right hand side of~\eqref{eq:indeq} is $G$-equivariantly isomorphic to the quotient by $H$ of the \mbox{$(G\times H)$-invariant} subvariety
$$
T:=\{(g_1,g_2,y)\;\mid\;(g_1g_2h^{-1},h(y))=(g_2,y) \,\,{\rm for}\,\,{\rm (unique)}\,\, h\in H\}\subset G\times G\times Y\,,
$$
where, again, $y$ denotes a $\bar k$-point of $Y$ and the action of $G\times H$ on $G\times H\times Y$ is given by the formula
$$
(g',h')\;:\;(g_1,g_2,y)\longmapsto \big(g'g_1(g')^{-1},g'g_2(h')^{-1},h'(y)\big)\,.
$$
Finally, we have a $(G\times H)$-equivariant isomorphism
$$
S\stackrel{\sim}\longrightarrow T\,,\qquad (g,h,y)\longmapsto (ghg^{-1},g,y)\,.
$$

$(ii)$ This follows directly from part~$(i)$ and formula~\eqref{eq:inerexpl}.
\end{proof}

By Lemma~\ref{lem:iner}, we obtain a ring homomorphism (see formula~\eqref{eq:pi} for $\gamma$)
\begin{equation}\label{eq:piiner}
\gamma\circ\iner\;:\;K_0(\Var_k^{\eq})\longrightarrow K_0(\Var_k)\,,\qquad \{X\}^G\longmapsto \{X/^{\ex}G\}=\sum_{[g]\in C(G)}\{X^g/Z(g)\}\,.
\end{equation}

\begin{remark}\label{rem:Eulerorb}
\hspace{0cm}
\begin{itemize}
\item[(i)]
When $k=\C$, the composition of $\gamma\circ\iner$ (respectively, of $\gamma\circ\iner^{n}$ for an integer~${n\geqslant 1}$) with the Euler characteristic map $\chi\colon K_0(\Var_{\C})\to\Z$ is called an orbifold Euler characteristic (respectively, an $n$-th orbifold Euler characteristic), which was introduced by Atiyah and Segal in~\cite{AS}, see also papers by Hirzebruch and H\"ofer~\cite{HH} and by Bryan and Fulman~\cite{BF}.
\item[(ii)]
It is shown in~\cite{AS} that for any smooth complex algebraic $G$-variety $X$, there are canonical isomorphisms
$$
K_i^{G,\tp}(X)_{\C}\simeq K_i^{\tp}(X/^{\ex}G)_{\C}\,,\qquad i=0,1\,,
$$
where $K_i^{G,\tp}(X)$ are $G$-equivariant topological $K$-groups of~$X$.
\end{itemize}
\end{remark}

In view of Remark~\ref{rem:Eulerorb}(ii), it makes sense to compare the classes $\{X\}^G$ and $\{X/^{\ex}G\}$ in~$K_0(\Var_k^{\eq})$ or, rather, to compare their images under homomorphisms from $K_0(\Var_k^{\eq})$ to some other groups or rings. With this aim, we give the following definition.

\medskip

Let $\cS$ be a (possibly, infinite) collection of exact sequences of finite groups
$$
1\longrightarrow N_s\longrightarrow G_s \longrightarrow H_s\longrightarrow 1\,,\qquad s\in \cS\,.
$$

\begin{defi}\label{def:sadeq}
\hspace{0cm}
\begin{itemize}
\item[(i)]
We say that a $G$-variety $X$ is {\it $\cS$-adequate} if for any $\bar k$-point~$x$ of~$X$, the exact sequence
$$
1\longrightarrow N_x\longrightarrow G_x\longrightarrow H_x\longrightarrow 1
$$
belongs to $\cS$, where $N_x\subset G$ is the stabilizer of $x$ and $G_x\subset G$ is the normalizer of~$N_x$ in~$G$. \item[(ii)]
Let $R_{\cS}\subset K_0(\Var_k^{\eq})$ be the subgroup generated by elements of type
\begin{equation}\label{eq:first}
\{V\}^G-\{V/G\}\,,
\end{equation}
where $V$ is a $G$-variety with a free action of a group $G$, and by elements of type
$$
\{Y\}^{G_s}-\{C(N_{s})\times Y\}^{H_{s}}\,,
$$
where $s$ runs over elements of $\cS$, $Y$ is a smooth projective $H_s$-variety, the action of~$G_s$ on $Y$ is through~$H_s$, and~$H_s$ acts on~$C(N_s)$ by conjugation.
\end{itemize}
\end{defi}

Note that formula~\eqref{eq:first} does not depend on the collection~$\cS$.

Let ${\rho\colon K_0(\Var_k^{\eq})\to A}$ be a homomorphism to an arbitrary abelian group. Lemma~\ref{lemma:strat} implies that checking
the equality in $A$
\begin{equation}\label{eq:local}
\rho^G(X)=\rho(X/^{\ex}G)
\end{equation}
for every
$\cS$-adequate $G$-variety $X$ is equivalent
to checking this equality only for special $\cS$-adequate equivariant varieties, see Proposition~\ref{lemma:techno}$(i)$ below. Note that these special $\cS$-adequate equivariant varieties are not necessarily smooth projective, even if one considers only smooth projective equivariant varieties~$X$ as above.

On the other hand, one often defines a homomorphism $\rho$ only on classes of smooth projective equivariant varieties verifying the blow-up relations~\eqref{eq:GBittner} and the induction relations~\eqref{eq:eq1} (see Remark~\ref{rem:EgGr}(ii)). Thus it is useful to find conditions ``more in terms of'' smooth projective equivariant varieties that would guarantee equality~\eqref{eq:local}. We give in Proposition~\ref{lemma:techno}$(ii)$ such sufficient conditions in terms of $R_{\cS}$. The conditions do not seem to be necessary. However, they are satisfied in our main application, that is, for the categorical measure and the collection~$\cS_0$ (see Proposition~\ref{theor:catmeas}$(ii)$ and Definition~\ref{def:S0} below).

\begin{prop}\label{lemma:techno}
Let $\rho\colon K_0(\Var_k^{\eq})\to A$ be a homomorphism of abelian groups.
\begin{itemize}
\item[(i)]
There is an equality
${\rho^G(X)=\rho(X/^{\ex}G)}$ for every $\cS$-adequate $G$-variety $X$ if and only if there is an equality ${\rho^{G_s}(Y)=\rho(Y/^{\ex}G_s)}$ for
every $s\in \cS$ and every
$G_s$-variety $Y$ such that the action of $G_s$ on $Y$ is through a free action of $H_s$ on $Y$.
\item[(ii)]
Suppose that the homomorphism $\rho$ factors through the quotient $K_0(\Var_k^{\eq})/R_{\cS}$, that is, we have $\rho(R_{\cS})=0$. Then for
every $\cS$-adequate $G$-variety $X$, there is an equality ${\rho^G(X)=\rho(X/^{\ex}G)}$.
\end{itemize}
\end{prop}
\begin{proof}
$(i)$
This follows directly from Lemma~\ref{lemma:strat} and the fact that $\gamma\circ\iner$ is a group homomorphism (see formula~\eqref{eq:piiner}).

$(ii)$
For any $s\in \cS$, the map
$$
K_0(\Var^{H_s}_k)\longrightarrow K_0(\Var^{\eq}_k)\,,\qquad
\{Y\}^{H_s}\longmapsto \{Y\}^{G_s}-\{C(N_{s})\times Y\}^{H_{s}}\,,
$$
is a homomorphism of groups. Therefore, since $K_0(\Var_k^{H_s})$ is generated by classes of smooth projective $H_s$-varieties, we see that for any $H_s$-variety $Y$, not necessarily smooth projective, we have
$$
\{Y\}^{G_s}-\{C(N_{s})\times Y\}^{H_{s}}\in R_{\cS}\,.
$$
Furthermore, suppose that the action of $H_s$ on a variety $Y$ is free. Then the action of $H_s$ on~${C(N)\times Y}$ is free as well and by relation~\eqref{eq:first}, we have
$$
\{C(N_{s})\times Y\}^{H_{s}}-\big\{\big(C(N_s)\times Y\big)/H_s\big\}\in R_{\cS}\,.
$$
By Example~\ref{ex:extquottriv}, there is an equality
$$
\big(C(N_s)\times Y\big)/H_s=Y/^{\ex}G_{s}\,.
$$
Hence we obtain that
$$
\{Y\}^{G_s}-\{Y/^{\ex}G_{s}\}\in R_{\cS}
$$
when the action of $H_s$ on $Y$ is free. Now we finish the proof using part~$(i)$.
\end{proof}

\subsection{The Grothendieck ring of smooth projective DM-stacks}\label{subsect:GrothDMstacks}

One has an analog for DM-stacks of the group $K_0(\Var^{G,\smp}_k)$ (see Definition~\ref{defi:GBittner}). Note that there are ``more birational morphisms'' between DM-stacks than between varieties, because, in general, the root construction is not decomposed as a composition of blow-ups and blow-downs. Thus in the Bittner presentation for smooth projective DM-stacks one should consider stacky blow-ups (see Section~\ref{sec:not}).

\begin{defi}\label{defi:Grothstacks}
The {\it Grothendieck ring $K_0(\DM_k^{\smp})$ of smooth projective DM-stacks} is the abelian group generated by isomorphism classes $\{\cX\}$ of smooth projective DM-stacks $\cX$ over~$k$ subject to the {\it stacky blow-up relations}
\begin{equation}\label{eq:stackBittner}
\{\cX\} -\{\cZ\} = \{\widetilde{\cX}\}-\{\cE\} \,, \qquad \{\varnothing\} = 0\,,
\end{equation}
where $\widetilde{\cX}\to \cX$ is a stacky blow-up of a smooth projective DM-stack $\cX$ in a smooth center $\cZ\subset \cX$ and $\cE\subset\widetilde{\cX}$ is the exceptional divisor. Multiplication in $K_0(\DM_k^{\smp})$ is defined by the formula $\{\cX\}\cdot \{\cY\}:=\{\cX\times \cY\}$.
\end{defi}

%For our applications we need only the group $K_0(\DM_k^{\spj})$ of smooth projective DM-stacks and not the group $K_0(\DM_k^{\smp})$ of smooth %proper DM-stacks. However, we decided to include proper DM-stacks in our considerations for the sake of completeness and because the categorical %measure $\mu\colon K_0(\DM_k^{\spj})\to K_0(\DG_k^{\gm})$ (see Proposition~\ref{theor:catmeas} below) factors through a natural homomorphism %$\beta\colon K_0(\DM_k^{\spj})\to K_0(\DM_k^{\smp})$ (see formula~\eqref{eq:manymaps} and Remark~\ref{rmk:propercatmeas}(ii) below).

As in Remark~\ref{remark:disjoint}, there is an equality in $K_0(\DM^{\smp}_k)$
$$
\{\cX\sqcup \cY\}=\{\cX\}+\{\cY\}
$$
for all smooth projective DM-stacks $\cX$ and $\cY$.

It follows from Remark~\ref{rem:EgGr}(ii) that one has a well-defined ring homomorphism
\begin{equation}\label{eq:manymaps}
\alpha\;:\;K_0(\Var^{\eq}_k)\longrightarrow K_0(\DM_k^{\smp})\,,\qquad \{X\}^G\longmapsto \{[X/G]\}\,,
\end{equation}
where $X$ is a smooth projective $G$-variety and $G$ is a finite group.

\begin{remark}\label{remark:alphabeta}
For a non-trivial finite group $G$, the composition
$$
K_0(\Var_k^G)\longrightarrow K_0(\Var_k^{\eq})\longrightarrow K_0(\DM^{\smp}_k)
$$
is not a homomorphism of rings (cf. Remark~\ref{rmk:multl}(ii)).
\end{remark}

The homomorphism $\gamma$ (see formula~\eqref{eq:pi}) equals to the composition of $\alpha$ with the ring homomorphism
$$
K_0(\DM^{\smp}_k)\longrightarrow K_0(\Var_k)\,,\qquad \{\cX\}\longmapsto \{\cX_{\,\cs}\}\,.
$$
Indeed, it is enough to check this on generators given by smooth projective equivariant varieties, in which case this is clearly true.

\medskip

The following theorem is crucial for us.

\begin{theo}\label{theo:comparison}
Let $V$ be a $G$-variety with a free action of a group $G$. Then the equality
$$
\alpha^G(V)=\alpha(V/G)
$$
holds in $K_0(\DM_k^{\smp})$.
\end{theo}

Theorem~\ref{theo:comparison} is equivalent to saying that the homomorphism $\alpha$ sends to zero elements from formula~\eqref{eq:first}. We prove Theorem~\ref{theo:comparison} in Subsection~\ref{subsec:proofcomp}.

\medskip

\begin{remark}
It is natural to define the {\it Grothendieck ring $K_0(\DM_k)$ of DM-stacks over $k$} as the abelian group generated by isomorphism classes $\{\cX\}$ of DM-stacks $\cX$ over $k$ subject to the scissors relations
$$
\{\cX\} = \{\cX \smallsetminus \cZ\} + \{\cZ\}\,,
$$
where $\cZ$ is a closed substack of a DM-stack $\cX$.
Multiplication in $K_0(\DM_k)$ is induced by taking products of stacks.

There is a natural ring homomorphism (cf. formula~\eqref{eq:Bittner})
$$
\beta\;:\;K_0(\DM_k^{\smp})\longrightarrow K_0(\DM_k)\,,
$$
which is proven to be an isomorphism in a parallel work by Bergh~\cite[Theor.\,1.1]{Ber1}.

The proof in {\it op.~cit.} uses weak factorization for proper birational maps between smooth
DM-stacks which are quotient stacks. A full proof of weak factorization for DM-stacks has recently been given by Harper in~\cite{Harper} and independently by Rydh~(unpublished).

Note that Theorem~\ref{theo:comparison} follows easily from the fact that $\beta$ is an isomorphism. In this work, we do not need the full power of this result, but choose to provide a somewhat simpler proof, which is given in the next section.
\end{remark}

\subsection{Proof of Theorem~\ref{theo:comparison}}\label{subsec:proofcomp}

Let $G$ be a finite group.

\begin{defi}
A {\it standard pair of varieties} is a pair~$(X,D)$, where $X$ is a smooth projective variety and~$D$ is a simple normal crossing divisor on~$X$. A standard pair is {\it $G$-equivariant} if $X$ is a $G$-variety and $D$ is a union of $G$-invariant smooth (not necessarily connected) divisors. We define standard pairs of DM-stacks similarly, with smooth projective varieties being replaced by smooth projective stacks.
\end{defi}

Note the slight subtlety in the definition of a $G$-equivariant standard pair, namely, that~$D$ is not just required to be $G$-invariant. The rationale behind our definition is that we want $G$-equivariant standard pairs $(X, D)$ of varieties to give standard pairs $([X/G], [D/G])$ of the quotient stacks.

\begin{example}
\label{ex:eq-standard}
Let $X$ be the product $\PP^1\times\PP^1$, $G$ be the group $\mathbb{Z}/2\mathbb{Z}$ acting by permuting the multiples, and $D$ be the union $\PP^1\times\{0\}\cup\{0\}\times\PP^1$. Then $(X, D)$ is not a standard pair. However, the blow-up $\widetilde{X}$ of $X$ at the point $\{0\}\times\{0\}$ with the preimage $\widetilde{D}$ of $D$ form a \mbox{$G$-equivariant} standard pair.
\end{example}
The technique illustrated in Example~\ref{ex:eq-standard} generalizes to $G$-varieties of arbitrary dimension.

\begin{lemma}
\label{lem:eq-standpair}
Let $X$ be a smooth projective $G$-variety and let $D$ be a reduced, $G$-invariant simple normal crossing divisor on $X$.
Then there exists a birational modification $\pi\colon X' \to X$ such that $\pi$ is an isomorphism over the complement of $D$ and $\pi^{-1}D$ is a simple normal crossings divisor that is a union of $G$-invariant smooth divisors.
\end{lemma}
\begin{proof}
We may without loss of generality assume that $X$ is of pure dimension $n$.
We construct $\pi$ as the composition
$$
X' = X_{n-1} \to \cdots \to X_0 = X
$$
of blow-ups in $G$-invariant smooth centers $Z_i \subset X_i$ as follows.
Let $D_i$ denote the strict transform of $D$ to $X_i$ and let $E_i$ denote
the exceptional locus of the composition $X_i \to X$.
Furthermore, we let $Z_i \subset X_i$ be the locus where $n-i$ of the irreducible components of $D_i$ meet.
By a simple induction argument one shows that at each step $i$, the divisor $E_i + D_i$ is a simple normal crossings divisor,
the divisor $E_i$ is a sum of smooth $G$-invariant divisors and the locus $Z_i$ is smooth, $G$-invariant and has normal crossings with $E_i + D_i$.
Here the key observation is that the locus where $n - i + 1$ of the irreducible components of $D_i$ meet is empty,
which makes $Z_i$ a disjoint union of intersections of irreducible components of~$D_i$.
In particular, after the last step the divisor $D_{n-1} = Z_{n-1}$ is smooth and $G$-invariant,
so $\pi^{-1}D = E_{n-1} + D_{n-1}$ has the desired properties.
\end{proof}

Let $(\cX,\cD)$ be a standard pair of DM-stacks. Let $I$ be the set of irreducible components~$\cD_i$, $i\in I$, of $\cD$. Given a (possibly, empty) subset $J\subset I$, put
\begin{equation}\label{eq:DJ}
\cD_{J}:=\mbox{$\bigcap\limits_{i \in J} \cD_i$}\,.
\end{equation}
In particular, we have that $\cD_{\varnothing}:=\cX$. Define the element in $K_0(\DM^{\smp}_k)$
$$
\{\cX, \cD\}:= \sum_{J\subset I} (-1)^{|J|}\{\cD_{J}\}\,,
$$
where $J$ runs over all subsets of $I$ including the empty set.

\begin{theo}\label{theo:excision}
Let $(\cX,\cD)$ be a standard pair of DM-stacks such that $U:=\cX\smallsetminus \cD$ is a (quasi-projective) variety. Then there is an equality in~$K_0(\DM_k^{\smp})$
$$
\{\cX,\cD\}=\alpha(U)\,.
$$
\end{theo}

Before we prove Theorem~\ref{theo:excision}, let us show how it implies Theorem~\ref{theo:comparison}. For this we use the following lemma.

\begin{lemma}\label{lem:stand}
Let $(X,D)$ be a $G$-equivariant standard pair of varieties and define a $G$-variety $V:=X\smallsetminus D$. Then there is an equality in $K_0(\DM^{\smp}_k)$
$$
\alpha^G(V)=\{\cX,\cD\}\,,
$$
where $\cX:=[X/G]$ and $\cD=[D/G]$.
\end{lemma}
\begin{proof}
Let $K$ be the set of irreducible components of $D$. Then the group $G$ acts naturally on $K$ and $I:=K/G$ is the set of irreducible components of~$\cD$. Furthermore, for each~${i\in I}$, we have $\cD_i\simeq [D_i/G]$, where $D_i$ is the union of irreducible components of $D$ from the corresponding $G$-orbit in $K$ (note that these components do not intersect with each other and the union is disjoint). Besides, for each subset $J\subset I$, we have $\cD_J\simeq [D_J/G]$, where, similarly as above, we put $D_J:=\bigcap\limits_{i\in J}D_i$.

Since $D=\bigcup\limits_{i\in I}D_i$, by the inclusion-exclusion principle, we have the equality in~$K_0(\Var^{\eq}_k)$
$$
\{V\}^G=\{X\}^G-\{D\}^G=\sum_{J\subset I} (-1)^{|J|}\{D_{J}\}^G\,.
$$
Thus, applying $\alpha$, we obtain the needed equality.
\end{proof}

\begin{proof}[Proof of Theorem~\ref{theo:comparison}]
By Remark~\ref{rmk:Ggener}(i), we may assume that $V$ is a smooth affine \mbox{$G$-variety} with a free action of $G$. Then there is a $G$-equivariant standard pair of varieties $(X,D)$ with a $G$-equivariant isomorphism $V\simeq X\smallsetminus D$. Indeed, this follows from the proof of~\cite[Lem.\,7.1]{Bit} and Lemma~\ref{lem:eq-standpair} (cf. Remark~\ref{rmk:Ggener}(iii)).

We obtain a standard pair of DM-stacks $([X/G],[D/G])$. (Note that $G$ may not act freely on the compactification $X$ of~$V$, so~$[X/G]$ will usually be a genuine DM-stack.) It remains to apply Lemma~\ref{lem:stand} and Theorem~\ref{theo:excision} with $\cX=[X/G]$, $\cD=[D/G]$, and~${U=V/G}$.
\end{proof}

\medskip

Now let us prove Theorem~\ref{theo:excision}. Firstly, let us recall the destackification theorem from the paper by Bergh~\cite{Ber}. For this we introduce some more notation. Let~$(\cX,\cD)$ be a standard pair of DM-stacks and let $\cZ\subset \cX$ be an irreducible smooth closed substack. Let $I$ be the set of irreducible components~$\cD_i$, $i\in I$, of $\cD$. Suppose that $\cZ$ has {\it simple normal crossings} with~$\cD$, that is, for any subset $J \subseteq I$, the intersection (defined as a fibred product) of $\cZ$ with $\cD_J$ is smooth (see formula~\eqref{eq:DJ} for~$\cD_J$). Let~${\widetilde{\cX}\to\cX}$ be a stacky blow-up of $\cX$ in~$\cZ$ (according to our terminology,~$\widetilde{\cX}$ is a root stack when $\cZ$ is a divisor), let $\widetilde{\cD}\subset \widetilde{\cX}$ be the strict transform of~$\cD$, and let~${\cE\subset \widetilde{\cX}}$ be the exceptional divisor. Then $\widetilde{\cD}\cup\cE$ is a simple normal crossing divisor on~$\widetilde{\cX}$, that is, $(\widetilde{\cX},\widetilde{\cD}\cup\cE)$ is a standard pair of DM-stacks.

We denote a transform as above by $(\widetilde{\cX},\widetilde{\cD}\cup \cE)\to (\cX,\cD)$ and call it a {\it smooth stacky blow-up} of the standard pair of DM-stacks~$(\cX,\cD)$. Note that the locus over $\cX\smallsetminus \cD$ remains unmodified if and only if we have $\cZ\subset\cD$.

More generally, we denote a sequence of such transforms by
$$
(\cX',\cD')\longrightarrow (\cX,\cD)
$$
and call it a {\it smooth stacky blow-up sequence}.
Note also that if such a blow-up sequence does not modify the locus over $\cX\smallsetminus\cD$,
then we have an isomorphism $\cX_{i+1}\smallsetminus\cD_{i+1} \simeq \cX_i\smallsetminus\cD_i$ for each step ${(\cX_{i+1},\cD_{i+1})\to (\cX_i,\cD_i)}$ in the sequence, where $0\leqslant i\leqslant n-1$ with $n$ being the number of steps, $(\cX_0,\cD_0)=(\cX,\cD)$, and $(\cX_n,\cD_n)=(\cX',\cD')$.

\medskip

The following destackification theorem is a combination of~\cite[Theor.\,1.2]{Ber} and a discussion before~\cite[Cor.\,1.4]{Ber}.

\begin{theo}\label{theor:destack}
Let $(\cX,\cD)$ be a standard pair of DM-stacks such that $\cX\smallsetminus\cD$ is a variety.
Then there is a smooth stacky blow-up sequence of standard pairs of DM-stacks ${(\cX',\cD')\to (\cX,\cD)}$ such that
\begin{itemize}
\item[(i)]
$(X,D):=(\cX'_{\,\cs},\cD'_{\,\cs})$ is a standard pair of varieties;
\item[(ii)]
the canonical morphism $(\cX',\cD')\to (X,D)$ is a smooth stacky blow-up sequence of standard pairs of DM-stacks (actually, only root constructions are involved here);
\item[(iii)]
the locus over $\cX\smallsetminus\cD$ remains unmodified, that is, there are natural isomorphisms of varieties
$$
\cX\smallsetminus \cD\simeq \cX'\smallsetminus\cD'\simeq X\smallsetminus D\,.
$$
\end{itemize}
\end{theo}

One can say that Theorem~\ref{theor:destack} gives an explicit way to pass from a stacky smooth compactification $(\cX,\cD)$ of the variety $\cX\smallsetminus\cD$ to a non-stacky one $(X,D)$ by a smooth stacky blow-up sequence.

\medskip

Secondly, we investigate how the element $\{\cX,\cD\}\in K_0(\DM_k^{\smp})$ is transformed by a smooth stacky blow-up of a projective standard pair $(\cX,\cD)$. The proof of the next lemma repeats a part of Bittner's proof of~\cite[Theor.\,3.1]{Bit} with varieties being replaced by DM-stacks. We provide it here for convenience of the reader.

\begin{lemma}\label{lemma:Bittner}
Let $(\cX,\cD)$ be a projective standard pair of DM-stacks, $\cZ\subset \cD$ be an irreducible smooth closed substack that has simple normal crossings with $\cD$, and let ${(\widetilde{\cX},\widetilde{\cD}\cup\cE)\to (\cX,\cD)}$ be a smooth stacky blow-up with center $\cZ$, where $\widetilde{\cD}$ is the strict transform of~$\cD$ and~$\cE$ is the exceptional divisor. Then there is an equality in $K_0(\DM^{\smp}_k)$
$$
\{\cX,\cD\}=\{\widetilde{\cX}, \widetilde{\cD}\cup\cE\}\,.
$$
\end{lemma}
\begin{proof}
Let $I$ be the set of irreducible components $\cD_i$, $i\in I$, of $\cD$.
The irreducible components of $\widetilde{\cD}$ are the strict transforms $\widetilde{\cD}_i$ of $\cD_i$ for $i\in I$.

Take a possibly empty
subset $J\subset I$ and put
$$
\cZ_J:=\cZ\cap \cD_J\,,\qquad \widetilde{\cD}_J:=\mbox{$\bigcap\limits_{i\in J}\widetilde{\cD}_i$}\,.
$$
(As usual, we put $\cZ_{\varnothing}:=\cZ$ and $\widetilde{\cD}_{\varnothing}:=\widetilde{\cX}$.) Then $\widetilde{\cD}_{J}$ is a stacky blow-up of the smooth projective DM-stack $\cD_{J}$ in~$\cZ_{J}$ and~$\widetilde{\cD}_J\cap \cE$ is the exceptional divisor of this stacky blow-up (note that the blow-up of a DM-stack in itself is empty). Therefore we have an equality in~$K_0(\DM_k^{\smp})$
$$
\{\cD_{J}\}-\{\cZ_J\}=\{\widetilde{\cD}_{J}\}-\{\widetilde{\cD}_J\cap\cE\}\,.
$$

On the other hand, we have equalities
$$
\{\cX,\cD\}=\sum_{J\subset I}(-1)^{|J|}\{\cD_J\}\,,\qquad \{\widetilde{\cX},\widetilde{\cD}\cup\cE\}=\sum_{J\subset I} (-1)^{|J|}\big(\{\widetilde{\cD}_{J}\}-\{\widetilde{\cD}_J\cap\cE\}\big)\,.
$$
Thus it is enough to show that the sum $\sum_J (-1)^{|J|}\{\cZ_J\}$ equals zero. Since $\cZ\subset \cD$ and $\cZ$ is irreducible, there is $i_0\in I$ such that $\cD_{i_0}$ contains $\cZ$. Then for any subset $L\subset I\smallsetminus\{i_0\}$, we have that $\cZ_L=\cZ_{L\,\cup\,\{i_0\}}$, whence $\{\cZ_{L}\}-\{\cZ_{L\,\cup\,\{i_0\}}\}=0$. Since all subsets $J\subset I$ are grouped into pairs $(L,L\cup\{i_0\})$, where $L\subset  I\smallsetminus\{i_0\}$, this finishes the proof.
\end{proof}

\medskip

Finally, we are ready to prove Theorem~\ref{theo:excision}.

\begin{proof}[Proof of Theorem~\ref{theo:excision}]
Apply Theorem~\ref{theor:destack} to the standard pair of \mbox{DM-stacks}~$(\cX,\cD)$, getting a standard pair of DM-stacks $(\cX',\cD')$ whose coarse space is a standard pair of varieties~$(X,D)$. Since the locus over~${\cX\smallsetminus \cD}$ remains unmodified, the centers of all smooth stacky blow-ups involved are contained in the corresponding simple normal crossing divisors. Thus, by Lemma~\ref{lemma:Bittner}, we have the equalities ${\{X,D\}=\{\cX',\cD'\}=\{\cX,\cD\}}$ in $K_0(\DM_k^{\smp})$. On the other hand, the isomorphism $K_0(\Var_k^{\smp})\simeq K_0(\Var_k)$ (see formula~\eqref{eq:Bittner} and
the discussion after~it) implies the equality $\alpha(U)=\{X,D\}$, which finishes the proof.
\end{proof}

\section{The categorical measure}\label{sec:cat}

\subsection{The Grothendieck ring of geometric dg-categories}\label{subsec:Grothdg}

We start by reviewing dg-categories. A general reference is Keller's paper~\cite{Ke}; see also a survey by Kuznetsov and Lunts in~\cite[\S\,3]{KL}. Let $\cM$ be a dg-category. We denote its homotopy category by $H^0(\cM)$, that is, the category whose objects are the same as the objects in $\cM$, and whose hom-sets are formed by the zeroth cohomology groups of the morphism-complexes in $\cM$. A dg-functor $\Phi\colon\cM\to\cN$ is a {\it quasi-equivalence} if~$\Phi$ induces quasi-isomorphisms on the morphism-complexes and the functor ${H^0(\Phi)\colon H^0(\cM)\to H^0(\cN)}$ is an equivalence of categories (in fact, under the first condition, it is enough to require that~$H^0(\Phi)$ is essentially surjective). Two dg-categories $\cM$ and~$\cN$ are {\it quasi-equivalent} if they are connected by a chain of quasi-equivalences, that is, if there is a diagram of dg-categories ${\cM\leftarrow \cQ_1\to\ldots\leftarrow \cQ_n\to \cN}$, where all arrows are dg-functors that are quasi-equivalences.

%By $\DG_k$ denote the category whose objects are small dg-categories (with sets of objets and morphisms belonging to a chosen %universe) and whose morphisms are dg-functors. By $\Ho(\DG_k)$ denote the localization of $\DG_k$ by all quasi-equivalences.

Let $\text{Mod-}\cM$ be the dg-category of right dg-modules over $\cM$, that is, contravariant \mbox{dg-func\-tors} from $\cM$ to the dg-category of complexes of $k$-vector spaces. We have a natural \mbox{dg-Yoneda} embedding $\cM\to\text{Mod-}\cM$, which induces a fully faithful functor ${H^0(\cM) \to H^0(\text{Mod-}\cM)}$. Note that the category $H^0(\text{Mod-}\cM)$ has a canonical triangulated structure. A right dg-module over $\cM$ is {\it perfect} if it, when considered as an object in~$H^0(\text{Mod-}\cM)$, lies in the smallest triangulated subcategory of~$H^0(\text{Mod-}\cM)$ closed under direct summands and containing~$H^0(\cM)$.

A~\mbox{dg-category}~$\cM$ is \mbox{{\it pre-tri\-an\-gu\-lated}} if $H^0(\cM)$ is a triangulated subcategory of~$H^0(\text{Mod-}\cM)$. A dg-category $\cM$ is {\it proper} if for any two objects of $\cM$, the total cohomology of their morphism-complex is finite-dimensional. A~\mbox{dg-category}~$\cM$ is {\it smooth} if the diagonal \mbox{$\cM$-bimodule} is perfect in $\text{Bimod-}\cM=\text{Mod-}(\cM^{\rm op}\otimes_k \cM)$, where the diagonal \mbox{$\cM$-bimodule} sends a pair of objects in $\cM$ to their morphism-complex. A dg-category~$\cM$ is called {\it saturated} if the homotopy category $H^0(\cM)$ is idempotent complete and $\cM$ is pre-triangulated, proper, and smooth. Note that all these properties are preserved under quasi-equivalences.

\medskip

Similarly, a triangulated category $T$ is {\it proper} if for any two objects of $T$, the graded morphism-space is finite-dimensional. Following Bondal and Kapranov~\cite[Def.\,2.5]{BK}, we say that a triangulated category $T$ is {\it saturated} if $T$ is proper and all covariant and contravariant cohomological functors from~$T$ to the category of finite-dimensional graded $k$-vector spaces are representable and corepresentable, respectively.

For any saturated dg-category $\cM$, its homotopy category $H^0(\cM)$ is saturated as a triangulated category. Indeed, since $\cM$ is smooth, the class of the diagonal bimodule in~$H^0(\text{Bimod-}\cM)$ is a direct summand of an object $P$ in the triangulated subcategory of~$H^0(\text{Bimod-}\cM)$ generated by $H^0(\cM^{\rm op}\otimes_k \cM)$. Explicitly, this means that there is a finite filtration of $P$ (in a triangulated sense) whose adjoint quotients are classes of representable $\cM$-bimodules $E_i\otimes F_i$, where $E_i$, $F_i$ are objects of $\cM$ and ${1\leqslant i\leqslant n}$ with $n$ being the length of the filtration. Taking the action of bimodules on modules and using properness of $\cM$, we see that any object of $H^0(\cM)$ can be obtained from the class of the object $E:=\bigoplus_i E_i$ by taking finite direct sums, direct summands, and at most $n-1$ cones. In other words,~$E$ is a {\it strong generator} of~$H^0(\cM)$. Since~$H^0(\cM)$ is idempotent complete and proper, we have that the triangulated category~$H^0(\cM)$ is saturated by a result of Bondal and Van den Bergh~\cite[Theor.\,1.3]{BvdB}.

\medskip

A {\it semiorthogonal decomposition} of a triangulated category $T$ is an ordered pair of full triangulated subcategories $M,N\subset T$ such that there are no non-zero morphisms from objects of $N$ to objects of $M$, and such that $T$ coincides with the smallest triangulated subcategory containing~$M$ and $N$; one denotes this by $T=\langle M,N\rangle$. A full triangulated subcategory $N$ of a triangulated category $T$ is {\it admissible} if the embedding functor $N\to T$ has a left and a right adjoint functors. In this case, we have semiorthogonal decompositions ${T=\langle N,{}^{\bot}N\rangle=\langle N^{\bot},N\rangle}$. Also, note that the composition of the embedding functor $N\to T$ with the either left or right adjoint functor $T\to N$ is isomorphic to the identity functor on~$N$.

The following fact is well known to experts and we do not pretend to any originality
either in its statement or in its proof. We include this fact here as it can be viewed as a natural reason to consider semiorthogonal decompositions and admissible subcategories in the context of saturated dg-categories.

\begin{prop}\label{prop:dg}
Let $\cT$ be a saturated dg-category and put $T:=H^0(\cT)$. Let $\cM$ be a full dg-subcategory in $\cT$ and let $M:=H^0(\cM)$ be the corresponding full subcategory in $T$. Then the following conditions are equivalent:
\begin{itemize}
\item[(i)]
the dg-category $\cM$ is saturated;
\item[(ii)]
we have that $M$ is a triangulated admissible subcategory of $T$;
\item[(iii)]
we have that $M$ is a triangulated subcategory of $T$ and there is a semiorthogonal decomposition $T=\langle M,N\rangle$.
\end{itemize}
\end{prop}
\begin{proof}
Assume condition~$(i)$. Then $M$ is saturated as a triangulated category, whence~$M$ is admissible in $T$, see~\cite[Prop.\,2.6]{BK}. Clearly, condition~$(ii)$ implies condition~$(iii)$. Finally, assume condition~$(iii)$. One easily checks that $M$ is idempotent complete and that the \mbox{dg-category} $\cM$ is pre-triangulated and proper. Also, it was proved by Lunts and Schn\"urer~\cite[Theor.\,3.24]{LS} that the dg-category~$\cM$ is smooth (see also~\cite[Cor.\,2.21]{LS2}). Thus $\cM$ is saturated.
\end{proof}

\medskip

Let $M$ be a triangulated category. A {\it dg-enhancement} of $M$ is a pre-triangulated \mbox{dg-category}~$\cM$ together with an equivalence of triangulated categories $H^0(\cM)\stackrel{\sim}\longrightarrow M$.

\medskip

Here is an example of a saturated dg-category. Let $X$ be a smooth projective variety and denote the bounded derived category of the abelian category~$\coh(X)$ by~$D^b(X)$. One has various ways to construct an enhancement of $D^b(X)$, for example, with the help of bounded below complexes of injective quasi-coherent sheaves. However, Lunts and Orlov~\cite[Theor.\,9.9]{LO} proved that, in fact, all \mbox{dg-enhancements} of $D^b(X)$ are (strongly) quasi-equivalent. In other words, we have a \mbox{dg-category~$\cD(X)$}, well-defined up to a quasi-equivalence, such that~$\cD(X)$ is a dg-enhancement of~$D^b(X)$. Actually, by a result of Schn\"urer~\cite[Theor.\,1.1]{Snu}, there is a functorial \mbox{dg-enhancement} of $D^b(X)$ simultaneously for all varieties such that it respects inverse images, direct images, and tensor products of coherent sheaves. The dg-category $\cD(X)$ is smooth, see Lunts and Schn\"urer~\cite[Theor.\,1.2]{LS3}, and one easily concludes that~$\cD(X)$ is saturated.

Following Orlov~\cite[Def.\,4.3]{Orl}, we say that a dg-category $\cM$ is {\it geometric} if there is a smooth projective variety $X$ and an admissible subcategory $N\subset D^b(X)$ such that the corresponding full dg-subcategory $\cN\subset \cD(X)$ with $H^0(\cN)=N$ is quasi-equivalent to the dg-category $\cM$. In particular, any geometric dg-category is saturated by Proposition~\ref{prop:dg}.

%Any coherent sheaf on $X\times X$ is a direct summand of a complex of vector bundles of type $p_1^*E\otimes p_2^*F$, where $p_i\colon %X\times X\to X$ are natural projections and $E$, $F$ are vector bundles on $X$, see, e.g.,~\cite[Lem.\,5.2]{Ku}. Applying this to the %structure sheaf of the diagonal and using the functorial dg-enhancement $\cD(X)$ from~\cite{Snu}, we see directly that the diagonal %$\cD(X)$-bimodule is perfect.

\medskip

In~\cite[Def.\,5.2]{BLL} Bondal, Larsen, and Lunts define the product $\cM\bullet\cN$ of dg-categories $\cM$ and $\cN$. One shows that the product of saturated dg-categories is saturated (see, e.g., \cite[Prop.\,2.17]{LS2}). Given smooth projective varieties $X$ and $Y$, the dg-category $\cD(X)\bullet\cD(Y)$ is \mbox{quasi-equivalent} to~$\cD(X\times Y)$, see~\cite[Theor.\,6.6]{BLL}. It follows that the product of geometric \mbox{dg-categories} is geometric (see also~\cite[Prop.\,1.6]{GO}).

\begin{defi}\label{defin:Deltageom}
The {\it Grothendieck ring of geometric dg-categories} is the abelian group~$K_0(\DG_k^{\gm})$ generated by quasi-equivalence classes $\{\cM\}$ of geometric $k$-linear \mbox{dg-categories} subject to the relations
\begin{equation}
\{\cM\} = \{\cM'\} + \{\cM''\}\,,
\end{equation}
where $\cM',\cM''\subset \cM$ are full dg-subcategories inducing a semiorthogonal decomposition
$$
H^0(\cM)=\langle\, H^0(\cM'),H^0(\cM'')\,\rangle\,.
$$
Multiplication in $K_0(\DG_k^{\gm})$ is given by the formula $\{\cM\}\cdot\{\cN\}:=\{\cM\bullet\cN\}$ with unit being~${1=\{\cD(k)\}}$, the class of the dg-category of finite-dimensional vector spaces.
\end{defi}

Note that by a result of Orlov~\cite[Theor.\,4.15]{Orl}, if $\cM',\cM''$ are full dg-subcategories of a \mbox{dg-category}~$\cM$ inducing a semiorthogonal decomposition
$$
H^0(\cM)=\langle\, H^0(\cM'),H^0(\cM'')\,\rangle
$$
and $\cM'$, $\cM''$ are geometric, then $\cM$ is geometric as well.

\begin{remark}
It is also possible to consider other classes of dg-categories in Definition~\ref{defin:Deltageom}. Taking only dg-categories $\cD(X)$, where $X$ is a smooth projective variety, one gets the ring~$\Gamma$ from~\cite[Def.\,8.1]{BLL}. On the other hand, taking all saturated dg-categories, one gets the ring~$K_0(\DG_k^{\sat})$ as in~\cite[Def.\,2.27]{BLL}, in Tabuada's paper~\cite[\S\,7]{Tab}, and in To\"en's paper~\cite[\S\,5.4]{ToDG}, where it is called a secondary $K$-group of $k$ (see also~\cite[\S\,2]{LS2} for other various versions of Grothendieck rings of \mbox{dg-categories}, including a $\bbZ/2\bbZ$-graded version). Clearly, one has homomorphisms of rings
$$
\Gamma\longrightarrow K_0(\DG^{\gm}_k)\longrightarrow K_0(\DG^{\sat}_k)\,.
$$
None of these homomorphisms is known to be either injective, or surjective.
\end{remark}

\subsection{The categorical measure for smooth projective DM-stacks}\label{subsec:catmeas}

Let $\cX$ be a smooth projective DM-stack. As in the case of a variety, the bounded derived category~$D^b(\cX)$ of the abelian category $\coh(\cX)$ of coherent sheaves of $\cO_{\cX}$-modules has a dg-enhancement $\cD(\cX)$ obtained with the help of bounded below complexes of injective quasi-coherent sheaves, see, e.g.,~\cite[Rem.\,A.3, Ex.\,5.4]{BLS}. In addition, all dg-enhancements of $D^b(\cX)$ are quasi-equivalent by a result of Canonaco and Stellari~\cite[Prop.\,6.10]{CS} (see also~\cite[Rem.\,5.8]{BLS}), thus the dg-category $\cD(\cX)$ is well-defined up to quasi-equivalences. It is proved by Bergh, Lunts, Schn\"urer~\cite[Theor.\,6.6]{BLS} that the \mbox{dg-category~$\cD(\cX)$} is geometric (the ground field $k$ is of zero characteristic). Also, recall the following facts.

\begin{prop}\label{prop:semiorth}
Let $\widetilde{\cX}\to\cX$ be a stacky blow-up (see Section~\ref{sec:not}) of a smooth projective \mbox{DM-stack}~$\cX$ in a smooth center $\cZ\subset\cX$ and let $\cE\subset\widetilde{\cX}$ be the exceptional divisor. Let $r$ be a positive integer defined as follows:
\begin{itemize}
\item[(i)]
if $\cZ$ is a divisor in $\cX$, then $r$ is such that $\widetilde{\cX}$ is isomorphic to the $r$-th root stack $\cX_{r^{-1}\cZ}$;
\item[(ii)]
if $\cZ$ is not a divisor in $\cX$, then $r\geqslant 2$ is the codimension of $\cZ$ in $\cX$.
\end{itemize}
Then there are semiorthogonal decompositions
$$
D^b(\widetilde{\cX})=\big\langle\underbrace{D^b(\cZ),\ldots, D^b(\cZ)}_{r-1}, D^b(\cX)\big\rangle\,,
$$
$$
D^b(\cE)=\big\langle\underbrace{D^b(\cZ),\ldots, D^b(\cZ)}_{r}\big\rangle
$$
induced by fully faithful dg-functors between dg-enhacements of the triangulated categories (in the case $(i)$, the second decomposition is just a direct sum of categories).
\end{prop}

The case~$(i)$ of Proposition~\ref{prop:semiorth} follows from results of Ishii and Ueda~\cite[Theor.\,1.5, Theor.\,1.6]{IU} or from~\cite[Theor.\,4.7]{BLS}. The case~$(ii)$ of Proposition~\ref{prop:semiorth} follows from Elagin's results \cite[Theor.\,10.1, Theor.\,10.2]{Ela}, because smooth projective DM-stacks are quotient stacks, see Section~\ref{sec:not} (the case of varieties goes back to Orlov~\cite[Theor.\,4.3]{O0}).

%Actually, by the same references, the case~$(i)$ of Proposition~\ref{prop:semiorth} is valid for smooth proper DM-stacks as well. However, we could not find a reference for the %case~$(ii)$ of Proposition~\ref{prop:semiorth} for arbitrary smooth proper DM-stacks, not necessarily projective, though it is most likely that this holds true (we shall not %use this fact).

\medskip

Proposition~\ref{prop:semiorth} implies the following important result.

\begin{prop}\label{theor:catmeas}
\hspace{0cm}
\begin{itemize}
\item[(i)]
There is a well-defined homomorphism of groups
$$
K_0(\DM_k^{\smp})\longrightarrow K_0(\DG_k^{\gm})\,,\qquad \{\cX\}\longmapsto \{\cD(\cX)\}\,.
$$
\item[(ii)]
There is a well-defined homomorphism of groups
$$
\mu\;:\;K_0(\Var_k^{\eq})\longrightarrow K_0(\DG_k^{\gm})\,,\qquad \{X\}^G\longmapsto \{\cD^G(X)\}\,,
$$
where $G$ is a finite group, $X$ is a smooth projective $G$-variety and $\cD^G(X)$ is a \mbox{dg-enhancement} of the bounded derived category $D^G(X)$ of the abelian category~$\coh^G(X)$ of $G$-equivariant coherent sheaves on $X$.
\end{itemize}
\end{prop}
\begin{proof}
$(i)$ This follows directly from Proposition~\ref{prop:semiorth} and Definition~\ref{defi:Grothstacks}.

$(ii)$ Take the composition of $\alpha$ (see formula~\eqref{eq:manymaps}) and the homomorphism from part~$(i)$.
\end{proof}

The homomorphism $\mu$ in Proposition~\ref{theor:catmeas}$(ii)$ is called a {\it categorical measure}.
For example, we have $\mu(\bbP^1)=2$ and $\mu(\bbA^1)=1$.

\begin{remark}\label{rmk:catmeas}
\hspace{0cm}
\begin{itemize}
\item[(i)]
Presumably one can show that the homomorphisms in Proposition~\ref{theor:catmeas} are, actually, ring homomorphisms, e.g., by using Schn\"urer's methods from~\cite{Snu}, but we do not need this fact.
\item[(ii)]
It is impossible to extend the homomorphism in Proposition~\ref{theor:catmeas}$(i)$ to all smooth projective Artin stacks. For example, for the smooth projective Artin stack $[*/\Gm]$, the category~$D^b([*/\Gm])$ is isomorphic to the direct sum of countably many copies of~$D^b(k)$. Thus the triangulated category $D^b([*/\Gm])$ does not have a strong generator and is not equivalent to the homotopy category of a saturated dg-category.
\item[(iii)]
Clearly, the composition ${K_0(\Var_k)\to K_0(\Var^{\eq}_k)\stackrel{\mu}\to K_0(\DG_k^{\gm})}$
factors through the homomorphism $K_0(\Var_k)\to \Gamma$ defined in~\cite[\S\,8.2]{BLL}.
\item[(iv)]
The composition ${K_0(\Var_k^G)\to K_0(\Var^{\eq}_k)\to K_0(\DG_k^{\gm})}$ factors through the quotient of the group $K_0(\Var^G_k)$ considered in~\cite[\S\,7]{Bit} (cf. Remark~\ref{rmk:Bittnerquot}). Namely, it follows from a result of Elagin~\cite[Theor.\,2.1]{Ela0} that for an $n$-dimensional $G$-representation $V$, we have
$$
\mu^G\big(\bbP(V)\big)=n\cdot\mu([*/G])=\mu^G\big(\bbP^{n-1}\big)\,,
$$
where we consider the trivial action of $G$ on $\bbP^{n-1}$.
\end{itemize}
\end{remark}

\subsection{$K$-motives of geometric dg-categories}\label{subsec:DGKM}

Let us explain a relation between geometric dg-categories and $K$-motives. We use this in one of our main results (see Subsection~\ref{subsec:diffcatmeas}).

The category $\KM_k$ of {\it $K$-motives} over $k$ is defined as the idempotent completion of the additive category whose objects are smooth projective varieties and whose morphisms are \mbox{$K_0$-groups} of the products of varieties, see more details in Manin's exposition~\cite{Ma}. The category~$\KM_k$ has a symmetric monoidal structure that comes from products of varieties. We have a covariant functor $X\mapsto KM(X)$ from the category of smooth projective varieties to the category of $K$-motives. Also, algebraic $K$-groups are well-defined for $K$-motives.

\medskip

Given smooth projective varieties $X$, $Y$ and an object $\cE\in D^b(X\times Y)$, denote the corresponding Fourier--Mukai functors between derived categories by
$$
FM_{\cE_*}\;:\; D^b(X)\longrightarrow D^b(Y)\,,\qquad \cF\longmapsto R\,p_{Y*}\big(p_{X}^*(\cF)\otimes^L \cE\big)\,,
$$
$$
FM_{\cE}^*\;:\; D^b(Y)\longrightarrow D^b(X)\,,\qquad \cG\longmapsto R\,p_{X*}\big(p_{Y}^*(\cG)\otimes^L \cE\big)\,.
$$
Also, denote the corresponding morphisms between $K$-motives by
$$
[\cE]_*\;:\; KM(X)\longrightarrow KM(Y)\,,\qquad [\cE]^*\;:\; KM(Y)\longrightarrow KM(X)\,.
$$
Let $\Phi\colon \cD(X)\to \cD(Y)$ be a quasi-functor, that is, $\Phi$ is given by a diagram of \mbox{dg-categories} ${\cD(X)\leftarrow \cQ_1\to\ldots \leftarrow \cQ_n\to \cD(Y)}$, where all arrows are dg-functors and the backward arrows are quasi-equivalences. Then To\"en's theorem~\cite[Theor.\,8.15]{To}
asserts that there exists an object ${\cE\in D^b(X\times Y)}$ such that the functor from $D^b(X)$ to $D^b(Y)$ induced by $\Phi$ is isomorphic to~$FM_{\cE*}$ (this was previously proved by Bondal, Larsen, Lunts~\cite[Theor.\,7.7]{BLL} in the case when~$\Phi$ is an actual dg-functor).
%In addition, the isomorphism class of $\cE$ in $D^b(Y\times Y')$ is defined uniquely by the quasi-equivalence class of the %quasi-functor $\Phi$, see~\cite{???}. TO CHECK THIS AND TO GIVE A REFERENCE!!!

It is shown by Gorchinskiy and Orlov~\cite[\S\,4]{GO} that for a smooth projective variety~$X$ and a triangulated admissible subcategory $N\subset D^b(X)$, one has a well-defined $K$-motive~$KM(X,N)$, which is a direct summand in $KM(X)$. Let us recall this construction. Let
$$
p\;:\; D^b(X)\longrightarrow N
$$
be the left (or right) adjoint to the embedding $i\colon N\to D^b(X)$. As explained in {\it op.~cit.}, the functor $p$ is induced by a quasi-functor
$$
\Pi\;:\; \cD(X)\longrightarrow\cN\,.
$$
Also, the functor $i$ is clearly induced by a dg-functor from $\cN$ to $\cD(X)$, which we denote similarly to ease notation. We see that the functor ${i\circ p\colon D^b(X)\to D^b(X)}$ is induced by the quasi-functor ${i\circ \Pi\colon \cD(X)\to \cD(X)}$. Hence by~\cite[Theor.\,8.15]{To}, there exists an object
$$
\cP\in D^b(X\times X)
$$
such that~${i\circ p\simeq FM_{\cP*}}$. One checks that $[\cP]_*\colon KM(X)\to KM(X)$ is a projector and one defines $KM(X,N)$ as the image of $[\cP]_*$.

Recall that the $K_0$-group of a triangulated category is generated by isomorphism classes of objects subject to a relation defined naturally by distinguished triangles. The endofunctor~$i\circ p$ induces a projector from ${K_0\big(D^b(X)\big)\simeq K_0(X)}$ to itself, whose image is $K_0(N)$. It follows that we have an isomorphism ${K_0\big(KM(X,N)\big)\simeq K_0(N)}$.

\begin{prop}\label{prop:Kgeom}
Let $X$, $X'$ be smooth projective varieties and let
${N\subset D^b(X)}$, ${N'\subset D^b(X')}$ be admissible subcategories. Let $\cN\subset\cD(X)$, $\cN'\subset\cD(X')$ be the corresponding full \mbox{dg-subcategories} with $H^0(\cN)=N$, $H^0(\cN')=N'$. Suppose that the dg-categories~$\cN$ and~$\cN'$ are quasi-equivalent. Then the $K$-motives $KM(X,N)$ and $KM(X',N')$ are isomorphic.
\end{prop}
\begin{proof}
We let $p$, $\Pi$, $\cP$ be as above and we use analogous notation in the case of $X'$ and~$N'$. By assumption of the proposition, there is a quasi-functor $\Psi\colon \cN\to \cN'$ which defines a quasi-equivalence between $\cN$ and $\cN'$. Let $G\colon N\to N'$ be the corresponding equivalence between triangulated categories.
The composition
$$
F\;:\; D^b(X)\stackrel{p}\longrightarrow N\stackrel{G}\longrightarrow N'\stackrel{i'}\longrightarrow D^b(X')\,.
$$
is induced by a quasi-functor
$$
\cD(X)\stackrel{\Pi}\longrightarrow \cN\stackrel{\Psi}\longrightarrow \cN'\stackrel{i'}\longrightarrow \cD(X')\,.
$$
Therefore, by~\cite[Theor.\,8.15]{To}, there exists an object~${\cE\in D^b(X\times X')}$ such that~${F\simeq FM_{\cE*}}$.
Let $f$ denote the composition
$$
KM(X,N)\longrightarrow KM(X)\stackrel{[\cE]_*}\longrightarrow KM(X')\longrightarrow KM(X',N')\,,
$$
in the category of $K$-motives, where the first morphism is the embedding and the last morphism is the projection.

Let us show that $f$ is an isomorphism. By construction, $f$ induces an isomorphism between $K_0$-groups
$$
K_0\big(KM(X,N)\big)\simeq K_0(N)\stackrel{G}\longrightarrow K_0(N')\simeq K_0\big(KM(X',N')\big)\,.
$$
Let $Y$ be an arbitrary smooth projective variety. Then $f$ defines the morphism of $K$-motives
$$
f\otimes{\rm id}\;:\; KM(X,N)\otimes KM(Y)\longrightarrow KM(X',N')\otimes KM(Y)\,.
$$
Replacing $X$ by $X\times Y$ and replacing the admissible subcategory $N\subset D^b(X)$ by the admissible subcategory ${N\boxtimes D^b(Y)\subset D^b(X\times Y)}$, we see that the above argument implies that $f\otimes{\rm id}$ also induces an isomorphism between $K_0$-groups
$$
K_0\big(KM(X,N)\otimes KM(Y)\big)\simeq K_0\big(N\boxtimes D^b(Y)\big)\stackrel{G\boxtimes{\rm id}}\longrightarrow
$$
$$
\stackrel{G\boxtimes{\rm id}}\longrightarrow K_0\big(N'\boxtimes D^b(Y)\big)\simeq K_0\big(KM(X',N')\otimes KM(Y)\big)\,.
$$
Now by Manin's identity principle~\cite[\S\,3]{Ma} (that is, by Yoneda lemma for the category of $K$-motives), we see that $f$ is an isomorphism of $K$-motives.
\end{proof}

It is not clear whether Proposition~\ref{prop:Kgeom} remains valid if one assumes only that the triangulated categories $N$ and $N'$ are equivalent, without requiring that this equivalence is induced by a quasi-equivalence of dg-categories. (Note that, according to experts, there exist equivalences between admissible subcategories that are not induced by
quasi-equivalences between the corresponding dg-categories.)

\medskip

Now let $\cM$ be a geometric \mbox{dg-category}. Let $X$ be a smooth projective variety and let ${N\subset D^b(X)}$ be an admissible subcategory such that $\cM$ is quasi-equivalent to the full \mbox{dg-subcategory} $\cN\subset \cD(X)$ with $H^0(\cN)=N$. Then, by Proposition~\ref{prop:Kgeom}, the isomorphism class of the $K$-motive $KM(X,N)$ is well-defined by the quasi-equivalence class of~$\cM$; that is, it does not depend on the choices of~$X$ and~$N$. Thus it makes sense to denote the \mbox{$K$-motive} $KM(X,N)$ just by $KM(\cM)$ (in what follows, we consider only isomorphism classes of \mbox{$K$-motives}).

It is easy to show that the assignment $\cM\mapsto KM(\cM)$ is additive with respect to semiorthogonal decompositions and is monoidal with respect to products of geometric dg-categories. Hence we obtain a homomorphism of rings
\begin{equation}\label{eq:DGKM}
K_0(\DG^{\gm}_k)\longrightarrow K_0(\KM_k)\,,\qquad \{\cM\}\longmapsto \{KM(\cM)\}\,,
\end{equation}
where relations in the Grothendieck group of the additive category $\KM_k$ are given by direct sums of objects. This is a geometric version of the homomorphism from $K_0(\DG_k^{\sat})$ to the Grothendieck group of non-commutative motives constructed in~\cite[Prop.\,7.1]{Tab}.

\medskip

Now assume that $k$ is the field of complex numbers $\bbC$. Later we shall use the following homomorphisms on the group $K_0(\DG_{\C}^{\gm})$ defined by topological $K$-groups.

For any complex smooth projective variety $X$, we have a natural homomorphism of rings $K_0(X)\to K_0^{\tp}(X)$, which commutes with pull-backs. By Atiyah and Hirzebruch, see~\cite[Theor.\,4.2]{AH}, it also commutes with  push-forwards. Therefore we have a homomorphism of the (non-commutative) rings of correspondences ${K_0(X\times X)\to K_0^{\tp}(X\times X)}$, where multiplication is defined by composition of correspondences. Since the groups $K_i^{\tp}(X)$, $i=0,1$, are modules over the ring of correspondences~${K_0^{\tp}(X\times X)}$, we see that topological $K$-groups are well-defined for $K$-motives.

Also, recall that since complex algebraic varieties carry the structure of a finite CW-complex, their topological $K$-groups are finitely generated, see~\cite[Cor.\,2.5]{AH0}. Hence we obtain additive functors
$$
\KM_\bbC\longrightarrow \mods(\bbZ)\,,\qquad KM(X)\longmapsto K_i^{\tp}(X)\,,\qquad i=0,1\,,
$$
where $\mods(\bbZ)$ is the category of finitely generated abelian groups and $X$ is a complex smooth projective variety. Applying $K_0$ to these homomorphisms, we obtain homomorphisms of abelian groups
\begin{equation}\label{eq:KtopKm}
K_0(\KM_\bbC)\longrightarrow K_0^{\oplus}(\bbZ)\,,\qquad \{KM(X)\}\longmapsto \{K_i^{\tp}(X)\}\,,\qquad i=0,1\,,
\end{equation}
where $K^{\oplus}_0(\bbZ)$ denotes the Grothendieck group of the additive category of finitely generated abelian groups with relations given by direct sums of abelian groups (not by arbitrary exact sequences).

\begin{defi}\label{defi:kappa}
Let
$$
\kappa_i^{\tp}\,:\,K_0(\DG_\bbC^{\gm})\longrightarrow K^{\oplus}_0(\bbZ)\,,\qquad \{\cD(X)\}\longmapsto \{K_i^{\tp}(X)\}\,,\qquad i=0,1\,,
$$
be the composition of the homomorphisms given in formulas~\eqref{eq:DGKM} and~\eqref{eq:KtopKm}, where $X$ is a complex smooth projective variety.
\end{defi}

\begin{remark}
More generally, topological $K$-groups are defined by Blanc in~\cite{Bla} for any $\bbC$-linear \mbox{dg-category}. However, it seems that it is not known whether topological $K$-groups are finitely generated for an arbitrary saturated $\bbC$-linear dg-category, not necessarily geometric. Thus it is not clear whether the homomorphisms $\kappa_i^{\tp}$ factor through the natural homomorphism ${K_0(\DG^{\gm}_{\bbC})\to K_0(\DG_{\bbC}^{\sat})}$. This is the reason why we restrict ourselves to the Grothendieck group of geometric dg-categories.
\end{remark}

\section{Equivariant sheaves for non-effective actions}\label{sec:noneff}

\subsection{Two equivalences of categories}\label{subsec:twoequiv}

In this section, we study the following problem. Let
\begin{equation}\label{eq:exgroup}
1\longrightarrow N\longrightarrow G \stackrel{\pi}\longrightarrow H\longrightarrow 1
\end{equation}
be an exact sequence of finite groups and let $X$ be an $H$-variety. Our aim is to describe \mbox{$G$-equivariant} coherent sheaves on $X$ in terms of $H$-equivariant coherent sheaves on $X$ with an additional structure. Here $G$ acts on $X$ through the homomorphism $\pi$.

\medskip

Let $V$ be a finite-dimensional representation of $G$ and define an $G$-equivariant coherent sheaf on $X$
$$
\cV:=V\otimes_k\cO_X\,.
$$
Also, define a finite-dimensional \mbox{$H$-equivariant} algebra
$$
A:=\End_N(V)^{\rm op}
$$
and an $H$-equivariant coherent sheaf of $\cO_X$-algebras on $X$
$$
\cA:=A\otimes_k\cO_X\simeq \cE nd_{N}(\cV)^{\rm op}\,.
$$
Note that $A$ and $\cA$ are also $G$-equivariant algebras, where $G$ acts through the homomorphism~$\pi$. Furthermore, $V$ is naturally a $G$-equivariant right $A$-module and $\cV$ is naturally a $G$-equivariant coherent right $\cA$-module.

By the construction of $A$, we see that~$V$ is an $(N{-}A)$-bimodule and $\cV$ is an~\mbox{$(N{-}\cA)$-bimodule}, where the actions of~$N$ on~$V$ and~$\cV$ are obtained by taking the restriction of the $G$-equivariant structures. Therefore we have adjoint functors
\begin{equation}\label{eq:firstadjfuncts}
\cV\otimes_{\cA}-\;:\;\coh(\cA)\longrightarrow\coh^N(X)\,,\qquad \cH om_N(\cV,-)\;:\;\coh^N(X)\longrightarrow\coh(\cA)\,,
\end{equation}
that is, for any coherent $\cA$-module $\cM$ and any $N$-equivariant coherent sheaf $\cF$ on $X$, there is a functorial isomorphism
\begin{equation}\label{eq:firstadj}
{\rm Hom}_N(\cV\otimes_{\cA} \cM ,\cF)\stackrel{\sim}\longrightarrow {\rm Hom}_{\cA}\big(\cM,\cH om_N(\cV,\cF)\big)\,.
\end{equation}
Here, the actions of $N$ on $\cV\otimes_{\cA}\cM$ and of $\cA$ on $\cH om_N(\cV,\cF)$ are induced by their (commuting) actions on $\cV$.

Suppose, in addition, that $\cM$ is an $H$-equivariant coherent $\cA$-module and $\cF$ is a \mbox{$G$-equivariant} coherent sheaf on $X$. Then $\cM$ has a $G$-equivariant structure via the homomorphism $\pi$, and
${\cV\otimes_{\cA} \cM}$ is naturally a $G$-equivariant coherent sheaf on~$X$. Furthermore, $\cH om_N(\cV,\cF)$ is naturally an $H$-equivariant coherent $\cA$-module. Thus the adjoint functors in formula~\eqref{eq:firstadjfuncts} extend to the functors
$$
\cV\otimes_{\cA}-\;:\;\coh^H(\cA)\longrightarrow\coh^G(X)\,,\qquad \cH om_N(\cV,-)\;:\;\coh^G(X)\longrightarrow\coh^H(\cA)\,,
$$
which we denote similarly by abuse of notation. The following lemma claims that these functors are adjoint as well.

\begin{lemma}\label{lemma:adj}
For any $\cM$ in $\coh^H(\cA)$ and $\cF$ in $\coh^G(X)$, there is a functorial isomorphism
$$
{\rm Hom}_G(V\otimes_A \cM ,\cF)\stackrel{\sim}\longrightarrow {\rm Hom}_{\cA,H}\big(\cM,\cH om_N(\cV,\cF)\big)\,.
$$
\end{lemma}
\begin{proof}
We have a standard adjunction
\begin{equation}\label{eq:adjsimple}
{\rm Hom}(\cV\otimes_\cA \cM ,\cF)\stackrel{\sim}\longrightarrow {\rm Hom}_{\cA}\big(\cM,\cH om(\cV,\cF)\big)
\end{equation}
given by the right $\cA$-module $\cV$. All sheaves involved in the isomorphism~\eqref{eq:adjsimple} are given with compatible $G$-equivariant structures. It follows that $G$ acts naturally on both sides of the isomorphism~\eqref{eq:adjsimple} and this isomorphism commutes with the action of $G$. It remains to take $G$-invariants of both sides of the isomorphism~\eqref{eq:adjsimple} by first taking $N$-invariants and then taking $H$-invariants (note that, taking $N$-invariants only, we also obtain the adjunction~\eqref{eq:firstadj}).
\end{proof}

\medskip

Let us say that a representation $V$ of $G$ is an {\it $N$-generator} if the restriction of $V$ to $N$ contains all irreducible representations of $N$ as direct summands.

\begin{theo}\label{thm:equivcoh}
Let $V$ be a finite-dimensional representation of $G$ which is an $N$-generator. Put
$$
\cV:=V\otimes_k\cO_X\,,\qquad A:=\End_N(V)^{\rm op}\,,\qquad \cA:=A\otimes_k\cO_X\,.
$$
Then the adjoint functors $\cV\otimes_\cA-$ and $\cH om_N(\cV,-)$ define equivalences of categories
$$
\coh(\cA)\simeq \coh^N(X)\,,\qquad \coh^H(\cA)\simeq\coh^G(X)\,.
$$
These equivalences commute naturally with the forgetful functors ${\coh^H(\cA)\to\coh(\cA)}$ and~${{\coh^G(X)\to \coh^N(X)}}$.
\end{theo}
\begin{proof}
By the isomorphism~\eqref{eq:firstadj} and Lemma~\ref{lemma:adj}, it is enough to show that for any $\cM$ in~$\coh(\cA)$ (respectively, in $\coh^H(\cA)$) and $\cF$ in $\coh^N(X)$ (respectively, in $\coh^G(X)$), the natural morphisms
\begin{equation}\label{eq:comps}
\cM\longrightarrow \cH om_N(\cV,\cV\otimes_\cA \cM)\,,\qquad \cV\otimes_\cA\cH om_N(\cV,\cF)\longrightarrow \cF
\end{equation}
are isomorphisms. Clearly, it is enough to consider the case when $\cM$ is in $\coh(\cA)$, $\cF$ is in~$\coh^N(X)$, and $X$ is affine. In this case, the needed assertion is well-known, because $\cV$ is then a projective generator in $\coh^N(X)$. We provide details for the sake of completeness.

First note that $\cV$ is a projective object in the category $\coh^N(X)$, because the category~$\rep(N)$ is semisimple and $X$ is affine. Hence the functor ${\cH om_N(\cV,-)}$ is exact. It follows that both morphisms in~\eqref{eq:comps} are between right exact functors.

Furthermore, we claim that $\cA$ and $\cV$ are generators of the categories~$\coh(\cA)$ and~$\coh^N(X)$, respectively. In other words, for any $\cM$ in $\coh(\cA)$ and $\cF$ in $\coh^N(X)$, there are surjective morphisms
\begin{equation}\label{eq:surjmorph}
\cA^{\oplus n}\longrightarrow\cM\,,\qquad \cV^{\oplus n}\longrightarrow \cF
\end{equation}
in $\coh(\cA)$ and $\coh^N(X)$, respectively, for some natural number $n$. Indeed, the first morphism exists because $X$ is affine. To show the existence of the second morphism, decompose~$\cF$ into a direct sum of isotypic components
$$
\cF\simeq \bigoplus\limits_{W\in\Irr(N)}W\otimes_k\cH om_N(\cW,\cF)\,,
$$
where $\cW:=W\otimes_k\cO_X$. Then use again that $X$ is affine and that for any irreducible representation $W$ of $N$, there is a surjective morphism $V\to W$, because $V$ is an $N$-generator.

Finally, for $\cM=\cA$ and $\cF=\cV$, one checks directly that the morphisms~\eqref{eq:comps} are isomorphisms with the help of the isomorphism ${\cA\simeq\cE nd_N(\cV)^{\rm op}}$. Thus we finish the proof using right exactness of the functors in~\eqref{eq:comps} and surjective morphisms~\eqref{eq:surjmorph}.
\end{proof}

\begin{remark}\label{remark:arbvectbund}
More generally, let $\cV$ be a $G$-equivariant vector bundle on $X$ such that for any irreducible representation $W$ of $N$, the $W$-isotypic component of $\cV$ is non-zero. Then Theorem~\ref{thm:equivcoh} holds also for the $H$-equivariant coherent sheaf of $\cO_X$-algebras $\cA:=\cE nd_N(\cV)^{\rm op}$ (the proof remains the same).
\end{remark}

\subsection{Interpretation in terms of an equivariant Brauer group}\label{subsection:Brinterp}

In order to apply Theorem~\ref{thm:equivcoh}, it is useful to interpret the algebra $A$ and the coherent sheaf of $\cO_X$-algebras $\cA$ as Azumaya algebras.

First we define a commutative algebra
$$
Z:=Z(k[N])
$$
and an affine scheme
$$
\uIrr(N):=\Spec(Z)\,.
$$
Note that there is a decomposition into a product of fields over $k$
$$
Z\simeq \mbox{$\prod\limits_{W\in \Irr(N)}Z\big(\End(W)\big)$}\,.
$$
Hence the underlying set of the scheme $\uIrr(N)$ is canonically bijective with the set $\Irr(N)$. Analogous facts hold over any extension of the field $k$. In particular, over any algebraic closure~$\bar k$ of~$k$, we have an isomorphism $Z_{\bar k}\simeq \prod\limits_{W\in \Irr_{\bar k}(N)}\bar k$ and a bijection ${\uIrr(N)(\bar k)\simeq \Irr_{\bar k}(N)}$.

The group $H$ acts on the algebra $Z$ by conjugation: an element $h\in H$ sends $z\in Z$ to~$\tilde h z \tilde h^{-1}$, where $\tilde h\in G$ is any preimage of $h$ with respect to $\pi$. The action of~$H$ on the scheme $\uIrr(N)$ defined by the action of $H$ on $Z$ agrees with the action of $H$ on the set $\Irr(N)$ defined by conjugation of representations, that is, an element $h\in H$ sends the isomorphism class of an irreducible representation $\rho$ of $N$ to the isomorphism class of the representation  ${g\mapsto \rho(\tilde h^{-1}g\tilde h)}$, where~$g\in N$.

For any representation $U$ of $N$, we have a canonical central homomorphism of algebras ${Z\to \End_N(U)}$, that is, $\End_N(U)$ is naturally a $Z$-algebra. One checks directly that for any representation $V$ of~$G$, the homomorphism $Z\to\End_N(V)$ commutes with the action of $H$, that is, $\End_N(V)$ is an $H$-equivariant $Z$-algebra.

\medskip

%\begin{lemma}\label{lemma:enhanced}
%Let $V$ and $A$ be as in Theorem~\ref{thm:equivcoh}. Then for all $M$ and $M'$ in $\mods^H(A)$, the isomorphism (see Theorem~\ref{thm:equivcoh} applied with $X=\Spec(k)$)
%$$
%\Hom_A(M,M')\stackrel{\sim}\longrightarrow \Hom_N(V\otimes_A M,V\otimes_A V')
%$$
%is an isomorphism of $H$-equivariant $Z$-modules.
%\end{lemma}

Recall that the $H$-equivariant Brauer group $\Br^H(Y)$ of an $H$-variety $Y$ consists of classes of $H$-equivariant sheaves of Azumaya algebras on $Y$ modulo $H$-equivariant Morita equivalence over $Y$. Namely, given $H$-equivariant sheaves of Azumaya algebras $\cB$ and $\cB'$ on~$Y$, we have $\cB\sim \cB'$ if and only if there exists an $H$-equivariant coherent $(\cB'\otimes_{\cO_Y} \cB^{\rm op})$-module~$\cP$ such that
$$
\cP\otimes_{\cB}-\colon \coh^H(\cB)\stackrel{\sim}\longrightarrow\coh^H(\cB')
$$
is an equivalence of categories. Actually, it is equivalent to require that an $H$-equivariant coherent $(\cB'\otimes_{\cO_Y} \cB^{\rm op})$-module $\cP$ gives an equivalence of categories ${\cP\otimes_{\cB}-\colon \coh(\cB)\stackrel{\sim}\longrightarrow\coh(\cB')}$.

Given a class $\eta\in\Br^H(Y)$, we let~$\coh^H(\eta)$ denote the category $\coh^H(\cB)$, where $\cB$ represents~$\eta$. This is well-defined up to equivalence.

For an affine $H$-variety $Y=\Spec(R)$, we also use notation $\Br^H(R)$ for $\Br^H(Y)$ and, given an element $\eta\in\Br^H(R)$, we also denote $\coh^H(\eta)$ by~$\mods^H(\eta)$.

\begin{prop}\label{prop:Azum}
\hspace{0cm}
\begin{itemize}
\item[(i)]
Let $V$ be a representation of $G$ which is an $N$-generator. Then $A=\End_N(V)^{\rm op}$ is an \mbox{$H$-equivariant} Azumaya algebra over $Z$.
\item[(ii)]
Let $A'$ be an $H$-equivariant Azumaya algebra over $Z$. Then $A'$ is $H$-equivariantly Morita equivalent to $A$ as in part~$(i)$ if and only if there exists a representation $V'$ of $G$ such that $V'$ is an $N$-generator and there is an isomorphism of $H$-equivariant $Z$-algebras ${A'\simeq \End_N(V')^{\rm op}}$.
\end{itemize}
\end{prop}
\begin{proof}
$(i)$
We need to show that $A$ is an Azumaya algebra over $Z$. Since $Z$ is the center of~$k[N]$ and the category $\mods\big(k[N]\big)\simeq\rep(N)$ is semisimple, we see that $k[N]$ is an Azumaya algebra over $Z$. Since $V$ is an $N$-generator, $V$ is faithful as a $Z$-module. Hence the equality $\End_N(V)^{\rm op}=\End_{k[N]}(V)^{\rm op}$ implies that $A$ is an Azumaya algebra over $Z$.

%By construction, the actions of $Z$ on $V$ through $k[N]$ and $A$ coincide. Hence $V$ is a $(k[N]\otimes_Z A^{\rm op})$-module. By Theorem~\ref{thm:equivcoh} applied with %$X=\Spec(k)$, the $(k[N]\otimes_Z A^{\rm op})$-module~$V$ gives an equivalence of categories
%$$
%V\otimes_A-\;:\;\mods(A)\stackrel{\sim}\longrightarrow\rep(N)\simeq \mods\big(k[N]\big)\,.
%$$
%This defines an isomorphism of $Z$-algebras $Z(A)\simeq Z\big(k[N]\big)$, whence $A$ is a central $Z$-algebra. Finally, the category $\mods(A)\simeq\rep(N)$ is semisimple, %whence $A$ is an Azumaya algebra over its center $Z$.

$(ii)$
It follows from the definition of $H$-equivariant Morita equivalence over $Z$ that~${A'\sim A}$ if and only if there exists an $H$-equivariant $A$-module $Q$ such that $Q$ is faithful as a $Z$-module and there is an isomorphism of $H$-equivariant $Z$-algebras ${A'\simeq \End_A(Q)^{\rm op}}$. By Theorem~\ref{thm:equivcoh} applied with $X=\Spec(k)$, we have a $2$-commutative diagram
$$
\begin{CD}
\mods^H(A) @>V\otimes_A->> \rep(G) \\
@VVV   @VVV \\
\mods(A) @>V\otimes_A->> \rep(N)
\end{CD}
$$
with horizontal arrows being equivalences of categories. These equivalences also commute with the $Z$-linear structures on the categories. Hence it remains to put~${V':=V\otimes_A Q}$ with the tensor product of the actions of~$G$ on $V$ and on $Q$ through the homomorphism~$\pi$.
\end{proof}

By Proposition~\ref{prop:Azum}, the class of $\End_N(V)^{\rm op}$ in $\Br^H(Z)$ does not depend on the choice of a representation~$V$ of $G$ which is an $N$-generator. In other words, this class is well-defined by the exact sequence~\eqref{eq:exgroup}. Denote this class by
\begin{equation}\label{eq:alpha}
\theta\in\Br^H(Z)\,.
\end{equation}

Theorem~\ref{thm:equivcoh} has the following interpretation in terms of equivariant Brauer groups. As above, let $X$ be an $H$-variety. Consider the $H$-equivariant projection
$$
f\;:\; \uIrr(N)\times X\longrightarrow \uIrr(N)\,.
$$

\begin{theo}\label{cor:trivobstr}
\hspace{0cm}
\begin{itemize}
\item[(i)]
There is an equivalence of categories
$$
{\coh^G(X)\simeq\coh^H(f^*\theta)}\,.
$$
\item[(ii)]
Suppose that $\theta=0$ in $\Br^H(Z)$. Then for any $H$-variety $X$, there is an equivalence of categories
$$
\coh^G(X)\simeq\coh^H\big(\uIrr(N)\times X\big)\,.
$$
\end{itemize}
\end{theo}
\begin{proof}
$(i)$
Let $A$ be a representative of $\theta$ constructed as in Theorem~\ref{thm:equivcoh}. Since the $H$-equivariant projection $p\colon X\times\uIrr(N)\to X$ is a finite morphism, we have an equivalence of categories
$$
\coh^H(f^*A)\stackrel{\sim}\longrightarrow \coh^H(p_*f^*A)\,,\qquad \cM\longmapsto p_*\cM\,.
$$
Since $p_*f^*A=A\otimes_k\cO_X$, we finish the proof applying Theorem~\ref{thm:equivcoh}.

$(ii)$ This follows directly from part~$(i)$.
\end{proof}

\begin{remark}
Proposition~\ref{prop:Azum} and Theorem~\ref{cor:trivobstr} admit the following generalizations with almost the same proofs. Let $\cV$ be as in Remark~\ref{remark:arbvectbund}. Then $\cA:=\cE nd_N(\cV)^{\rm op}$ is naturally an $H$-equivariant sheaf of Azumaya algebras on $\uIrr(N)\times X$. Given an $H$-equivariant sheaf of Azumaya algebras~$\cA'$ on $\uIrr(N)\times X$, we have an $H$-equivariant Morita equivalence $\cA\sim\cA'$ if and only if there is a $G$-equivariant vector bundle $\cV'$ such that for
every irreducible representation $W$ of~$N$, the $W$-isotypic component of $\cV'$ is non-zero and there is an \mbox{$H$-equivariant} isomorphism $\cA'\simeq \cE nd_N(\cV')^{\rm op}$ of sheaves of algebras on $\uIrr(N)\times X$. It follows that the class of~$\cA$ in ${\Br^H\big(\uIrr(N)\times X\big)}$ equals $f^*\theta$. In particular, we have $f^*\theta=0$ if and only if there is a \mbox{$G$-equivariant} vector bundle $\cV$ on $X$ such that the natural map ${\cO_{\uIrr(N)\times X}\to \cE nd_N(\cV)}$ is an isomorphism. In this case, there is an equivalence of categories $\coh^G(X)\simeq\coh^H\big(\uIrr(N)\times X)$.
\end{remark}

\subsection{A criterion for vanishing of the class $\theta$}\label{subsec:critvan}

In view of Theorem~\ref{cor:trivobstr}$(ii)$, it is important to have a criterion for vanishing of~$\theta$. Such criterion is obtained from Proposition~\ref{prop:Azum}$(ii)$. First let us state the following simple lemma.

\begin{lemma}\label{lem:brtrivgr}
The following conditions are equivalent:
\begin{itemize}
\item[(i)]
the class of $k[N]$ in $\Br(Z)$ vanishes;
\item[(ii)]
the natural homomorphism of algebras $Z\to\End_N(U)$ is an isomorphism, where we put ${U:=\bigoplus\limits_{W\in\Irr(N)}W}$;
%\item[(iii)]
%For any $W\in\Irr(N)$, the algebra $\End_N(W)$ is commutative.
\item[(iii)]
for any $W\in\Irr(N)$, the representation $W_{\bar k}:=W\otimes_k \bar k$ of $N$ over any algebraic closure~$\bar k$ of $k$ is isomorphic to the direct sum of (some of the) irreducible representations of $N$ over~$\bar k$ taken with multiplicity one.
\end{itemize}
\end{lemma}
\begin{proof}
$(i)\Leftrightarrow(ii)$
The class of $k[N]$ in $\Br(Z)$ vanishes if and only if there is a $k[N]$-module~$U$ such that the natural homomorphism $Z\to \End_{k[N]}(U)=\End_N(U)$ is an isomorphism. Also note that if $Z\to \End_N(U)$ is an isomorphism, then $U$ is isomorphic to the direct sum ${\bigoplus\limits_{W\in\Irr(N)}W}$ (the converse implication does not hold for an arbitrary group $N$).

$(ii)\Leftrightarrow(iii)$
The map $Z\to \End_N(U)$ is an isomorphism if and only if the map ${Z_{\bar k}\to \End_N(U_{\bar k})}$ is an isomorphism. Since $\bar k$ is algebraically closed, this is equivalent to the existence of an isomorphism ${U_{\bar k}\simeq \bigoplus\limits_{E\in \Irr_{\bar k}(N)}E}$. In turn, this is equivalent to~$(iii)$.
\end{proof}

\begin{prop}\label{cor:split}
The following conditions are equivalent:
\begin{itemize}
\item[(i)]
we have $\theta=0$ in $\Br^H(Z)$;
\item[(ii)]
the group $N$ satisfies the equivalent conditions of Lemma~\ref{lem:brtrivgr} and the $N$-representation $\bigoplus\limits_{W\in\Irr(N)}W$ extends to a $G$-representation.
\end{itemize}
\end{prop}
\begin{proof}
By Proposition~\ref{prop:Azum}$(ii)$, we have $\theta=0$ if and only if there is a representation $V$ of~$G$ such that the natural map ${Z\to\End_N(V)}$ is an isomorphism. Note that the latter condition implies that $V$ is an $N$-generator and, moreover, that there is an isomorphism ${V|_N\simeq \bigoplus\limits_{W\in\Irr(N)}W}$. This finishes the proof.
\end{proof}

%\begin{remark}\label{remark:extact}
%The action of~$N$ on $\bigoplus\limits_{W\in\Irr(N)}W$ extends to an action of $G$ on it if and only if for any $H$-orbit $O\subset \Irr(N)$, the action of~$N$ on %$\bigoplus\limits_{W\in O}W$ extends to an action of $G$ on it. Indeed, one implication is obvious. For another implication, given $W\in\Irr(N)$, by $e_W\in Z$ denote the %projector on the $W$-isotypic component. Then for any $H$-orbit $S\subset\Irr(N)$, the projector $e_O=\sum\limits_{W\in O}e_W\in Z$ is invariant under the action of $H$, that %is, $e_O$ is in the center $Z\big(k[G]\big)$. Hence, applying $e_O$ to the representation $\bigoplus\limits_{W\in\Irr(N)}W$ of $G$, we obtain the action of $G$ on %$\bigoplus\limits_{W\in S}W$ that extends the action of~$N$.
%\end{remark}

\medskip

Here is an example of a group $N$ for which the equivalent conditions of Lemma~\ref{lem:brtrivgr} are not satisfied. In particular, by Proposition~\ref{cor:split}, one has $\theta\ne 0$ for any exact sequence~\eqref{eq:exgroup} with this $N$ (see also Subsection~\ref{sec:furthexam} below for less trivial examples with $\theta\ne 0$).

\begin{example}
Let $N$ be the group of quaternions. Then the algebra $k[N]$ is isomorphic to the product $k^{\times 4}\times Q$, where $Q$ is the standard quaternion algebra over $k$ with a basis~$1$,~$i$,~$j$,~$ij$ such that $i^2=j^2=-1$ and ${ij=-ji}$. In particular, there is an isomorphism $Z\simeq k^{\times 5}$ and the class of $k[N]$ in $\Br(Z)\simeq\Br(k)^{\oplus 5}$ is~$(0,0,0,0,[Q])$. Thus if $Q$ does not split over $k$, then this class is non-trivial.
\end{example}

%Note that $U$ as in Lemma~\ref{cor:split}$(i)$ is isomorphic to the direct sum of all irreducible %\mbox{$k$-representations} of $N$ (but this is not equivalent to the condition on $U$ in Lemma~\ref{cor:split}$(i)$).

%The condition of Lemma~\ref{cor:split}$(ii)$ says that the representation $\mbox{$\bigoplus\limits_{W\in \Irr_{\bar %k}(N)}W$}$ of $N$ over $\bar k$ is defined over $k$ and extends to a representation of~$G$ over $k$.

Here are examples when the conditions of Proposition~\ref{cor:split} are satisfied.

\begin{example}\label{ex:commut}
Suppose that $N$ is abelian. Then the equivalent conditions of Lemma~\ref{lem:brtrivgr} are satisfied for $N$, because in this case, the representation $U=\bigoplus\limits_{W\in\Irr(N)}W$ is isomorphic to~$k[N]$ and we have ${Z=k[N]\simeq \End_N\big(k[N]\big)}$. Let $H$ be any finite group acting on $N$ and let $G=N\rtimes H$ be the semidirect product. Then the action of~$H$ on $U\simeq k[N]$ gives an action of~$G$ on~$U$ that extends the action of~$N$. Hence, by Proposition~\ref{cor:split}, we have $\theta=0$ in $\Br^H(Z)$.
\end{example}

\begin{example}\label{exam:wreath}
Let $T$ be a finite group and $n$ be a natural number. The symmetric group $\Sigma_n$ acts by permutations on the group $T^{\times n}$ as follows:
${\sigma (t_1,\ldots,t_n):=(t_{\sigma ^{-1}(1)},\ldots,t_{\sigma ^{-1}(n)})}$, where $\sigma\in\Sigma_n$, $t_1,\ldots,t_n\in T$. The corresponding semidirect product group $G=T^{\times n}\rtimes\Sigma_n$ is called a {\it wreath product}. Thus
we have a split exact sequence of groups
\begin{equation}\label{eq:wreath}
1\longrightarrow T^{\times n}\longrightarrow G\longrightarrow \Sigma_n\longrightarrow 1\,.
\end{equation}
Note that
$$
\Irr(T^{\times n})=\big\{ W_{1}\otimes \ldots\otimes W_{n}\,\mid\, W_1,\ldots,W_n\in\Irr(T)\big\}\,,
$$
where an element $(t_1,\ldots,t_n)\in T^{\times n}$ sends a tensor ${w_1\otimes\ldots\otimes w_n\in W_1\otimes\ldots\otimes W_n}$ to~$t_1(w_1)\otimes\ldots\otimes t_n(w_n)$. The corresponding action of $\Sigma_n$ on $\Irr(T^{\times n})$ is given by the formula
$$
\sigma(W_1\otimes \ldots\otimes W_n)=W_{\sigma ^{-1}(1)}\otimes \ldots\otimes W_{\sigma ^{-1}(n)}\,,\qquad \sigma\in\Sigma_n\,,\,\,\,W_1,\ldots,W_n\in\Irr(T)\,.
$$
The point is that for each $\sigma\in\Sigma_n$, there is a linear map
$$
\sigma\, :\,W_1\otimes \ldots\otimes W_n\longrightarrow W_{\sigma ^{-1}(1)}\otimes \ldots\otimes W_{\sigma ^{-1}(n)}\,,\qquad w_1\otimes \ldots\otimes w_n\mapsto w_{\sigma^{-1}(1)}\otimes \ldots\otimes w_{\sigma^{-1}(n)}\,,
$$
with the property that for any element $t=(t_1,\ldots,t_n)\in T^{\times n}$, the following diagram is commutative:
$$
\begin{CD}
W_1\otimes \ldots\otimes W_n  @>{\sigma}>>  W_{\sigma^{-1}(1)}\otimes \ldots\otimes W_{\sigma^{-1}(n)}  \\
 @V t VV @V \sigma(t) VV \\
W_1\otimes \ldots\otimes W_n @>{\sigma}>> W_{\sigma^{-1}(1)}\otimes \ldots\otimes W_{\sigma^{-1}(n)}\,.
\end{CD}
$$
It follows that the group $\Sigma_n$ acts on the direct sum of all irreducible representations of $T^{\times n}$ and this direct sum has a natural structure of a representation of~$G$.

Suppose now that the equivalent conditions of Lemma~\ref{lem:brtrivgr} are satisfied for the group~$T$. Then they are also satisfied for $T^{\times n}$. Hence, by Proposition~\ref{cor:split}, we have $\theta=0$ in $\Br^{\Sigma_n}(Z)$, where $Z=Z(k[T^{\times n}])$ and $\theta$ is associated with the exact sequence~\eqref{eq:wreath}.
\end{example}

\medskip

The above statements and constructions admit the following generalization.

\begin{remark}\label{rem:S}
Let $S\subset \Irr(N)$ be an $H$-invariant subset. Then $S$ corresponds to an \mbox{$H$-invariant} subscheme $\underline{S}$ of $\uIrr(N)$ with an $H$-equivariant algebra of regular functions~$Z_S$. Clearly,~$Z_S$ is an $H$-invariant direct factor of the algebra $Z$. Let $\rep(N)_S$ denote the category of \mbox{$S$-isotypic} finite-dimensional representations of $N$, that is, whose all irreducible direct summands are from~$S$, and let~$\rep(G)_S$ denote the category of finite-dimensional representations of $G$ whose restriction to~$N$ is $S$-isotypic. Let $\theta_S\in \Br^H(Z_S)$ be the restriction of $\theta\in \Br^H(Z)$ from~$Z$ to~$Z_S$.

Then there are equivalences of categories
$$
\coh(\theta_S)\simeq\rep(N)_S\,,\qquad \coh^H(\theta_S)\simeq \rep(G)_S\,.
$$
This follows from Theorem~\ref{thm:equivcoh} applied with $X=\Spec(k)$ and from the definition of the $Z$-linear structure on the algebra $A$ in this theorem (see Subsection~\ref{subsection:Brinterp}). Also, an analog of Theorem~\ref{thm:equivcoh} for an arbitrary $H$-variety $X$ holds in the $S$-isotypic setting as well. Furthermore, one can show that $\theta_S$ is the class of the $Z_S$-algebra $\End_N(V)^{\rm op}$, where $V$ is any representation in $\rep(G)_S$ whose restriction to $N$ contains as direct summands all irreducible representations of~$N$ from~$S$.

Lemma~\ref{lem:brtrivgr} has a direct generalization with $Z$, $k[N]$, $\Br(Z)$, and $\Irr(N)$ being replaced by~$Z_S$, $Z_S\cdot k[N]$, $\Br(Z_S)$, and $S$, respectively. Proposition~\ref{cor:split} also has a direct generalization with $\theta$ and $\Br^H(Z)$ being replaced by $\theta_S$ and $\Br^H(Z_S)$, respectively.

For example, the trivial character $\chi_0$ of $N$ defines an $H$-invariant $k$-point $\chi_0$ on $\uIrr(N)$ and the corresponding element $\theta_{\{\chi_0\}}\in\Br^H(k)$ is trivial.
\end{remark}

\subsection{Further examples}\label{sec:furthexam}

Let us give several examples to the class $\theta$ (see formula~\eqref{eq:alpha}). We will not really use this in what follows; however, we decided to provide these examples for a better understanding of the class~$\theta$ and for possible further applications.

\medskip

First we describe in detail a particular example of an $H$-equivariant algebra $A$ as in Theorem~\ref{thm:equivcoh}. The regular representation $V=k[G]$ of $G$ is an $N$-generator. Indeed, a choice of a section of $\pi$ defines an isomorphism $k[G]\simeq k[N]^{\oplus |H|}$ of representations of $N$ and~$k[N]$ contains as direct summands all irreducible representations of $N$. Since ${\End_N\big(k[N]\big)^{\rm op}\simeq k[N]}$, the algebra $A=\End_N\big(k[G]\big)^{\rm op}$ is isomorphic to the matrix algebra $\Mat_{|H|}\big(k[N]\big)$ over~$k[N]$. In particular, the dimension of~$A$ is equal to~$|H|^2\cdot|N|$.

Define the actions of the group $G$ on the algebra of regular functions $\cO(G)={\rm Map}(G,k)$ by left translations
\begin{equation}\label{eq:lefttrans}
L_{g}\;:\;\cO(G)\longrightarrow \cO(G)\,,\qquad \psi\longmapsto\big(g'\mapsto \psi(g^{-1}g')\big)\,,
\end{equation}
and by right translations
\begin{equation}\label{eq:righttrans}
R_g\;:\;\cO(G)\longrightarrow \cO(G)\,,\qquad \psi\longmapsto\big(h'\mapsto \psi(g'g)\big)\,,
\end{equation}
where $g,g'\in G$ and $\psi\in\cO(G)$. Then there is a canonical isomorphism $k[G]^{\vee}\simeq \cO(G)$ of representations of $G$, where we consider the action of $G$ on $\cO(G)$ by left translations. Hence there is a canonical isomorphism of $H$-equivariant algebras ${\End_N\big(k[G]\big)^{\rm op}\simeq\End_N\big(\cO(G)\big)}$.

One has the following alternative description of $\End_N\big(\cO(G)\big)$. We use similar notation as in formulas~\eqref{eq:lefttrans} and~\eqref{eq:righttrans} for the action of the group $H$ by translations on the algebra of regular functions $\cO(H)$. Define the algebra
$\cO(H)*G$ to be the {\it skew group algebra} of $G$ with coefficients in~$\cO(H)$, where we consider the action of~$G$ on $\cO(H)$ by right translations. Explicitly, $\cO(H)*G$ is isomorphic to $\cO(H)\otimes_k k[G]$ as a $k$-vector space and multiplication in~$\cO(H)*G$ is defined by the formula
$$
(\varphi\otimes g)\cdot(\varphi'\otimes g')=(\varphi\cdot R_{\pi(g)}\varphi')\otimes gg'\,,\qquad \varphi,\varphi'\in \cO(H),\,\, g,g'\in G\,.
$$
Since the action of the group $H$ on $\cO(H)$ by left translations commutes with the action of~$G$ on $\cO(H)$ by right translations, the group $H$ acts on the algebra $\cO(H)*G$ by the formula
$$
h\,:\,\varphi\otimes g\longmapsto L_h\varphi\otimes g\,,\qquad h\in H,\,\,\varphi\in\cO(H),\,\, g\in G\,.
$$
The $k$-vector space~$\cO(G)$ has a natural structure of a module over $\cO(H)*G$, where an element $\varphi\otimes g\in \cO(H)*G$ acts on $\cO(G)$ as $(\cdot \pi^*\varphi)\circ R_g$. The action of $\cO(H)*G$ on~$\cO(G)$ commutes with the action of $N$ by left translations, which defines a homomorphism of algebras
$$
\xi\;:\;\cO(H)*G\longrightarrow \End_N\big(\cO(G)\big)\,.
$$

\begin{prop}\label{lemma:explskew}
The homomorphism $\xi$ is an isomorphism of \mbox{$H$-equivariant} algebras.
\end{prop}
\begin{proof}
The homomorphism $\xi$ commutes with the action of $H$, because for all $\varphi\in\cO(H)$, $g\in G$, $h\in H$, and any preimage $\tilde h\in G$ of $h$ with respect to $\pi$, there are equalities between linear maps from $\cO(G)$ to itself
$$
L_{\tilde h}\circ\xi(\varphi\otimes g)\circ L_{\tilde h}^{-1}=L_{\tilde h}\circ (\cdot\pi^*\varphi)\circ R_g\circ L_{\tilde h}^{-1}=\big(\cdot\pi^*(L_h\varphi)\big)\circ L_{\tilde h}\circ R_g\circ L_{\tilde h}^{-1}=
$$
$$
=\big(\cdot\pi^*(L_h\varphi)\big)\circ R_g\circ L_{\tilde h}\circ L_{\tilde h}^{-1}=\xi\big(h(\varphi\otimes g)\big)\,.
$$
Clearly, the dimension of $\cO(H)*G$ is~$|H|\cdot|G|$, which is equal to the dimension $|H|^2\cdot|N|$ of~$\End_N\big(\cO(G)\big)$. Thus in order to prove that~$\xi$ is an isomorphism, it is enough to show that~$\xi$ is injective, that is, that the action of $\cO(H)*G$ on~$\cO(G)$ is faithful.

For any non-zero element $a=\sum\limits_{g\in G}\varphi_g\otimes g\in \cO(H)*G$, there are elements $g_1,g_2\in G$ such that $(\pi^*\varphi_{g_2})(g_1)\ne 0$. Hence for the characteristic function $I_{g_1g_2}\in\cO(G)$ of the element~$g_1g_2\in G$, we have
$$
(aI_{g_1g_2})(g_1)=\mbox{$\sum\limits_{g\in G}(\pi^*\varphi_g)(g_1)\cdot I_{g_1g_2}(g_1g)$}=(\pi^*\varphi_{g_2})(g_1)\ne 0\,.
$$
This implies faithfulness of the action of $\cO(H)*G$ on $\cO(G)$.
\end{proof}

\medskip

Now let us consider the example when $G$ is the Heisenberg group over the ring $\bbZ/n\bbZ$, that is, when $G$ is the group of upper triangular $(3\times 3)$-matrices with units on the diagonal and with entries in $\Z/n\Z$. Let $N$ be the center of~$G$ and let $H=G/N$. Let us choose generators $x,y$ of~$H$, which fixes an isomorphism $H\simeq\Z/n\Z\times\Z/n\Z$. Also, $x$ and $y$ determine the generator~$[x,y]$ of $N$, which, in turn, fixes an isomorphism $N\simeq\Z/n\Z$.

Let $\zeta\in k$ be a primitive $n$-th root of unity. Consider the character $\chi\colon N\to k^*$, $[x,y]\mapsto \zeta$. Since $N$ is central, $\chi$ defines an $H$-invariant point on~$\uIrr(N)$ and one has the corresponding element $\theta_{\zeta}:=\theta_{\{\chi\}}\in\Br^H(k)$ (see Remark~\ref{rem:S}).

One has two coordinate projections from $H=\Z/n\Z\times\Z/n\Z$ to $\bbZ/n\bbZ$. Together with $\zeta$ they define two characters
$$
\psi\;:\; H\longrightarrow k^*\,,\qquad x\longmapsto \zeta\,,\qquad y\longmapsto 1\,,
$$
$$
\psi'\;:\; H\longrightarrow k^*\,,\qquad x\longmapsto 1\,\qquad y\longmapsto \zeta\,.
$$
Let $L$ and $L'$ be one-dimensional representations of~$H$ given by the characters $\psi$ and $\psi'$, respectively. Note that there are canonical isomorphisms ${L^{\otimes n}\simeq (L')^{\otimes n}\simeq k}$ of representations of $H$. Define an $H$-equivariant algebra
$$
C_{\zeta}:=\big(k\oplus L\oplus\ldots\oplus L^{\otimes (n-1)}\big)\otimes_k \big(k\oplus (L')\oplus\ldots\oplus (L')^{\otimes (n-1)}\big)\,,
$$
where we put $l\cdot l'=\zeta\, l'\cdot l$ for all $l\in L$, $l'\in L'$. We call this algebra a {\it cyclic algebra}.

\begin{lemma}\label{lem:Heis}
The class of $C_{\zeta^{-1}}$ in $\Br^H(k)$ equals $\theta_{\zeta}$.
\end{lemma}
\begin{proof}
Note that there is a unique irreducible representation $V_{\zeta}$ of $G$ whose restriction to~$N$ is \mbox{$\chi$-isotypic} (this is an analog of the Stone--von Neumann theorem). Namely, let $G'\simeq (\Z/n\Z)^{\oplus 2}$ be a normal subgroup such that $G'$ contains~$N$ and there is an isomorphism $G/G'\simeq \Z/n\Z$, and let a character $\chi'\colon G'\to k^*$ be an extension of $\chi$ from $N$ to $G'$. Then the restriction $\ind^G_{G'}(\chi')|_{G'}$ is isomorphic to the direct sum of all extensions of $\chi$ from $N$ to~$G'$ (one uses here that $\zeta$ is primitive). Using Frobenius reciprocity, one deduces that the representation $V_{\zeta}:=\ind^G_{G'}(\chi')$ is irreducible, does not depend on the choice of $\chi'$, and for any representation $V$ of $G$ whose restriction to $N$ is $\chi$-isotypic, there is a non-zero morphism from $V_\zeta$ to $V$ (see more detail, e.g., in a paper by Beloshapka and Gorchinskiy~\cite[Ex.\,4.10]{BG}).

Hence by Theorem~\ref{thm:equivcoh} applied to a point and Remark~\ref{rem:S}, we see that~$\theta_\zeta$ is the class of the $H$-equivariant algebra $\End_N(V_{\zeta})^{\rm op}$. The latter algebra equals $\End_k(V_{\zeta})^{\rm op}$, because the restriction of $V_{\zeta}$ to $N$ is $\chi$-isotypic.

Burnside theorem together with the above characterization of $V_{\zeta}$ imply that the natural homomorphism of $H$-equivariant algebras
$$
k[G]/I_{\zeta}\longrightarrow \End_k(V_{\zeta})
$$
is an isomorphism, where $I_{\zeta}$ is the two-sided ideal generated by the element $z-\zeta$.

Applying the anti-involution $g\mapsto g^{-1}$, we obtain an isomorphism of $H$-equivariant algebras
$$
(k[G]/I_{\zeta})^{\rm op}\stackrel{\sim}\longrightarrow k[G]/I_{\zeta^{-1}}\,.
$$
Finally, an explicit calculation implies an isomorphism of $H$-equivariant algebras ${k[G]/I_{\zeta^{-1}}\simeq C_{\zeta^{-1}}}$.
\end{proof}

\begin{example}\label{exam:Heis}
Take two elements $a,b\in K^*$ and consider the variety ${X:=\Spec\big(K[T,S]/(T^n-a,S^n-b)\big)}$. Then~$X$ is an \mbox{$H$-torsor} over $k$ and it defines a morphism $t\colon {\Spec(k)\to BH=[\Spec(k)/H]}$. We have a canonical isomorphism $\Br^H(k)\simeq \Br(BH)$ and Lemma~\ref{lem:Heis} implies that the pull-back $t^*(\theta_{\zeta^{-1}})\in\Br(k)$ is the class of the classical cyclic algebra $A_{\zeta^{-1}}(a,b)$ associated with $a$,~$b$, and~$\zeta^{-1}$. In paticular, by Theorem~\ref{thm:equivcoh} and Remark~\ref{rem:S}, there is an equivalence of categories ${\coh^G(X)\simeq \mods\big(A_{\zeta^{-1}}(a,b)\big)}$.
\end{example}

\medskip

Let us construct an exact sequence~\eqref{eq:exgroup} that splits but still has a non-trivial class~$\theta$. Let~$W$ be a finite-dimensional $k$-vector space and $N\subset \GL(W)$ be a finite subgroup such that the representation of $N$ in $W$ is irreducible over $\bar k$ and the quotient
$$
H:=N/(N\cap k^*)\subset \PGL(W)
$$
does not lift group-theoretically to $\GL(W)$. The group $H$ acts on $N$ by conjugation and the representation $W$ is isomorphic to its conjugates by elements of~$H$. Let $G$ be the semidirect product $N\rtimes H$.

By construction, the corresponding exact sequence~\eqref{eq:exgroup} splits. On the other hand, consider the restriction $\theta_{\{W\}}\in\Br^H(k)$ of $\theta\in\Br^H(Z)$ to the $H$-invariant point on $\uIrr(N)$ that corresponds to $W$. By Proposition~\ref{cor:split} and Remark~\ref{rem:S}, we have $\theta_{\{W\}}=0$ if and only if the $N$-representation $W$ extends to a $G$-representation. One checks that giving such a representation of $G$ is the same as giving a lifting of $H$ to $\GL(W)$. Thus $\theta_{\{W\}}\ne 0$ and henceforth we have $\theta\ne 0$.

Here is a concrete example of the situation as above.

\begin{example}\label{ex:splitnontriv}
Let $k$ be algebraically closed, $W$ have dimension two, $H$ be the group~$A_4$ embedded into $\PGL_2(k)$, and $N$ be the preimage of $H$ with respect to the homomorphism $\SL_2(k)\to \PSL_2(k)=\PGL_2(k)$. Then a faithful two-dimensional representation of $N$ is irreducible, because $N$ is not abelian. Also, $H$ does not lift to~$\GL_2(k)$, because $A_4$ does not have an irreducible two-dimensional representation.
\end{example}

\medskip

At the end of this section, let us characterize the class $\theta\in\Br^H(Z)$ in a more invariant way.
Let~$A$ be a representative of $\theta$ and let $\underline{\mods}^H(A)$ be the category of finite-dimensional \mbox{$H$-equivariant} modules over~$A$, considered as a category enhanced in \mbox{$H$-equivariant} \mbox{$Z$-modules}.
Note that the class~$\theta$ is uniquely determined by the enhanced equivalence class of $\underline{\mods}^H(A)$.

Now consider the enhanced category $\underline{\rep}(G)$ whose objects are finite-dimensional representations of $G$ and whose $H$-equivariant $Z$-linear morphism modules are given by~$\Hom_N(V,V')$, where $V$ and $V'$ are representations of $G$. It follows from Theorem~\ref{thm:equivcoh} applied with~$X$ being a point and from Proposition~\ref{prop:Azum} that~$\underline{\rep}(G)$ is enhanced equivalent to $\underline{\mods}^H(A)$ for an $H$-equivariant Azumaya algebra $A$ over $Z$ whose class is~$\theta$.

\section{Main results}\label{sec:mainres}

\subsection{Extended quotient and quotient stack: equal categorical measures}\label{subsec:equalmeas}

\begin{defi}\label{def:S0}
Let $\cS_0$ be the collection of all exact sequences of finite groups
$$
1\longrightarrow N\longrightarrow G\longrightarrow H\longrightarrow 1
$$
that satisfy the following conditions:
\begin{itemize}
\item[(i)]
all irreducible representations of $N$ over $\bar k$ are defined over $k$;
\item[(ii)]
the $N$-representation $\bigoplus\limits_{W\in\Irr(N)}W$ extends to a $G$-representation;
\item[(iii)]
there is an $H$-equivariant bijection of sets $C(N)\simeq \Irr(N)$.
\end{itemize}
\end{defi}

Let us make some comments on the conditions of Definition~\ref{def:S0}.

\begin{remark}\label{rmk:S0}
\hspace{0cm}
\begin{itemize}
\item[(i)]
Condition~(i) implies that there is an isomorphism of permutation representations of~$H$ over $k$
$$
\big(k\cdot C(N)\big)^{\vee}\simeq k\cdot \Irr(N)
$$
given by the evaluation of characters of irreducible representations of $N$ on conjugacy classes in~$N$. Since $k\cdot C(N)$ is a permutation representation of $H$, it is (canonically) self-dual, whence we have an isomorphism of representations of $H$
$$
k\cdot C(N)\simeq k\cdot \Irr(N)\,.
$$
Condition (iii) strengthens this.
\item[(ii)]
Explicitly, condition~(iii) is equivalent to the following: the actions of~$H$ on $C(N)$ and~$\Irr(N)$ have the same collections of stabilizers. \item[(iii)]
By Proposition~\ref{cor:split}, conditions~(i) and~(ii) imply that the class ${\theta\in \Br^H(Z)}$ (see formula~\eqref{eq:alpha}) vanishes, where $Z=Z\big(k[N]\big)$.
\item[(iv)]
Conversely, the equality $\theta=0$ together with condition~(iii) imply conditions (i) and~(ii). Indeed, it follows from Proposition~\ref{cor:split} that condition~(ii) is satisfied and that any irreducible representation of $N$ over $k$ has no multiple irreducible summands over $\bar k$ (see also Lemma~\ref{lem:brtrivgr}). On the other hand, by condition~(iii), the sets $\Irr(N)$ and $\Irr_{\bar k}(N)$ have the same cardinality, which is equal to the cardinality of $C(N)$. Therefore we obtain condition~(i).
\item[(v)]
The collection $\cS_0$ is closed under taking pull-backs of exact sequences along homomorphisms to the right terms and under taking direct products of exact sequences.
\end{itemize}
\end{remark}

Here are examples of exact sequences from $\cS_0$.

\begin{example}\label{examp:S0}
\hspace{0cm}
\begin{itemize}
\item[(i)]
The wreath product exact sequence~\eqref{eq:wreath} (see Example~\ref{exam:wreath}) belongs to $\cS_0$ provided that all irreducible representations of $T$ over $\bar k$ are defined over $k$. Indeed, condition~(i) of Definition~\ref{def:S0} is obviously satisfied, while condition~(ii) of Definition~\ref{def:S0} is satisfied by Example~\ref{exam:wreath}. Condition~(iii) of Definition~\ref{def:S0} is satisfied, because any bijection of sets $C(T)\simeq \Irr(T)$ gives a $\Sigma_n$-equivariant bijection between ${C(T^{\times n})\simeq C(T)^{\times n}}$ and~${\Irr(T^{\times n})\simeq \Irr(T)^{\times n}}$.
\item[(ii)]
Let $G$ be an abelian group such that all characters of $G$ over $\bar k$ are defined over $k$. Then any exact sequence with such $G$ in the middle belongs to $\cS_0$. Indeed, since the map
$$
\Hom(G,k^*)=\Hom(G,\bar k^*)\longrightarrow \Hom(N,\bar k^*)
$$
is surjective, we see that conditions~(i) and~(ii) of Definition~\ref{def:S0} are satisfied. Condition~(iii) of Definition~\ref{def:S0} is satisfied, because the action of $H$ on $\Irr(N)$ and $C(N)=N$ is trivial.
\end{itemize}
\end{example}

Note that conditions~(ii) and (iii) of Definition~\ref{def:S0} are not equivalent, as the following example shows.

\begin{example}\label{exam:S0}
Suppose that $k$ contains a primitive $n$-th root of unity.
\begin{itemize}
\item[(i)]
Consider the exact sequence associated with the Heisenberg group over the ring~$\Z/n\Z$ (see Subsection~\ref{sec:furthexam}). Then condition~(ii) of Definition~\ref{def:S0} is not satisfied by Lemma~\ref{lem:Heis} (see also Example~\ref{exam:Heis}). At the same time, condition~(iii) of Definition~\ref{def:S0} is satisfied, because $C(N)$ and $\Irr(N)$ are sets of the same cardinality $n$ with the trivial action of $H$.
\item[(ii)]
Let $N=(\Z/n\Z)^{\oplus 3}$ and let $H=(\Z/n\Z)^{\oplus 2}$ act on elements of $N$ considered as
column vectors by matrices of type
$$
\left(\begin{array}{ccc}
1 & a & b\\
0 & 1 & 0\\
0 & 0 & 1
\end{array}\right)\,,
$$
where $a,b\in\Z/n\Z$. Let $G=N\rtimes H$ be the semidirect product. By Example~\ref{ex:commut}, condition~(ii) of Definition~\ref{def:S0} is satisfied. One checks directly that the collection of stabilizers of the action of $H$ on $C(N)\simeq N$ consists of $n$ copies of $H$ and of~${n^3-n}$ subgroups $K\subset H$ such that each $H/K$ is a non-trivial cyclic group. On the other hand, the collection of stabilizers of the action of $H$ on ${\Irr(N)\simeq \Hom(N,\Z/n\Z)}$ consists of subgroups of type ${\Z/d\Z\times \Z/d\Z}$, where $d$ is a positive divisor of $n$. Thus by Remark~\ref{rmk:S0}(ii), condition~(iii) of Definition~\ref{def:S0} is not satisfied.
\end{itemize}
\end{example}

\medskip

Now we are ready to prove one of our main results. Recall that we have constructed in Proposition~\ref{theor:catmeas}$(ii)$ the categorical measure
$$
\mu\;:\;K_0(\Var_k^{\eq})\longrightarrow K_0(\DG^{\gm}_k)\,,
$$
from the Grothendieck ring of equivariant varieties $K_0(\Var^{\eq}_k)$ (see Definition~\ref{def:Groeq}(ii)) to the Grothendieck ring of geometric dg-categories~${K_0(\DG^{\gm}_k)}$ (see Definition~\ref{defin:Deltageom}).

\begin{theo}\label{theorem:main}
Let $G$ be a finite group and let $X$ be an $\cS_0$-adequate $G$-variety (see Definition~\ref{def:sadeq}(i)), that is, for any $\bar k$-point~$x$ of~$X$, the exact sequence
$$
1\longrightarrow N_x\longrightarrow G_x\longrightarrow H_x\longrightarrow 1
$$
belongs to $\cS_0$ (see Definition~\ref{def:S0}), where $N_x\subset G$ is the stabilizer of $x$ and $G_x\subset G$ is the normalizer of $N_x$ in $G$. Then the equality
$$
\mu^G(X)=\mu(X/^{\ex}G)
$$
holds in $K_0(\DG_k^{\gm})$ (see Definition~\ref{def:extquot}(ii) for the extended quotient $X/^{\ex}G$).
\end{theo}
\begin{proof}
By Proposition~\ref{lemma:techno}$(ii)$, we need to show that $\mu(R_{\cS_0})=0$ (see Definition~\ref{def:sadeq}(ii) for~$R_{\cS_0})$. For this, we need to check two conditions. Firstly, we need to show that $\mu$ vanishes on elements of type
\begin{equation}\label{eq:first2}
\{V\}^F-\{V/F\}\,,
\end{equation}
where $V$ is an $F$-variety with a free action of a finite group $F$. By Proposition~\ref{theor:catmeas} and its proof, the homomorphism $\mu$ equals the composition
$$
K_0(\Var^{\eq}_k)\stackrel{\alpha}\longrightarrow K_0(\DM_k^{\smp})\longrightarrow K_0(\DG^{\gm}_k)\,.
$$
By Theorem~\ref{theo:comparison}, $\alpha$ vanishes on elements from formula~\eqref{eq:first2}, so the same holds for~$\mu$.

Secondly, we need to show that $\mu$ vanishes on elements of type
\begin{equation}\label{eq:second2}
\{Y\}^{G_s}-\{C(N_{s})\times Y\}^{H_{s}}\,,
\end{equation}
where $s\in\cS_0$ and $Y$ is a smooth projective $H_s$-variety. By Proposition~\ref{cor:split}, conditions~(i) and~(ii) of Definition~\ref{def:S0} imply that~${\theta=0}$ in~$\Br^{H_s}(Z)$, where $Z:=Z\big( k[N_s]\big)$ and $\theta$ is associated with the exact sequence
$$
1\longrightarrow N_s\longrightarrow G_s\longrightarrow H_s\longrightarrow 1
$$
(cf. Remark~\ref{rmk:S0}(iii)). Hence, by Theorem~\ref{cor:trivobstr}$(ii)$, we have an equivalence of categories
$$
\coh^{G_s}(Y)\simeq\coh^{H_s}\big(\uIrr(N_s)\times Y\big)\,.
$$
By condition~(iii) of Definition~\ref{def:S0}, there is an $H$-equivariant isomorphism of $H$-varieties
$$
\uIrr(N_s)\times Y\simeq C(N_s)\times Y\,.
$$
This gives an equivalence of categories
$$
\coh^{G_s}(Y)\simeq\coh^{H_s}\big(C(N_s)\times Y\big)\,.
$$
Clearly, this implies the equality ${\mu^{G_s}(Y)=\mu^{H_s}\big(C(N_s)\times Y\big)}$ in $K_0(\DG_k^{\gm})$, that is, $\mu$ vanishes on elements from formula~\eqref{eq:second2}. This finishes the proof of the theorem.
\end{proof}

\begin{remark}
Theorem~\ref{theorem:main} generalizes to the case when $G$ is a linear algebraic group and~$X$ is a $G$-variety with a proper action of $G$. For this, one uses a $G$-equivariant version of the weak factorization theorem, see~\cite[\S\,0.3]{AKMW}, which implies that one has a Bittner presentation for the Grothendieck ring of $G$-varieties with proper action. Furthermore, one uses a generalization of results of Section~\ref{sec:noneff} when $N$ is a normal finite group subscheme of a linear algebraic group $G$ and one considers algebraic (finite-dimensional) representations of $G$. All the other arguments go through without any changes.
\end{remark}

\subsection{Motivic and categorical zeta-functions}\label{subsec:GS}

In~\cite{Kap}, Kapranov defined the {\it motivic zeta-function} of a variety $X$ as the power series
$$
Z_{mot}(X,t):=\sum _{n\geqslant 0}\{S^n(X)\}\,t^n\in K_0(\Var_k)[[t]]\,,
$$
in a formal variable $t$, where $S^n(X):=X^{\times n}/\Sigma_n$ is the $n$-th symmetric power of $X$. In~\cite{GS}, Galkin and Shinder defined the {\it categorical zeta-function} of $X$ as the power series
$$
Z_{cat}(X,t):=\sum_{n\geqslant 0}\mu^{\Sigma_n}(X^{\times n})\,t^n\in K_0(\DG^{\gm}_k)[[t]]\,.
$$
The categorical measure $\mu\colon K_0(\Var_k) \to K_0(\DG^{\gm}_k)$ extends termwise
to a ring homomorphism $K_0(\Var_k)[[t]] \to K_0(\DG^{\gm}_k)[[t]]$, which we also denote by $\mu$.

We prove the following comparison result, which was conjectured by Galkin and Shinder and also proved by them in {\it op.~cit.} for
varieties $X$ of dimension at most two.

\begin{theo}\label{thm:GS}
For any variety $X$, the equality
$$
Z_{cat}(X,t) = \mu\Big(\,\prod\limits_{i\geqslant 1}Z_{mot}(X,t^i)\Big)
$$
holds in the ring of power series $K_0(\DG_k^{\gm})[[t]]$.
\end{theo}

\medskip

The proof of the theorem is based on the following proposition. We are grateful to Galkin and Shinder for having communicated the proof of part~$(i)$ of this proposition to us.

\begin{prop}\label{prop:GS}
\hspace{0cm}
\begin{itemize}
\item[(i)]
For any variety $X$, the equality
$$
\prod_{i\geqslant 1}Z_{mot}(X,t^i)=\sum_{n\geqslant 0}\{X^{\times n}/^{\ex}\Sigma_n\}t^n
$$
holds in the ring of power series $K_0(\Var_k)[[t]]$.
\item[(ii)]
The $\Sigma_n$-variety $X^{\times n}$ is $\cS_0$-adequate for any $n\geqslant 0$ (see Definitions~\ref{def:sadeq}(i) and~\ref{def:S0}).
\end{itemize}
\end{prop}
\begin{proof}
$(i)$
Fix an integer~$n\geqslant 0$. Let us describe explicitly the extended quotient $X^{\times n}/^{\ex}\Sigma_n$. Recall that by formula~\eqref{eq:extquotexpl}, we have a decomposition into a disjoint union
$$
X^{\times n}/^{\ex}\Sigma_n=\coprod_{[g]\in C(\Sigma_n)}(X^{\times n})^g/Z(g)\,.
$$
The set $C(\Sigma_n)$ of conjugacy classes in $\Sigma_n$ is canonically bijective with the set $P_n$ of partitions of $n$, that is, with the set of tuples $(n_1,\ldots,n_r)$ of non-negative integers $n_i$ such that ${n=n_1\cdot 1+n_2\cdot 2+\ldots+n_r\cdot r}$. Explicitly, an element $\varpi=(n_1,\ldots,n_r)$ in $P_n$ corresponds to the conjugacy class of the element
$$
g_{\varpi}:=(\underbrace{\sigma_1,\ldots,\sigma_1}_{n_1},\underbrace{\sigma_2,\ldots,\sigma_2}_{n_2},\ldots,
\underbrace{\sigma_r,\ldots,\sigma_r}_{n_r})\in\Sigma_1^{n_1}\times\ldots\times \Sigma_r^{n_r}\subset \Sigma_n\,,
$$
where we consider the obvious inclusion map from $\Sigma_1^{n_1}\times\ldots\times \Sigma_r^{n_r}$ to $\Sigma_n$ and
the element $\sigma_i\in\Sigma_i$ is a cyclic permutation of length $i$ (in particular, $\sigma_1$ is the unique element of the trivial group $\Sigma_1$).

In order to describe $(X^{\times n})^{g_{\varpi}}$, define a closed embedding
$$
\delta_{\varpi}:=(\underbrace{\delta_1,\ldots,\delta_1}_{n_1},\underbrace{\delta_2,\ldots,\delta_2}_{n_2}\ldots,\underbrace{\delta_r,\ldots,\delta_r}_{n_r})\;:\;X^{\times n_1}\times X^{\times n_2}\times\ldots\times X^{\times n_r}\longrightarrow X^{\times n}\,,
$$
where $\delta_i\colon X\to X^{\times i}$ is the diagonal embedding. Let $M_{\varpi}\subset X^{\times n}$ be the closed subvariety which is the image of $\delta_{\varpi}$. Explicitly, $\bar k$-points of $M_{\varpi}\subset X^{\times n}$ are of the type
$$
(x_{11},\ldots,x_{1n_1},x_{21},x_{21},\ldots,x_{2n_2},x_{2n_2},\ldots, \underbrace{x_{r1},\ldots,x_{r1}}_r,\ldots,\underbrace{x_{rn_r},\ldots,x_{rn_r}}_r)\,.
$$
Then there is an equality ${(X^{\times n})^{g_{\varpi}}=M_{\varpi}}$.

Furthermore, one checks that the centralizer $Z(g_{\varpi})\subset \Sigma_n$ coincides with the subgroup of~$\Sigma_n$ generated by $g_{\varpi}$ and the subgroup
$$
H_{\varpi}:=\Sigma_{n_1}\times\Sigma_{n_2}\times\ldots\times\Sigma_{n_r}\subset \Sigma_n\,,
$$
where the embedding is defined naturally by the partition ${n=n_1\cdot 1+n_2\cdot 2+\ldots+n_r\cdot r}$. Thus, the group~$\Sigma_{n_1}$ permutes first $n_1$ elements $1,2,\ldots,n_1$, the group $\Sigma_{n_2}$ permutes $n_2$ ordered pairs $\{n_1+1,n_1+2\},\ldots,\{n_1+2n_2-1,n_1+2n_2\}$, etc. Moreover, the centralizer~$Z(g_\varpi)$ is isomorphic to the direct product ${\langle g_{\varpi}\rangle\times H_{\varpi}}$.

Altogether, this implies the equalities
$$
X^{\times n}/^{\ex}\Sigma_n=\coprod_{\varpi\in P_n}M_{\varpi}/H_{\varpi}=\coprod_{\varpi\in P_n}\big(S^{n_1}(X)\times\ldots\times S^{n_r}(X)\big)\,.
$$
Hence the equality
$$
\{X^{\times n}/^{\ex}\Sigma_n\}=\sum_{\varpi\in P_n}\{S^{n_1}(X)\}\cdot\ldots\cdot\{S^{n_r}(X)\}
$$
holds in $K_0(\Var_k)$. The right hand side is precisely the coefficient of $t^n$ in the series
$$
\prod_{i\geqslant 1}\Big(\sum_{m\geqslant 0}\{S^m(X)\}t^{im}\Big)\,.
$$

$(ii)$ We keep notation introduced in the proof of part~$(i)$. Given an element ${\varpi=(n_1,\ldots,n_r)}$ in $P_n$, define the subgroup
$$
N_{\varpi}:=\Sigma_1^{\times n_1}\times\Sigma_2^{\times n_2}\times\ldots\times \Sigma_r^{\times n_r}\subset\Sigma_n\,.
$$
For any point of $X^{\times n}$, its stabilizer is conjugate to $N_{\varpi}$ for some $\varpi\in P_n$. Besides, the normalizer~$G_{\varpi}\subset\Sigma_n$ of~$N_{\varpi}$ is generated by the subgroups $N_{\varpi}$ and $H_{\varpi}$ and is isomorphic to the semidirect product ${N_{\varpi}\rtimes H_{\varpi}}$, where $H_{\varpi}$ acts naturally on $N_{\varpi}$ by permuting factors. Hence we are reduced to showing
that the exact sequence
$$
1\longrightarrow N_{\varpi}\longrightarrow G_{\varpi}\longrightarrow H_{\varpi}\longrightarrow 1
$$
belongs to $\cS_0$. This follows from Example~\ref{exam:wreath} and Remark~\ref{rmk:S0}(v).
\end{proof}

Now we are ready to prove Theorem~\ref{thm:GS}.

\begin{proof}[Proof of Theorem~\ref{thm:GS}]
By Proposition~\ref{prop:GS}$(i)$, it is enough to show that the equality ${\mu^{\Sigma_n}(X^{\times n})=\mu(X^{\times n}/^{\ex}\Sigma_n)}$ holds in $K_0(\DG_k^{\gm})$. By Proposition~\ref{prop:GS}$(ii)$, the \mbox{$\Sigma_n$-variety}~$X^{\times n}$ is \mbox{$\cS_0$-adequate}. Thus the theorem follows from Theorem~\ref{theorem:main}.
\end{proof}

\subsection{Extended quotient and quotient stack: different categorical measures}\label{subsec:diffcatmeas}

In this section, we construct examples of $G$-varieties $X$ such that the categorical measures~$\mu^G(X)$ and~$\mu(X/^{\ex}G)$ are not equal in $K_0(\DG^{\gm}_k)$, showing that the conclusion of Theorem~\ref{theorem:main} does not hold always. All varieties below are over~$\C$. Essentially, we adopt a method of Ekedahl in the proof of~\cite[Prop.\,3.9]{Eked}, replacing singular cohomology by topological $K$-groups.

\medskip

We shall use the following measure of torsion in abelian groups.

\begin{defi}
\hspace{0cm}
\begin{itemize}
\item[(i)]
Given a finitely generated abelian group $A$, let
$$
\tau(A):=\log|A_{\rm tors}|\in\R\,,
$$
where $A_{\rm tors}$ denotes the torsion subgroup of $A$ (the logarithm is taken with respect to any fixed real number greater than~$1$).
\item[(ii)]
Given a graded finitely generated abelian group $A\simeq\bigoplus\limits_{i\in\Z}A^i$, let
$$
\tau(A^{\rm ev}):=\sum\limits_{i\in\Z}\tau(A^{2i})\,,\quad \tau(A^{\rm odd}):=\sum\limits_{i\in\Z}\tau (A^{2i+1})\,.
$$
\end{itemize}
\end{defi}

Clearly, $\tau$ defines a homomorphism of abelian groups
$$
\tau\;:\;K_0^{\oplus}(\Z)\longrightarrow \R\,,\qquad \{A\}\longmapsto \tau(A)\,.
$$
Here it is important that we consider the group $K_0^{\oplus}(\Z)$ with relations given by direct sums of abelian groups, not the group $K_0(\Z)$ (cf. Lemma~\ref{lemma:tau0}$(iii)$ below).

\medskip

Fix a natural number~$n$. Let $B$ be a complex algebraic variety. We let $H^i(B,-)$ denote cohomology with respect to the classical complex topology. Given an unramified Galois cover $X\to B$ with the Galois group $H$, one has canonical maps
\begin{equation}\label{eq:lambda}
\lambda\;:\;H^i(H,\Z/n\Z)\longrightarrow H^i(B,\Z/n\Z)\,,\qquad i\geqslant 0\,.
\end{equation}
These maps are defined as the edge maps given by the spectral sequence with the \mbox{$E_2$-term} $H^i\big(H,H^j(X,\Z/n\Z)\big)$ that converges to $H^{i+j}(B,\Z/n\Z)$. Alternatively, these maps are defined as the pull-back maps with respect to the corresponding map~${B\to K(H,1)}$ to the classifying space of~$H$. %We say that elements in the image of $\lambda$ {\it come from a cover}.

Furthermore, the exact sequence
\begin{equation}\label{eq:bock}
0\longrightarrow \Z\stackrel{n}\longrightarrow \Z\longrightarrow \Z/n\Z\longrightarrow 0
\end{equation}
%$$
%0\longrightarrow \Z/n\Z\longrightarrow \Z/n^2\Z\longrightarrow \Z/n\Z\longrightarrow 0
%$$
defines surjective coboundary maps
$$
\beta\;:\; H^i(B,\Z/n\Z)\longrightarrow H^{i+1}(B,\Z)_n\,,\qquad i\geqslant 0\,,
$$
%$$
%\qquad \beta_n\;:\; H^{i-1}(B,\Z/n\Z)\longrightarrow H^i(B,\Z/n\Z)\,,\qquad i>0\,,
%$$
called {\it Bockstein homomorphisms}. It is easy to see that these maps commute with multiplication by elements of~$H^j(B,\Z)$, $j\geqslant 0$.

%Let
%$$
%0\longrightarrow \Z/m\Z\longrightarrow G\longrightarrow H \longrightarrow 1
%$$
%be a central extension of finite groups. Suppose that $H$ acts freely on a complex algebraic variety $X$ and put $B:=X/H$. Then one has a canonical map $H^i(H,\Z/m\Z)\to %H^i(B,\Z/m\Z)$

\begin{theo}\label{theor:mainexamp}
Let
$$
0\longrightarrow \Z/n\Z\longrightarrow G\longrightarrow H \longrightarrow 1
$$
be a central extension of finite groups and let
$$
z\in H^2(H,\Z/n\Z)
$$
be its class. Let $H$ act freely on a complex smooth projective variety~$X$ and let~$G$ act on~$X$ through $H$. Put $B:=X/H$. Suppose that the following conditions are satisfied:
\begin{itemize}
\item[(i)]
the composition (see formulas~\eqref{eq:lambda} and~\eqref{eq:bock})
$$
H^2(H,\Z/n\Z)\stackrel{\lambda}\longrightarrow H^2(B,\Z/n\Z)\stackrel{\beta}\longrightarrow H^3(B,\Z)_n
$$
sends $z$ to a non-zero element $\beta(\lambda(z))$ in $H^3(B,\Z)_n$, where~$\lambda$ corresponds to the unramified Galois cover $X\to B$;
\item[(ii)]
there is an equality ${\tau\big(K_1^{\tp}(B)\big)=\tau\big(H^{\rm odd}(B,\Z)\big)}$.
\end{itemize}
Then there is an inequality of real numbers (see Definition~\ref{defi:kappa} for $\kappa^{\tp}_1$)
\begin{equation}\label{eq:ineqmain}
\tau\kappa^{\tp}_1\big(\mu^G(X)\big)<\tau\kappa^{\tp}_1\big(\mu(X/^{\ex}G)\big)\,.
\end{equation}
In particular, we have ${\mu^G(X)\ne \mu(X/^{\ex}G)}$ in $K_0(\DG^{\gm}_{\C})$.
\end{theo}

Proposition~\ref{prop:examp} below provides explicit examples of $B$ when the conditions of Theorem~\ref{theor:mainexamp} are satisfied.

\medskip

In what follows, we keep the notation and assumptions from Theorem~\ref{theor:mainexamp}. The proof of the theorem is based on several auxiliary statements. Proposition~\ref{prop:Bern} below allows to describe the left hand side of inequality~\eqref{eq:ineqmain} in terms of a Severi--Brauer scheme over $B$. (Actually, this statement is valid over any field $k$ that contains a primitive $n$-th root of unity.)

Define the character
$$
\chi\;:\;\Z/n\Z\longrightarrow\C^*\,,\qquad 1\longmapsto \exp(2\pi i/n)\,.
$$
We have the composition
\begin{equation}\label{eq:longcomp}
\nu\;:\;H^2(B,\Z/n\Z)\stackrel{\sim}\longrightarrow H^2_{\acute e t}(B,\Z/n\Z)\stackrel{\chi}\longrightarrow H^2_{\acute e t}(B,\cO_B^*)\stackrel{\sim}\longrightarrow \Br(B)\,.
\end{equation}
Here, for the first isomorphism see, e.g., Milne's book~\cite[Theor.\,III.3.12]{Milne} and for the last isomorphism see a theorem of Gabber and de Jong~\cite{dJ}.

\begin{prop}\label{prop:Bern}
Define an element in the Brauer group
$$
\alpha:=\nu(\lambda(z))\in \Br(B)\,.
$$
Let $E\to B$ be a Severi--Brauer scheme such that its class in $\Br(B)$ equals~$\alpha$. Then the equality
$$
d\cdot \mu^G(X)=n\cdot \mu(E)
$$
holds in~$K_0(\DG_{\C}^{\gm})$, where $E\to B$ is of relative dimension $d-1$.
\end{prop}
\begin{proof}
Let $\coh^{G,\chi}(B)$ denote the full subcategory in $\coh^G(B)$ that consists of all $G$-equivariant coherent sheaves $\cF$ on $X$ such that the subgroup $\Z/n\Z\subset G$ acts on $\cF$ by the character $\chi$. Clearly, there is an equivalences of categories
$$
\coh^G(X)\simeq \coh(B)\oplus \coh^{G,\chi}(B)\oplus\ldots\oplus\coh^{G,\chi^{n-1}}(B)\,.
$$
Let $\coh^{\alpha}(B)$ denote the category of $\alpha$-twisted sheaves. By definition of twisted sheaves, for each integer $i$, there is an equivalence of categories
$$
\coh^{G,\chi^i}(B)\simeq \coh^{i\alpha}(B)\,,
$$
see, e.g., Elagin's~\cite[Prop.\,1.9]{Ela0}. Whence there is an equivalence of categories
$$
\coh^G(X)\simeq \coh(B)\oplus \coh^{\alpha}(B)\oplus\ldots\oplus \coh^{(n-1)\alpha}(B)\,,
$$
and there is an equality
\begin{equation}\label{eq:onehand}
\mu^G(X)=\sum_{i=0}^{n-1}\big\{D^b\big(\coh^{i\alpha}(B)\big)\big\}
\end{equation}
in $K_0(\DG^{\gm}_{\C})$.

On the other hand, Bernardara~\cite[Theor.\,5.1]{Bern} proved that there is a semiorthogonal decomposition (an analogous result for algebraic $K$-groups was obtained previously by Quillen~\cite[Theor.\,4.1]{Qui})
$$
D^b(E)=\langle D^b(B),D^b\big(\coh^{-\alpha}(B)\big),\ldots,D^b\big(\coh^{-(d-1)\alpha}(B)\big)\rangle\,.
$$
Hence there is an equality
\begin{equation}\label{eq:otherhand}
\mu(E)=\sum_{i=0}^{d-1}\big\{D^b\big(\coh^{-i\alpha}(B)\big)\big\}
\end{equation}
in $K_0(\DG^{\gm}_{\C})$.

Let $r$ be the order of $\alpha$ in the Brauer group $\Br(B)$. Then $r$ divides both $d$ and $n$ and we finish the proof, comparing equalities~\eqref{eq:onehand} and~\eqref{eq:otherhand}.
\end{proof}

Explicitly, one can construct a Severi--Brauer scheme $E\to B$ that represents $\alpha$ as follows. Let $V$ be any representation of $G$ whose restriction to $\Z/n\Z$ is $\chi$-isotypic (for example, ${V=\ind_{\Z/n\Z}^G(\chi)}$). Then $H$ acts on~$\PP(V)$ and we put $E:=(\PP(V)\times X)/H$, where we consider the diagonal action of $H$. Note that a Severi--Brauer scheme over $B$ whose class equals $\alpha$ is not unique.

The number $d$ in Proposition~\ref{prop:Bern} is not necessarily equal to $n$: it may be that there is no a Severi--Brauer scheme $E\to B$ of relative dimension $n-1$ whose class in~$\Br(B)$ is~$\alpha$.

\medskip

We will also use the following simple properties of $\tau$.

\begin{lemma}\label{lemma:tau0}
Let $A$ be a finitely generated abelian group. The following holds true:
\begin{itemize}
\item[(i)]
for any subgroup $B\subset A$, we have $\tau(B)\leqslant \tau(A)$;
\item[(ii)]
let $f\colon A\to B$ be a surjective homomorphism such that $f_{\bbQ}\colon A_{\bbQ}\to B_{\bbQ}$ is an isomorphism of $\bbQ$-vector spaces; then $\tau(B)\leqslant \tau(A)$ and $f$ is an isomorphism if and only if ${\tau(B)=\tau(A)}$;
\item[(iii)]
given an exact sequence
$$
0\longrightarrow A'\longrightarrow A\longrightarrow A''\longrightarrow 0\,,
$$
we have $\tau(A)\leqslant \tau(A')+\tau(A'')$;
\item[(iv)]
given a finite decreasing filtration $F^iA$ on $A$, we have $\tau(A)\leqslant \sum_i \tau({\rm gr}^i_FA)$.
\end{itemize}
\end{lemma}
\begin{proof}
$(i)$ This is obvious.

$(ii)$ Since $f_{\bbQ}$ is an isomorphism, $\Ker(f)$ is a torsion group and we have an exact sequence
$$
0\longrightarrow \Ker(f)\longrightarrow A_{\rm tors}\longrightarrow B_{\rm tors}\longrightarrow 0\,.
$$
Therefore $\tau(B)\leqslant \tau(A)$ and one has the equality if and only if $\Ker(f)=0$.

$(iii)$ This is implied by the exact sequence
$$
0\longrightarrow A'_{\rm tors}\longrightarrow A_{\rm tors}\longrightarrow A''_{\rm tors}\,.
$$

$(iv)$ This follows directly from~$(iii)$.
\end{proof}

\begin{lemma}\label{lem:tau}
Let $E_r^{ij}$ be a spectral sequence that converges to groups $A^p$, where $i,j,p\in\Z$ and~$r\geqslant 2$. Suppose that all $E^{ij}_2$ and $A^p$ are finitely generated abelian groups and that the spectral sequence degenerates in the \mbox{$E_2$-term} after tensoring with $\bbQ$. Then the following holds true:
\begin{itemize}
\item[(i)]
For any $p$, we have $\tau(A^p)\leqslant \sum\limits_{i+j=p}\tau(E_2^{ij})$ (possibly, this sum equals infinity).
\item[(ii)]
If there is a non-zero differential that comes to the $p$-th diagonal, then
${\tau(A^p)<\sum\limits_{i+j=p}\tau(E_2^{ij})}$.
\item[(iii)]
If the spectral sequence degenerates in the $E_2$-term and for any $p$, the resulting
filtration on $A^p$ splits, then for any $p$, there is an equality ${\tau(A^p)=\sum\limits_{i+j=p}\tau(E_2^{ij})}$.
\end{itemize}
\end{lemma}
\begin{proof}
$(i)$ By Lemma~\ref{lemma:tau0}$(iv)$, we have $\tau(A^p)\leqslant \sum\limits_{i+j=p}\tau(E_{\infty}^{ij})$. For all $i,j$, the group $E_{\infty}^{ij}$ is a subquotient of the group $E_2^{ij}$ and by the assumption of the lemma, we have the equality ${(E_{\infty}^{ij})_{\bbQ}=(E_2^{ij})_{\bbQ}}$. Therefore, by Lemma~\ref{lemma:tau0}$(i),(ii)$, we have $\tau(E_{\infty}^{ij})\leqslant \tau(E_2^{ij})$.

$(ii)$ Suppose that there is a non-zero differential $d_r\colon E_r^{i-r,j+r-1}\to E_r^{ij}$ with $i+j=p$. Then the natural surjective homomorphism
$$
\Ker\big(d_r\colon E_{r}^{ij}\to E_r^{i+r,j-r+1}\big)\longrightarrow E^{ij}_{r+1}
$$
has a non-trivial kernel. Hence by Lemma~\ref{lemma:tau0}$(ii)$, we have a strict inequality ${\tau(E^{ij}_{r+1})<\tau(E^{ij}_{r})}$.

$(iii)$ This is obvious.
\end{proof}

Now recall that for a finite CW-complex $M$, there is an Atiyah--Hirzebruch spectral sequence, see~\cite[Sect.\,2]{AH0}, with $E_2^{ij}=H^i(M,\Z)$ for $j$ even and $E_2^{ij}=0$ for $j$ odd. The spectral sequence converges to~$K_{i+j}^{\tp}(M)$, is periodic with respect to the vertical shift by~$2$, and degenerates in the \mbox{$E_2$-term} after tensoring with~$\bbQ$. The resulting
filtration on topological $K$-groups is finite, exhaustive, and separated. In particular, Lemma~\ref{lem:tau}$(i)$ implies that
$$
\tau\big(K_0^{\tp}(M)\big)\leqslant \tau\big(H^{\rm ev}(M,\Z)\big)\,,\qquad
\tau\big(K_1^{\tp}(M)\big)\leqslant \tau\big(H^{\rm odd}(M,\Z)\big)\,.
$$

Recall that there is a canonical map
\begin{equation}\label{Psimap}
\Psi\;:\;\Br(B)\stackrel{\sim}\longrightarrow H^2_{\acute e t}(B,\cO_B^*)\longrightarrow H^2\big(B,(\cO_B^h)^*\big)\longrightarrow H^3(B,\Z)\,,
\end{equation}
where $\cO_B^{h}$ is the sheaf of holomorphic functions on $B$, $(\cO_B^{h})^*$ is the sheaf of invertible holomorphic functions, its cohomology is taken in the complex topology, and the last map is the boundary map associated with the exponential exact sequence
$$
0 \longrightarrow \Z \stackrel{2\pi i}\longrightarrow \cO_B^h \stackrel{\exp}\longrightarrow (\cO_B^h)^* \longrightarrow 1\,.
$$

%Let $E\to B$ be a Severi--Brauer scheme whose fiber is isomorphic \'etale locally to $\PP^{m-1}$ (this also holds in the complex topology). Then $E$ corresponds to an element %$[E]\in H^1_{\acute e t}(X,\PGL_n)$, see, e.g.,~\cite[Ch.\,III, \S\,4]{Milne}. The exact sequence of \'etale sheaves of groups
%$$
%1\longrightarrow \mu_m\longrightarrow \GL_m\longrightarrow \PGL_m\longrightarrow 1
%$$
%defines the coboundary map
%$$
%\partial\;:\; H^1_{\acute e t}(B,\PGL_m)\to H^2_{\acute e t}(B,\mu_m)\simeq H^2_{\acute e t}(B,\Z/m\Z)\simeq H^2(B,\Z/m\Z)\,.
%$$
%Here, the first isomorphism is defined by the chosen isomorphism $\mu_m\simeq\Z/m\Z$ and the second one follows from the comparison of cohomology with finite coefficients with %respect to the \'etale and complex topologies, see, e.g.,~\cite[Theor.\,III.3.12]{Milne}.
%Put
%$$
%\Psi(E):=\beta\big(\partial([E])\big)\in H^3(B,\Z)_m\,.
%$$

\begin{prop}\label{prop:bound}
Let $f\colon E\to B$ be a Severi--Brauer scheme of relative dimension $d-1$. Suppose that the following conditions are satisfied:
\begin{itemize}
\item[(i)]
the element $\Psi\big([E]\big)\in H^3(B,\Z)$ is non-zero, where $[E]\in \Br(B)$ is the class of $E$;
\item[(ii)]
we have $\tau\big(K_1^{\tp}(B)\big)=\tau\big(H^{\rm odd}(B,\Z)\big)$.
\end{itemize}
Then $\tau\big(K_1^{\tp}(E)\big)<d\cdot \tau\big(K_1^{\tp}(B)\big)$.
\end{prop}
\begin{proof}
We have that $R^2f_*\Z$ is a locally constant sheaf of rank one on $B$. The first Chern class of the relative canonical sheaf of $f$ defines a non-zero section of $R^2f_*\Z$, whence this sheaf is constant. By multiplicativity, we obtain that $R^jf_*\Z$ is a constant sheaf of rank one for every~${j=2r}$, $0\leqslant r\leqslant d$. Clearly, ${R^jf_*\Z=0}$ otherwise. We see that the Leray spectral sequence of $f$ takes the form $E_2^{ij}=H^i(B,\Z)$ when~${j=2r}$, $0\leqslant r\leqslant d$, and with $E_2^{ij}=0$ otherwise.

Since $E\times _B E$ is the projectivization of a rank $d$ vector bundle on $E$, we see that the pull-back map $H^3(B,\Z)\to H^3(E,\Z)$ sends~$\Psi\big([E]\big)$ to zero. Therefore the differential ${d_3\colon H^0(B,\Z)\to H^3(B,\Z)}$ is non-zero. (Actually, one can show that $d_3$ sends the element $1\in H^0(B,\Z)$ to $\Psi\big([E]\big)\in H^3(B,\Z)$; see, e.g., the book of Gille and Szamuely~\cite[Prop.\,5.4.4]{GiSa} for the arithmetic counterpart of this).

Recall that by Deligne~\cite{Del0}, the Leray spectral sequence degenerates after tensoring with~$\bbQ$. Hence the explicit description of the spectral sequence given above, together with Lemma~\ref{lem:tau}$(ii)$, implies that we have a strict inequality
$$
\tau\big(H^{\rm odd}(E,\Z)\big)<d\cdot \tau\big(H^{\rm odd}(B,\Z)\big)\,.
$$
Now the proposition follows from condition $(ii)$ and the inequality
${\tau\big(K_1^{\tp}(E)\big)\leqslant \tau\big(H^{\rm odd}(E,\Z)
\big)}$
coming from the Atiyah--Hirzebruch spectral sequence as described above.
\end{proof}

\medskip

Now we are ready to prove Theorem~\ref{theor:mainexamp}.

\begin{proof}[Proof of Theorem~\ref{theor:mainexamp}]
First observe that, by Example~\ref{ex:extquottriv}, we have $X/^{\ex}G=\coprod\limits_{i=1}^n B$. Hence there is an equality
\begin{equation}\label{eq:extquot}
\tau\kappa_1^{\tp}\big(\mu(X/^{\ex}G)\big)=n\cdot \tau\big(K_1^{\tp}(B)\big)\,.
\end{equation}
Let an element $\alpha\in \Br(B)$ and a Severi--Brauer scheme $E\to B$ be as in Proposition~\ref{prop:Bern}. By this proposition, the equality $d\cdot \mu^G(X)=n\cdot \mu(E)$ holds in $K_0(\DG_{\C}^{\gm})$. Hence the equality
\begin{equation}\label{eq:equivmeas}
\tau\kappa_1^{\tp}\big(\mu^G(X)\big)=\frac{n}{d}\cdot\tau\big(K_1^{\tp}(E)\big)
\end{equation}
holds in $\R$.

Furthermore, there is a commutative diagram
$$
\begin{CD}
0 @>>> \Z @>n>> \Z  @>>> \Z/n\Z @>>> 0 \\
@. @VV=V   @VV\frac{2\pi i}{n}V @VV\chi V \\
0 @>>> \Z @> 2\pi i>> \cO_B^h  @>\exp>> (\cO_B^h)^* @>>> 1
\end{CD}
$$
of sheaves on $B$ in the complex topology. This yields that the Bockstein homomorphism ${\beta\colon H^2(B,\Z/n\Z)\to H^3(B,\Z)_n}$ is equal to the composition $\Psi\circ\nu$ (see formula~\eqref{eq:longcomp} for $\nu$). Therefore there are equalities
$$
\Psi\big([E]\big)=\Psi(\alpha)=\Psi\big(\nu(\lambda(z))\big)=\beta(\lambda(z))
$$
in $H^3(B,\Z)$ (see formula~\eqref{eq:lambda} for~$\lambda$). Hence, by conditions~$(i)$ and~$(ii)$ of the theorem and by Proposition~\ref{prop:bound}, we have a strict inequality
$$
\tau \big(K_1^{\tp}(E)\big)<d\cdot \tau\big(K_1^{\tp}(B)\big)
$$
of real numbers. Combining this with equations~\eqref{eq:extquot} and~\eqref{eq:equivmeas}, we finish the proof.
\end{proof}

\medskip

Let us construct examples for which the conditions of Theorem~\ref{theor:mainexamp} are satisfied. With this aim, we further analyze the Atiyah--Hirzebruch spectral sequence. As above, let $M$ be a finite \mbox{CW-complex}. The homomorphism ${K_0^{\tp}(M)\to H^0(M,\Z)}$ that arises from the Atiyah--Hirzebruch spectral sequence coincides with the rank of vector bundles and has a splitting defined by trivial vector bundles of a given rank. Denote the kernel of this homomorphism by~$\widetilde{K}^{\tp}_0(M)$. Thus we have a canonical isomorphism
$$
K^{\tp}_0(M)\simeq H^0(M,\Z)\oplus \widetilde{K}^{\tp}_0(M)\,.
$$
Recall that $K_1^{\tp}(M)$ is the group $[M,U(\infty)]$ of homotopy classes of continuous maps from~$M$ to the infinite-dimensional unitary group $U(\infty)$. The homomorphism ${K_1^{\tp}(M)\to H^1(M,\Z)}$ that arises from the Atiyah--Hirzebruch spectral sequence coincides with the homomorphism induced by the determinant $\det\colon U(\infty)\to U(1)$ and the natural isomorphism ${[M,U(1)]\simeq H^1(M,\Z)}$ (see, e.g.,~\cite[\S\,5]{Gor}). Thus the homomorphism ${K_1^{\tp}(M)\to H^1(M,\Z)}$ also has a splitting defined by any splitting $\iota\colon U(1)\to U(\infty)$ of the determinant (the splitting $H^1(M,\Z)\to K_1^{\tp}(M)$ does not depend on the choice of $\iota$). Denote the kernel is this homomorphism by $\widetilde{K}^{\tp}_1(M)$. Thus we have a canonical isomorphism
$$
K^{\tp}_1(M)\simeq H^1(M,\Z)\oplus \widetilde{K}^{\tp}_1(M)\,.
$$
In particular, we obtain the vanishing of all differentials in the Atiyah--Hirzebruch spectral sequence that come out of $H^0(M,\Z)$ and $H^1(M,\Z)$.

\begin{lemma}\label{lem:prodAH}
Let $M$ and $M'$ be two finite CW-complexes. Suppose that the Atiyah--Hirzebruch spectral sequence for $M$ degenerates in the $E_2$-term and the filtrations on~$K_0^{\tp}(M)$ and~$K_1^{\tp}(M)$ split. Suppose that all cohomology groups $H^i(M',\Z)$ are torsion-free. Then the Atiyah--Hirzebruch spectral sequence for $M\times M'$ degenerates in the $E_2$-term as well and the filtrations on $K_0^{\tp}(M\times M')$ and $K_1^{\tp}(M\times M')$ also split.
\end{lemma}
\begin{proof}
Since the groups $H^i(M',\Z)$ are torsion-free, the Atiyah--Hirzebruch spectral sequence for $M'$ degenerates in the $E_2$-term and the filtrations on $K_0^{\tp}(M')$ and $K_1^{\tp}(M')$ split. Also, the absence of torsion in the cohomology of $M'$ implies, by the K\"unneth formula, that we have isomorphisms
$$
H^p(M\times M')\simeq\bigoplus_{i+j=p}H^i(M,\Z)\otimes H^j(M',\Z)\,,\qquad p\geqslant 0\,.
$$
Recall that the Atiyah--Hirzebruch spectral sequence is multiplicative in the natural sense, in particular, the differentials satisfy the graded Leibniz rule, see Remark 4 after Theorem 15.27 in Switzer's book~\cite{Switz} and also~\cite[Theor.\,2.6]{AH0}. Therefore the Atiyah--Hirzubruch spectral sequence for $M\times M'$ degenerates in the $E_2$-term.

Consider the natural homomorphisms of filtered abelian groups
$$
\varphi_0\;:\;\big(K_0^{\tp}(M)\otimes K_0^{\tp}(M')\big)\oplus \big(K_1^{\tp}(M)\otimes K_1^{\tp}(M')\big)\longrightarrow K_0^{\tp}(M\times M')\,,
$$
$$
\varphi_1\;:\;\big(K_0^{\tp}(M)\otimes K_1^{\tp}(M')\big)\oplus \big(K_1^{\tp}(M)\otimes K_0^{\tp}(M')\big)\longrightarrow K_1^{\tp}(M\times M')\,.
$$
Since the filtrations on the topological $K$-groups of $M$ and $M'$ split, we see that the filtrations on the sources of $\varphi_0$ and $\varphi_1$ split as well and the adjoint quotients form the graded abelian groups
$$
\big(H^{\rm even}(M)\otimes H^{\rm even}(M')\big)\oplus \big(H^{\rm odd}(M)\otimes H^{\rm odd}(M')\big)\,,
$$
$$
\big(H^{\rm even}(M)\otimes H^{\rm odd}(M')\big)\oplus \big(H^{\rm odd}(M)\otimes H^{\rm even}(M')\big)\,,
$$
respectively. On the other hand, since the Atiyah--Hirzebruch spectral sequence for~${M\times M'}$ degenerates in the $E_2$-term, we obtain that the homomorphisms~$\varphi_0$ and~$\varphi_1$ induce isomorphisms on the adjoint quotients. This implies that~$\varphi_0$ and~$\varphi_1$ are isomorphisms of filtered abelian groups, because the filtrations are finite, exhaustive, and separated. We see that the filtrations on the topological $K$-groups of $M\times M'$ split.
\end{proof}

We shall apply Lemma~\ref{lem:prodAH} for a product of a curve and a surface. Obviously, for a complex smooth projective curve $C$, we have the isomorphisms
$$
K_0^{\tp}(C)\simeq H^0(C,\Z)\oplus H^2(C,\Z)\simeq \Z^{\oplus 2}\,,\qquad K_1^{\tp}(C)\simeq H^1(C,\Z)\simeq
\Z^{\oplus 2g}\,
$$
where $g$ is the genus of $C$.

The proof of the following lemma essentially copies the proof given by Galkin, Katzarkov, Mellit, and Shinder in~\cite[Lem.\,2.2]{GKMS}, where the authors compare torsion in the Picard group and in the algebraic $K_0$-group for arbitrary algebraic varieties.

\begin{lemma}\label{lem:surface}
For a complex smooth projective surface $S$, the Atiyah--Hirzebruch spectral sequence degenerates in the $E_2$-term, we have $\widetilde{K}^{\tp}_1(S)\simeq H^3(S,\Z)$, and the exact sequence
\begin{equation}\label{eq:Galkin}
0\longrightarrow H^4(S,\Z)\longrightarrow \widetilde{K}_0^{\tp}(S)\longrightarrow H^2(S,\Z)\longrightarrow 0
\end{equation}
splits.
\end{lemma}
\begin{proof}
The only non-trivial assertion is the splitting of the exact sequence~\eqref{eq:Galkin}.
Let ${x\in H^2(S,\Z)}$ be an element such that $lx=0$ for a natural number $l$. It is well-known that~${H^2(S,\Z)}$ is canonically bijective with the set of isomorphism classes of topological complex line bundles on $S$. Let $L\to S$ be the line bundle that corresponds to $x$ and take the element ${y:=\{L\}-1\in \widetilde{K}_0^{\tp}(S)}$. Clearly, $y$ is sent to $x$ under the homomorphism $\widetilde{K}_0^{\tp}(S)\to H^2(S,\Z)$, which is the first Chern class.

Since $lx=0$, the line bundle $L^{\otimes l}$ is trivial (actually, this implies that $L$ is algebraic), whence $(1+y)^l=1$. The filtration on topological $K$-groups is multiplicative, $y$ belongs to the first term $\widetilde{K}_0^{\tp}(S)\subset K_0^{\tp}(S)$ of this filtration, and $S$ is a surface. Therefore, we have $y^3=0$. Hence we obtain the equality
\begin{equation}\label{eq:Galkintrick}
ly+{l \choose 2}y^2=0\,.
\end{equation}
Multiplying the left hand side of~\eqref{eq:Galkintrick} by $y$, we get $ly^2=0$ in $H^4(S,\Z)\subset K_0^{\tp}(S)$. Since the group $H^4(S,\Z)\simeq\Z$ is torsion-free, we see that $y^2=0$. Thus equality~\eqref{eq:Galkintrick} implies that~$ly=0$.

We conclude that the homomorphism
$$
\widetilde{K}_0^{\tp}(S)_{\rm tors}\longrightarrow H^2(S,\Z)_{\rm tors}
$$
is surjective. Again using that $H^4(S,\Z)$ is torsion-free, we obtain that the latter homomorphism is an isomorphism. This implies that the exact sequence~\eqref{eq:Galkin} splits.
\end{proof}

Now we are ready to construct an example to Theorem~\ref{theor:mainexamp}. Let $G$ be the Heisenberg group over the ring $\bbZ/n\bbZ$ (cf. Lemma~\ref{lem:Heis}), let $N$ be the center of~$G$, and put~${H:=G/N}$. Choose generators $x,y$ of $H$, which fixes an isomorphism $H\simeq\Z/n\Z\times\Z/n\Z$, a generator~$[x,y]$ of $N$, and hence an isomorphism $N\simeq\Z/n\Z$. We have an exact sequence
$$
0\longrightarrow \Z/n\Z\longrightarrow G\longrightarrow H \longrightarrow 1
$$
and let $z\in H^2(H,\Z/n\Z)$ be the corresponding class.

Let $S$ be a complex smooth projective surface such that there is a non-zero torsion element $s\in H^2(S,\Z)_n$. For example, one can take for $S$ an Enriques surface with $n=2$ or a surface of general type with $h^{1,0}=h^{2,0}=0$ and with non-trivial torsion in the Picard group.

There exists an element $t\in H^1(S,\Z/n\Z)$ with $\beta(t)=s$. Let $T\to S$ be an $n$-cyclic cover of~$S$ that corresponds to $t$. Let $C$ be a complex smooth projective curve of positive genus, and let $D\to C$ be any non-trivial $n$-cyclic cover. We have a free action of $H$ on the variety $X:=T\times D$ given by the action of the first
factor on $T$ and of the second factor on $D$.
\begin{prop}\label{prop:examp}
The class $z\in H^2(H,\Z/n\Z)$ and the action of $H$ on $X$ as above satisfy the conditions of Theorem~\ref{theor:mainexamp}. Thus, we have $\mu^G(X)\ne \mu(X/^{\ex}G)$ in $K_0(\DG_k^{\gm})$.
\end{prop}
\begin{proof}
Clearly, the quotient $B=X/G$ is isomorphic to $S\times C$. The class $z\in H^2(H,\Z/n\Z)$ is equal to the cup-product of the classes in $H^1(H,\Z/n\Z)$ given by two coordinate projections from $H$ to $\Z/n\Z$. This implies the equality
$$
\lambda(z)=p_S^*(t)\cdot p_C^*(c)\in H^2(B,\Z/n\Z)\,,
$$
where the class $c\in H^1(C,\Z/n\Z)$ corresponds to the $n$-cyclic cover $D\to C$.

Let $\tilde{c}\in H^1(C,\Z)$ be any lift of $c$. We obtain the equalities in $H^3(B,\Z)$
$$
\beta(\lambda(z))=\beta\big(p_S^*(t)\cdot p_C^*(c)\big)=\beta\big(p_S^*(t)\cdot p_C^*(\tilde{c})\big)=
$$
$$
=\beta(p_S^*(t))\cdot p_C^*(\tilde{c})=p_S^*(\beta(t))\cdot p_C^*(\tilde{c})=p_S^*(s)\cdot p_C^*(\tilde{c})\,.
$$
By our assumptions, $s\ne 0$ in $H^2(S,\Z)_n$. Also, the cover $D\to C$ is non-trivial, whence $c\ne 0$ and $\tilde{c}$ is not $n$-divisible in $H^1(C,\Z)$. Hence, by the K\"unneth formula, the element $p_S^*(s)\cdot p_C^*(\tilde{c})$ is non-zero as it corresponds to the non-zero element
$$
s\otimes \tilde{c}\in H^2(S,\Z)\otimes H^1(C,\Z)\subset H^3(S\times C,\Z)\,.
$$
Thus we see that condition~$(i)$ of Theorem~\ref{theor:mainexamp} is satisfied.

Combining Lemma~\ref{lem:prodAH} and Lemma~\ref{lem:surface}, we see that the Atiyah--Hirzebruch spectral sequence for $B$ degenerates in the $E_2$-term and the filtrations on the topological $K$-groups of~$B$ split. Hence, by Lemma~\ref{lem:tau}$(iii)$, we have $\tau\big(K_1^{\tp}(B)\big)=\tau\big(H^{\rm odd}(B,\Z)\big)$ (this holds for $K_0^{\tp}$ and even cohomology as well), that is, condition~$(ii)$ of Theorem~\ref{theor:mainexamp} is satisfied as well.
\end{proof}

\begin{remark}\label{rmk:mainexamp}
Proposition~\ref{prop:examp} generalizes trivially to the case of a product of two complex smooth projective varieties $V$ and $V'$ such that the Atiyah--Hirzebruch spectral sequence for~$V$ degenerates in the $E_2$-term, the filtrations on $K_0^{\tp}(V)$ and $K_1^{\tp}(V)$ split, all cohomology groups~$H^i(V',\Z)$ are torsion-free, we have $H^2(V,\Z)_n\ne 0$, and there is a non-trivial $n$-cyclic cover of~$V'$.
\end{remark}

\subsection{On a conjecture by Polishchuk and Van den Bergh}\label{subsec:PVdB}

Polishchuk and Van den Bergh~\cite[Conj.\,A]{PV} made the following conjecture.

\begin{conj}\label{conj:pvdb}
Let $G$ be a finite group and $X$ be a smooth quasi-projective $G$-variety such that $G$ acts on $X$ effectively. Suppose that for every
$g\in G$, the variety $X^g/Z(g)$ is smooth. Then there exists a semiorthogonal decomposition of $D^G(X)$ whose components (up to order) form the collection of triangulated categories $D^b\big(X^g/Z(g)\big)$ parameterized by conjugacy classes $[g]\in C(G)$, where~$g\in G$ is a representative of a conjugacy class $[g]$.
\end{conj}

Recall that the disjoint union of varieties $X^g/Z(g)$, $[g]\in C(G)$, is isomorphic to the extended quotient~$X/^{\ex}G$ (see formula~\eqref{eq:extquotexpl}). In particular, if $X$ as in Conjecture~\ref{conj:pvdb} is also projective, then the conclusion of the conjecture implies the equality $\mu^G(X)=\mu(X/^{\ex}G)$ in $K_0(\DG_k^{\gm})$ (the conjecture predicts a much finer phenomena).

\medskip

Theorem~\ref{theor:mainexamp} and Proposition~\ref{prop:examp} show that a version of Conjecture~\ref{conj:pvdb} without effectiveness requirement is not true. The following result provides a simpler example when~$X$ is a curve, which also shows that the condition on the action to be effective is indeed necessary. Nevertheless, we do not know whether the inequality $\mu^G(X)\ne\mu(X/^{\ex}G)$ holds in this case.

\begin{theo}\label{thm:PolvdB}
Let $n$ be a natural number and let the groups $N\simeq(\Z/n\Z)^{\oplus 3}$, $G$, and~${H\simeq(\Z/n\Z)^{\oplus 2}}$ be as in Example~\ref{exam:S0}(ii). Let $E$ be an elliptic curve without complex multiplication over $k$, that is, such that $\Z\simeq\End_k(E)$. Suppose that all points in the $n$-torsion subgroup $E_n\subset E$ are $k$-rational
(and therefore that $k$ contains a primitive $n$-th root of unity).
Choose an isomorphism $H\simeq E_n$ and let $G$ act on $E$ by translation through~$H$. Then there is no semiorthogonal decomposition of $D^G(E)$ whose components (up to order) form the collection of triangulated categories $D^b\big(E^g/Z(g)\big)$, ${[g]\in C(G)}$.
\end{theo}

In the proof of Theorem~\ref{thm:PolvdB} we use the following lemmas.

\begin{lemma}\label{lem:explfree}
Given an exact sequence of finite groups
$$
1\longrightarrow N\longrightarrow G\longrightarrow H\longrightarrow 1\,,
$$
let $G$ act on a variety $X$ through a free action of $H$ on $X$. Suppose that the exact sequence above satisfies conditions~(i) and (ii) of Definition~\ref{def:S0}. Then there is an equivalence of categories
$$
\coh^G(X)\simeq \bigoplus_{[r]\in \Irr(N)/H}\coh(X/H_r)\,,
$$
where $r\in\Irr(N)$ is a representative of an $H$-orbit $[r]\in \Irr(N)/H$ and $H_r$ is the stabilizer of $r$ in~$H$.
\end{lemma}
\begin{proof}
By Remark~\ref{rmk:S0}(iii) and Theorem~\ref{cor:trivobstr}$(ii)$, there is an equivalence of categories
$$
\coh^G(X)\simeq\coh^H\big(\uIrr(N)\times X\big)\,
$$
By condition~(i) of Definition~\ref{def:S0}, the scheme $\uIrr(N)$ is the finite set of points~$\Spec(k)$ parameterized by $\Irr(N)$. Since $H$ acts on~$X$ freely, there is an equivalence of categories
$$
\coh^H\big(\Irr(N)\times X\big)\simeq \coh\big((\Irr(N)\times X)/H\big)\,.
$$
Finally, one has an isomorphism
$$
(\Irr(N)\times X)/H\simeq \coprod_{[r]\in \Irr(N)/H} X/H_r\,,
$$
which finishes the proof.
\end{proof}

\begin{lemma}\label{lem:admdec}
Let $T$ be an idempotent complete triangulated category admitting a semiorthogonal decomposition
$$
T=\langle T_1, \ldots,  T_m\rangle
$$
into triangulated categories, which are indecomposable into
direct sums. Suppose that there is also a decomposition
$$
T\simeq C_1 \oplus \ldots \oplus C_m
$$
into a direct sum of non-zero triangulated categories. Then there exists a permutation $\sigma \in \Sigma_m$ such that $T_i \simeq C_{\sigma(i)}$ for any $i$, $1 \leqslant i \leqslant m$.
\end{lemma}
\begin{proof}
Since $T$ is idempotent complete, each $T_i$ is idempotent complete as well.
Hence the direct sum decomposition on $T$ induces a direct sum decomposition on each $T_i$.
Since the categories~$T_i$ are assumed to be indecomposable,
we have $T_i \subset C_{\sigma(i)}$ for some map $\sigma$ from the index set to itself.
It follows that we actually have an orthogonal decomposition $T = T_1 \oplus \ldots \oplus T_m$.
Also, since each $C_i$ is assumed to be non-trivial, the map $\sigma$ must be surjective and therefore a permutation.
From this, it is easy to see that $T_i \simeq C_{\sigma(i)}$ for each~$i$.
\end{proof}

The following statement is well-known and is obtained by Orlov~\cite[Ex.\,4.16]{Orlov1} in
the more general case of principally polarized abelian varieties over an algebraically closed field of zero characteristic. For the convenience of the reader, we provide a separate proof in our simple case, still using one of the main results from~\cite{Orlov1}.

\begin{lemma}\label{lem:Orl}
Let $E$ be an elliptic curve without complex multiplication over $k$. Suppose that for an elliptic curve $E'$, there is an equivalence of triangulated categories ${D^b(E)\simeq D^b(E')}$. Then there is an isomorphism $E\simeq E'$.
\end{lemma}
\begin{proof}
By~\cite[Theor.\,2.19]{Orlov1}, the equivalence $D^b(E)\simeq D^b(E')$ implies that there is an isomorphism of abelian surfaces
$$
E\times E\simeq E'\times E'\,.
$$
It follows that the curves $E$ and $E'$ are isogenous over $k$.

Let $\phi\colon E'\to E$ be any isogeny. We have a homomorphism of abelian groups
$$
\Hom_k(E,E')\longrightarrow \End_k(E)\,,\qquad \psi\longmapsto \phi\circ\psi\,.
$$
This homomorphism is injective, because the composition of two isogenies is a (non-zero) isogeny. Since $\End_k(E)\simeq\Z$, we see that $\Hom_k(E,E')$ is infinite cyclic.

Let an isogeny $\xi\colon E\to E'$ be a generator of $\Hom_k(E,E')$. Then the kernel of any morphism ${E\to E'}$ contains the finite group scheme $\Ker(\xi)\subset E$. This implies that the kernel of any morphism $E\to E'\times E'$ contains $\Ker(\xi)$ as well. On the other hand, the isomorphism ${E\times E\simeq E'\times E'}$ gives an injective morphism $E\to E'\times E'$. Therefore $\Ker(\xi)$ is trivial, whence the isogeny~$\xi$ is an isomorphism.
\end{proof}

\medskip

Now we are ready to prove Theorem~\ref{thm:PolvdB}.

\begin{proof}
By Example~\ref{exam:S0}(ii), the set of stabilizers of the action of~$H$ on~$C(N)$ consists of~$n$ copies of~$H$ and of~${n^3-n}$ subgroups $K\subset H$ such that each $H/K$ is a non-trivial cyclic group. Hence by Example~\ref{ex:extquottriv}, the collection of varieties $E^g/Z(g)$, $[g]\in C(G)$, coincides with
$$
%\underbrace{E,\ldots,E}_n\,,E_1\ldots,E_{n^2-1}\,,
\underbrace{E,\ldots,E}_n,\,E_1,\ldots,E_M,
$$
where $M\geqslant 1$ is a natural number depending on $n$ and for each $i$, $1\leqslant i\leqslant M$,
there is a cyclic isogeny $E\to E_i$ of positive degree.

Now assume that there is a semiorthogonal decomposition
$$
D^G(X)=\langle T_1,\ldots, T_{n+M}\rangle
$$
such that the collection $T_1,\ldots,T_{n+M}$ coincides up to permutation with the collection
$$
\underbrace{D^b(E),\ldots,D^b(E)}_n,\,D^b(E_1)\ldots,D^b(E_M)\,.
$$
Again by Example~\ref{exam:S0}(ii), the set of stabilizers of the action of $H$ on $\Irr(N)$ consists of subgroups of type ${\Z/d\Z\times \Z/d\Z}$, where $d$ is a positive divisor of $n$. For each subgroup $H'\subset H$ of this type, there is an isomorphism $E/H'\simeq E$. Hence, by Lemma~\ref{lem:explfree}, we have an equivalence of triangulated categories
$$
D^G(X)\simeq D(E)^{\oplus L},
$$
where $L$ is a natural number depending on $n$ (in fact, one can show $L=n+M$).
The categories $D^b(E)$ and $D^b(E_i)$, $1\leqslant i\leqslant M$, are indecomposable into direct sums as elliptic curves are connected. Thus, by Lemma~\ref{lem:admdec} with $m=n+M$, we obtain that for each $i$, $1\leqslant i\leqslant M$, there is an equivalence $D^b(E_i)\simeq D^b(E)$.

Now, since $E$ has no complex multiplication, by Lemma~\ref{lem:Orl}, there is an isomorphism~${E\simeq E_i}$.
However, this is impossible since an elliptic curve without complex multiplication over $k$ has no $k$-endomorphism with non-trivial cyclic kernel.
\end{proof}


\begin{thebibliography}{99}

\bibitem{AGV}
D.\,Abramovich, T.\,Graber, A.\,Vistoli, {\it Gromov--Witten theory of Deligne--Mumford stacks}, Amer. J. Math., {\bf 130}:5 (2008), 1337--1398.

\bibitem{AKMW}
D.\,Abramovich, K.\,Karu, K.\,Matsuki, J.\,W{\l}odarczyk, {\it Torification and factorization of birational maps},
J. Amer. Math. Soc., {\bf 15}:3 (2002), 531--572.

\bibitem{AH0}
M.\,Atiyah, F.\,Hirzebruch, {\it Vector bundles and homogeneous spaces}, Proc. Sympos. Pure Math., {\bf 3} (1961), 7--38, American Mathematical Society, Providence, R.I.

\bibitem{AH}
M.\,Atiyah, F.\,Hirzebruch, {\it The Riemann-Roch theorem for analytic embeddings}, Topology, {\bf 1} (1962), 151--166.

\bibitem{AS}
M.\,Atiyah, G.\,Segal, {\it On equivariant Euler characteristics}, J. Geom. Phys., {\bf 6}:4 (1989), 671--677.

\bibitem{Bar}
V.\,Baranovsky, {\it Orbifold cohomology as periodic cyclic homology}, Internat. J. Math., {\bf 14}:8 (2003), 791--812.

\bibitem{BG}
I.\,V.\,Beloshapka, S.\,O.\,Gorchinskiy, {\it Irreducible representations of finitely generated nilpotent groups}, Sbornik: Mathematics, {\bf 207}:1 (2016), 41--64.

\bibitem{Ber0}
D.\,Bergh, {\it The Binomial Theorem and motivic classes of universal quasi-split tori}, preprint (2014), arXiv:1409.5410.

\bibitem{Ber}
D.\,Bergh, {\it Functorial destackification of tame stacks with abelian stabilisers}, Compositio Mathematica, {\bf 153}:6 (2017), 1257--1315.

\bibitem{Ber1}
D.\,Bergh, {\it Weak factorization and the Grothendieck group of Deligne-Mumford stacks}, preprint (2017), arXiv:1707.06040v1.

\bibitem{BLS}
D.\,Bergh, V.\,Lunts, O.\,Schn\"urer, {\it Geometricity for derived categories of algebraic stacks}, Selecta Math. (N.S.), {\bf 22}:4 (2016), 2535--2568.

%\bibitem{BS}
%D.\,Bergh, O.\,Schn\"urer, {\it Conservative descent for semiorthogonal decompositions}, in preparation.

\bibitem{Bern}
M.\,Bernardara, {\it A semiorthogonal decomposition for Brauer--Severi schemes}, Math. Nachr., {\bf 282}:10 (2009), 1406--1413.

%\bibitem{BM}
%E.\,Bierstone, P.\,Milman, {\it Functoriality in resolution of singularities}, Publ. Res. Inst. Math. Sci., {\bf 44}:2 (2008), %609--639.

\bibitem{Bit}
F.\,Bittner, {\it The universal Euler characteristic for varieties of characteristic zero}, Compos. Math., {\bf 140}:4 (2004), 1011--1032.

\bibitem{Bla}
A.\,Blanc, {\it Topological $K$-theory of complex noncommutative spaces}, Compositio Mathematica, {\bf 152}:3 (2016), 489--555.

\bibitem{BK}
A.\,I.\,Bondal, M.\,M.\,Kapranov, {\it Representable functors, Serre functors, and reconstructions},  Math. USSR-Izv., {\bf 35}:3 (1990), 519--541.

\bibitem{BLL}
A.\,Bondal, M.\,Larsen, V.\,Lunts, {\it Grothendieck ring of pretriangulated categories}, Int. Math. Res. Not., {\bf 2004}:29 (2004), 1461--1495.

\bibitem{BvdB}
A.\,Bondal, M.\,Van den Bergh, {\it Generators and representability of functors in commutative and noncommutative geometry}, Mosc. Math. J., {\bf 3}:1 (2003), 1--36, 258.

\bibitem{BKR}
T.\,Bridgeland, A.\,King, M.\,Reid, {\it The McKay correspondence as an equivalence of derived categories}, J. Amer. Math. Soc., {\bf 14}:3 (2001), 535--554.

\bibitem{BF}
J.\,Bryan, J.\,Fulman, {\it Orbifold Euler characteristic and the number of commuting $m$-tuples in the symmetric groups}, Ann. Comb. {\bf 2}:1 (1998), 1--6.

%\bibitem{Bry}
%J.-L.\,Brylinski, {\it Central extensions and reciprocity laws}, Cahiers Topologie G\'eom. Diff\'erentielle Cat\'eg., {\bf 38}:3 (1997), 193--215.

\bibitem{Cad}
C.\,Cadman, {\it Using stacks to impose tangency conditions on curves}, Amer. J. Math., {\bf 129}:2 (2007), 405--427.

%\bibitem{Cal}
%A.\,H.\,C\v{a}ld\v{a}raru, {\it Derived categories of twisted sheaves on Calabi--Yau manifolds}, Ph.D. Thesis (2000).

\bibitem{CS}
A.\,Canonaco, P.\,Stellari, {Uniqueness of dg enhancements for the derived category of a Grothendieck category}, to appear in J. Eur. Math. Soc., arXiv:1507.05509.

%\bibitem{Cho}
%U.\,Choudhury, {\it Motives of Deligne-Mumford stacks}, Adv. Math., {\bf 231}:6 (2012), 3094--3117.

\bibitem{Con}
B.\,Conrad, {\it The Keel--Mori theorem via stacks}, preprint (2005).

\bibitem{dJ}
A.\,J.\,de Jong, {\it A result of Gabber}, preprint, http://www.math.columbia.edu/$\sim$dejong/papers/2-gabber.pdf.

\bibitem{Del0}
P.\,Deligne, {\it Th\'eor\`eme de Lefschetz et crit\`eres de d\'eg\'en\'erescence de suites spectrales}, Inst. Hautes \'Etudes Sci. Publ. Math., {\bf 35} (1968), 259--278.

%\bibitem{Del}
%P.\,Deligne, {\it Action du groupe des tresses sur une cat\'egorie}, Invent. Math. {\bf 128}:1 (1997), 159--175.

\bibitem{Eked}
T.\,Ekedahl, {\it The Grothendieck group of algebraic stacks}, preprint (2009), arXiv:0903.3143.

\bibitem{Ela0}
A.\,D.\,Elagin, {\it Semiorthogonal decompositions of derived categories of equivariant coherent sheaves}, Izvestiya: Mathematics, {\bf 73}:5 (2009), 893--920.

\bibitem{Ela}
A.\,D.\,Elagin, {\it Descent theory for semiorthogonal decompositions}, Sbornik: Mathematics, {\bf 203}:5 (2012), 645--676.

\bibitem{GKMS}
S.\,Galkin, L.\,Katzarkov, A.\,Mellit, E.\,Shinder, {\it Minifolds and phantoms}, Advances in Mathematics, {\bf 278} (2015), 238--253.

\bibitem{GS}
S.\,Galkin, E.\,Shinder, {\it On a zeta-function of a dg-category}, preprint (2015), arXiv:1506.05831.

\bibitem{GP}
E.\,Getzler, R.\,Pandharipande, {\it The Betti numbers of $\overline{\mathcal M}_{0,n}(r,d)$}, Journal of Algebraic Geometry, {\bf 15}:4 (2006), 709--732.

\bibitem{GiSa}
Ph.\,Gille, T.\,Szamuely, {\it Central simple algebras and Galois cohomology},
Cambridge Studies in Advanced Mathematics, {\bf 101}, Cambridge University Press, Cambridge (2006).

\bibitem{GO}
S.\,Gorchinskiy, D.\,Orlov, {\it Geometric phantom categories}, Publ. Math. Inst. Hautes \'Etudes Sci., {\bf 117} (2013), 329--349.

\bibitem{Gor}
S.\,Gorchinskiy, {\it Integral Chow motives of threefolds with $K$-motives of unit type}, Bull. of the Korean Math. Soc. (2017), DOI: 10.4134/BKMS.b160759.

\bibitem{SGA}
A.\,Grothendieck et al., {\it Rev\^etements \'etales et groupe fondamental (SGA 1)}, Lecture Notes in Math., {\bf 224}, Springer-Verlag, Berlin--New York (1971).

\bibitem{GZ}
S.\,M.\,Gusein-Zade, I.\,Luengo, A.\,Melle-Hern\'andez, {\it Grothendieck ring of varieties with finite groups actions}, preprint (2017), arXiv:1706.00918.

%\bibitem{FMN}
%B.\,Fantechi, E.\,Mann, F.\,Nironi, {\it Smooth toric Deligne--Mumford stacks}, J. Reine Angew. Math., {\bf 648} (2010), 201--244.

%\bibitem{clo2012}
%B.\,Conrad, M.\,Lieblich, M.\,Olsson, {\it Nagata compactification for algebraic spaces}, J. Inst. Math. Jussieu, {\bf 11}:4 (2012), 747--814.

%\bibitem{Har}
%R.\,Hartshorne, {\it Algebraic geometry}, Graduate Texts in Mathematics, {\bf 52}, Springer-Verlag, New York--Heidelberg %(1977).

\bibitem{Harper}
A.\,Harper, {\it Factorization for stacks and boundary complexes}, preprint (2017), arXiv:1706.07999.

\bibitem{HH}
F.\,Hirzebruch, T.\,H\"ofer, {\it On the Euler number of an orbifold}, Mathematische Annalen, {\bf 286}:1--3 (1990), 255--260.

\bibitem{IU}
A.\,Ishii, K.\,Ueda, {\it The special McKay correspondence and exceptional collections}, Tohoku Math. J. (2), {\bf 67}:4 (2015), 585--609.

\bibitem{Kap}
M.\,Kapranov, {\it The elliptic curve in the $S$-duality theory and Eisenstein series for Kac--Moody groups}, preprint (2000), arXiv:math.AG/0001005.

\bibitem{Kaw1}
Y.\,Kawamata, {\it Log crepant birational maps and derived categories},
J. Math. Sci. Univ. Tokyo, {\bf 12}:2 (2005), 211--231.

\bibitem{Kaw2}
Y.\,Kawamata, {\it Derived McKay correspondence for $GL(3,{\rm \bf C})$}, preprint (2016), arXiv:1609.09540.

\bibitem{KM}
S.\,Keel, S.\,Mori, {\it Quotients by groupoids}, Ann. of Math. (2), {\bf 145}:1 (1997), 193--213.

\bibitem{Ke}
B.\,Keller, {\it On differential graded categories}, International Congress of Mathematicians, {\bf II}, 151--190, Eur. Math. Soc., Z\"urich, (2006).

\bibitem{Knu}
D.\,Knutson, {\it Algebraic spaces}, Lecture Notes in Math., {\bf 203}, Springer-Verlag, Berlin--New York (1971).

\bibitem{Kol0}
J.\,Koll\'ar, {\it Quotient spaces modulo algebraic groups}, Ann. of Math. (2), {\bf 145}:1 (1997), 33--79.

%\bibitem{Kol}
%J.\,Koll\'ar, {\it Lectures on resolution of singularities}, Annals of Mathematics Studies, {\bf 166}, Princeton University Press, Princeton, NJ %(2007).

%\bibitem{Kre99}
%A.\,Kresch, {\it Cycle groups for Artin stacks}, Invent. Math., {\bf 138}:3 (1999), 495--536.

%\bibitem{KV}
%A.\,Kresch, A.\,Vistoli, {\it On coverings of Deligne--Mumford stacks and surjectivity of the Brauer map}, Bull. London Math. Soc., {\bf 36}:2 %(2004), 188--192.

\bibitem{Kre}
A.\,Kresch, {\it On the geometry of Deligne--Mumford stacks}, Algebraic geometry --- Seattle 2005. Part 1, Proc. Sympos. Pure Math.,
{\bf 80}, Amer. Math. Soc., Providence, RI (2009) 259--271.

%\bibitem{Ku}
%A.\,Kuznetsov, {\it Base change for semiorthogonal decompositions}, Compos. Math., {\bf 147}:3 (2011), 852--876.

\bibitem{KL}
A.\,Kuznetsov, V.\,Lunts, {\it  Categorical resolutions of irrational singularities}, Int. Math. Res. Not., {\bf 2015}:13 (2015), 4536--4625.

\bibitem{Loo}
E.\,Looijenga, {\it Motivic measure}, S\'eminaire Bourbaki, Vol. 1999/2000, Ast\'erisque, {\bf 276} (2002), 267--297.

\bibitem{LO}
V.\,Lunts, D.\,Orlov, {\it Uniqueness of enhancement for triangulated categories}, J. Amer. Math. Soc., {\bf 23}:3 (2010), 853--908.

\bibitem{LS}
V.\,Lunts, O.\,Schn\"urer, {\it Smoothness of equivariant derived categories}, Proc. Lond. Math. Soc. (3), {\bf 108}:5 (2014), 1226--1276.

\bibitem{LS3}
V.\,Lunts, O.\,Schn\"urer, {\it New enhancements of derived categories of coherent sheaves and applications}, Journal of Algebra, {\bf 446} (2016), 203--274.

\bibitem{LS2}
V.\,Lunts, O.\,Schn\"urer, {\it Matrix factorizations and motivic measures}, Journal of Noncommutative Geometry, {\bf 10}:3 (2016), 981--1042.

\bibitem{Ma}
Ju.\,I.\,Manin, {\it Correspondences, motifs and monoidal transformations}, Mathematics of the USSR-Sbornik, {\bf 6}:4 (1968), 439--470.

\bibitem{Milne}
J.\,S.\,Milne, {\it \'Etale cohomology}, Princeton Mathematical Series, {\bf 33}, Princeton University Press, Princeton, N.J., (1980).

\bibitem{O0}
D.\,O.\,Orlov, {\it Projective bundles, monoidal transformations, and derived categories of coherent sheaves}, Russian Acad. Sci. Izv. Math., {\bf 41}:1 (1993), 133--141.

\bibitem{Orlov1}
D.\,O.\,Orlov, {\it Derived categories of coherent sheaves on abelian varieties and equivalences between them}, Izv. Math., {\bf 66}:3 (2002), 569--594.

\bibitem{Orl}
D.\,Orlov, {\it Smooth and proper noncommutative schemes and gluing of DG categories}, Adv. Math., {\bf 302} (2016), 59--105.

\bibitem{PV}
A.\,Polishchuk, M.\,Van den Bergh, {\it Semiorthogonal decompositions of the categories of equivariant coherent sheaves for some reflection groups}, to appear in J. Eur. Math. Soc., arXiv:1503.04160.

\bibitem{Qui}
D.\,Quillen, {\it Algebraic $K$-theory I}, Lecture Notes in Mathematics, {\bf 341} (1973), 85--147.

\bibitem{Ried}
M.\,Reid, {\it The McKay correspondence}, S\'eminaire Bourbaki, Vol. 1999/2000, Ast\'erisque, {\bf 276}  (2002), 53--72.

\bibitem{Ryd}
D.\,Rydh, {\it Existence and properties of geometric quotients}, J. Algebraic Geom., {\bf 22}:4 (2013), 629--669.

%\bibitem{Serr}
%J.-P.\,Serre, {\it Corps locaux},  Publications de l'Universit\'e de Nancago, {\bf VIII}, Hermann, Paris (1968).

\bibitem{Snu}
O.\,Schn\"urer, {\it Six operations on dg enhancements of derived categories of sheaves}, preprint (2015), arXiv:1507.08697.

\bibitem{Switz}
R.\,Switzer, {\it Algebraic topology: homotopy and homology}, Classics in mathematics, Springer-Verlag (1975).

\bibitem{Tab}
G.\,Tabuada, {\it Invariants additifs de DG-cat\'egories}, Int. Math. Res. Not., {\bf 2005}:53 (2005), 3309--3339.

\bibitem{TVdB}
G.\,Tabuada, M.\,Van den Bergh, {\it Additive invariants of orbifolds}, preprint (2016), arXiv:1612.03162.

\bibitem{To0}
B.\,To\"en, {\it Th\'eor\`emes de Riemann--Roch pour les champs de Deligne--Mumford}, $K$-Theory, {\bf 18}:1 (1999), 33--76.

\bibitem{To1}
B.\,To\"en, {\it Notes sur la $G$-th\'eorie rationnelle des champs de Deligne--Mumford}, preprint (1999),  	 arXiv:math/9912172.

\bibitem{To}
B.\,To\"en, {\it The homotopy theory of dg-categories and derived Morita theory}, Invent. Math., {\bf 167} (2007), 615--667.

\bibitem{ToDG}
B.\,To\"en, {\it Lectures on dg-categories}, in Topics in algebraic and topological $K$-theory, Springer Berlin Heidelberg (2011), 243--302.

\bibitem{VV}
G.\,Vezzosi, A.\,Vistoli, {\it Higher algebraic $K$-theory of group actions with finite stabilizers}, Duke Math. J., {\bf 113}:1 (2002), 1--55.

\bibitem{Vis0}
A.\,Vistoli, {\it Higher equivariant $K$-theory for finite group actions}, Duke Math. J., {\bf 63}:2 (1991), 399--419.

\bibitem{Vis}
A.\,Vistoli, {\it Grothendieck topologies, fibered categories and descent theory}, Fundamental algebraic geometry, Math. Surveys Monogr., {\bf 123}, Amer. Math. Soc., Providence, RI (2005), 1--104.

\bibitem{Wlo}
J.\,W{\l}odarczyk, {\it Toroidal varieties and the weak factorization theorem}, Invent. Math., {\bf 154}:2 (2003), 223--331.


\end{thebibliography}
\end{document}